\documentclass[english]{amsart}
\usepackage[latin9]{inputenc}
\usepackage[a4paper]{geometry}
\usepackage{babel}
\usepackage{verbatim}
\usepackage{amstext}
\usepackage{amsthm}
\usepackage{amssymb}
\usepackage[unicode=true,
 bookmarks=false,
 breaklinks=false,pdfborder={0 0 1},backref=section,colorlinks=false]
 {hyperref}

\makeatletter
\numberwithin{equation}{section}
\numberwithin{figure}{section}
\theoremstyle{plain}
\newtheorem{thm}{\protect\theoremname}
\theoremstyle{plain}
\newtheorem{lem}[thm]{\protect\lemmaname}
\theoremstyle{plain}
\newtheorem{cor}[thm]{\protect\corollaryname}
\theoremstyle{definition}
\newtheorem{defn}[thm]{\protect\definitionname}
\theoremstyle{plain}
\newtheorem{prop}[thm]{\protect\propositionname}
\theoremstyle{definition}
\newtheorem{example}[thm]{\protect\examplename}
\theoremstyle{remark}
\newtheorem{rem}[thm]{\protect\remarkname}

\usepackage[foot]{amsaddr}

\usepackage{amstext}

\DeclareTextSymbolDefault{\textquotedbl}{T1}

\usepackage{comment}

\theoremstyle{plain}

\DeclareMathOperator{\OF}{\Omega_F}
\DeclareMathOperator{\dOF}{\partial\Omega_F}
\DeclareMathOperator{\OS}{\Omega_S}
\DeclareMathOperator{\dOS}{\partial\Omega_S}
\DeclareMathOperator{\dO}{\partial\Omega}
\DeclareMathOperator{\Lp}{L}

\DeclareMathOperator{\H1}{H}
\DeclareMathOperator{\C}{C}
\DeclareMathOperator{\Div}{div}

\DeclareMathOperator{\Span}{span}

\newcommand{\oxi}{\bar{\xi}}
\newcommand{\oq}{\bar{q}}
\newcommand*{\dx}{\mathop{}\!\mathrm{d}}

\providecommand{\corollaryname}{Corollary}
\providecommand{\definitionname}{Definition}
\providecommand{\lemmaname}{Lemma}
\providecommand{\remarkname}{Remark}
\providecommand{\theoremname}{Theorem}

\makeatother

\providecommand{\corollaryname}{Corollary}
\providecommand{\definitionname}{Definition}
\providecommand{\examplename}{Example}
\providecommand{\lemmaname}{Lemma}
\providecommand{\propositionname}{Proposition}
\providecommand{\remarkname}{Remark}
\providecommand{\theoremname}{Theorem}

\begin{document}
\title[Long-time fluid-structure interaction]{Global existence and convergence to pressure waves in nonlinear fluid-structure
interaction}
\author{Karoline Disser}
\author{Michelle Luckas}
\address{Universität Kassel\\
 Institut für Mathematik \\
 Heinrich-Plett-Stra{ß}e 40 \\
 34132 Kassel, Germany }
\email{kdisser@mathematik.uni-kassel.de}
\email{mluckas@mathematik.uni-kassel.de}
\begin{abstract}
We consider a non-linear system modelling the dynamics of a linearly
elastic body immersed in an incompressible viscous fluid, without
damping on the elastic part. We prove local existence of strong solutions
and global existence and uniqueness for small data. At the same time,
depending on the geometric setting, non-trivial time-periodic solutions,
called pressure waves, may persist. Our main result is the characterization
of long-time behaviour of the elastic displacement: up to small rigid
motions, either the system comes to rest or converges to a pressure
wave. 
\end{abstract}

\thanks{The authors would like to thank Dorothee Knees for helpful discussions
and input on the topic. The work of Michelle Luckas was financially
supported by ``Studienstiftung des deutschen Volkes''. }
\maketitle

\section{Introduction}

The dynamics of a linearly elastic body immersed in an incompressible
viscous fluid is modelled by the system

\begin{eqnarray}
\begin{cases}
\begin{array}{rcll}
\dot{u}+(u\cdot\nabla)u-\Div(\sigma(u,p)) & = & 0 & \textrm{in }(0,T)\times\OF(0,T),\\
\Div(u) & = & 0 & \textrm{in }(0,T)\times\OF(0,T),\\
(\sigma(u,p)\circ X)\mathrm{Cof}(\nabla X)n & = & \Sigma(\xi)n & \textrm{on }(0,T)\times\dOS,\\
u\circ X & = & \dot{\xi} & \textrm{on }(0,T)\times\dOS,\\
u & = & 0 & \text{on }(0,T)\times\dO,\\
\ddot{\xi}-\textrm{div}(\Sigma(\xi)) & = & 0 & \textrm{in }(0,T)\times\OS,\\
\dot{X} & = & u\circ X & \text{in }(0,T)\times\OF,
\end{array}\end{cases}\label{eq:fullsystem-1}
\end{eqnarray}
With initial conditions 
\begin{eqnarray*}
u(0)=u_{0}\quad\text{ and} & X(0)=\mathrm{Id}, & \text{in }\OF,\\
\xi(0)=\xi_{0}\quad\text{ and} & \dot{\xi}(0)=\xi_{1}, & \text{in }\OS,
\end{eqnarray*}
given on a bounded domain $\Omega\subset\mathbb{R}^{3}$. Here, in
Eulerian coordinates, the Navier-Stokes equations are modelled on
a time-dependent domain $\OF(t)=\Omega\setminus\overline{\OS(t)}$
that changes according to the elastic displacement. On the fluid-solid
interface $\dOS$, the flow map $X$ enables the formulation of the
standard transmission conditions of continuity of forces and of velocities
in Lagrangian coordinates. The fluid is Newtonian with viscosity $\nu>0$
and stress tensor 
\[
\sigma(u,p):=2\nu\varepsilon(u)-p\mathrm{Id}
\]
and the structure is linearly elastic with Lamé constants $\lambda_{1},\lambda_{2}>0$
and stress tensor 
\[
\Sigma(\xi):=2\lambda_{1}\varepsilon(\xi)+\lambda_{2}\Div(\xi)\mathrm{Id},
\]
where 
\[
\varepsilon(v):=\frac{1}{2}\left(\nabla v+(\nabla v)^{T}\right)
\]
is the symmetric gradient.

In this work, we simplify \eqref{eq:fullsystem-1} to the system
\begin{eqnarray}
\begin{cases}
\begin{array}{rcll}
\dot{u}+(u\cdot\nabla)u-\Div(\sigma(u,p)) & = & 0 & \textrm{in }(0,T)\times\OF,\\
\Div(u) & = & 0 & \textrm{in }(0,T)\times\OF,\\
\sigma(u,p)n & = & \Sigma(\xi)n & \textrm{on }(0,T)\times\dOS,\\
u & = & \dot{\xi} & \textrm{on }(0,T)\times\dOS,\\
u & = & 0 & \text{on }(0,T)\times\dO,\\
\ddot{\xi}-\textrm{div}(\Sigma(\xi)) & = & 0 & \textrm{in }(0,T)\times\OS,\\
u(0) & = & u_{0} & \textrm{in }\OF,\\
\xi(0) & = & \xi_{0} & \textrm{in }\OS,\\
\dot{\xi}(0) & = & \xi_{1} & \textrm{in }\OS,
\end{array}\end{cases}\label{eq:nonlin_system}
\end{eqnarray}
with a fixed fluid domain and the corresponding relation $\overline{\OS}\cup\OF=\Omega$.
From a modelling point of view, this corresponds to transferring the
principle of geometric linearization underlying the assumption of
linear elasticiy to the fluid domain. We denote the exterior unit
normal vector field of $\OS$ at $\dOS$ by $n$. The unkowns are
the fluid velocity $u\colon(0,T)\times\OF\to\mathbb{R}^{3}$, fluid
pressure $p\colon(0,T)\times\OF\to\mathbb{R}^{3}$ and elastic displacement
$\xi\colon(0,T)\times\OS\to\mathbb{R}^{3}$. 

Our motivation is to study long-term viscous damping and show that
due to fluid viscosity and the interaction of fluid and structure
through the boundary, either elastic displacement of the structure
will disappear over time, or the way in which it persists can be characterized.
More precisely, we show that unique strong solutions exist globally
in time, if the initial data are small, and that 
\begin{equation}
\lim_{t\to\infty}\Vert u(t)\Vert_{\H1^{1}(\OF)}=0,\label{ulong}
\end{equation}
as well as, 
\begin{equation}
\lim_{t\to\infty}\Vert\xi(t)-\eta^{*}(t)-\varphi_{N}^{0}-r(t)\Vert_{\H1^{1}(\OS)}=0.\label{xilong}
\end{equation}
Here, the constant (in time) function $\varphi_{N}^{0}$ is known
a priori and satisfies 
\[
\begin{cases}
\Div\Sigma(\varphi)=0, & \text{in }\OS,\\
\Sigma(\varphi)n=q_{0}n, & \text{on }\dOS,
\end{cases}
\]
with some constant $q_{0}\in\mathbb{R}$. In particular, if $\Sigma(\xi_{0})=0$,
then $\varphi_{N}^{0}=0$, so this is a correction for $\OS$ being
initially deformed. The term $r(t)\in\mathrm{ker}\,\Sigma$ is a rigid
velocity in the kernel of the linear elastic stress tensor,for which
we can show convergence in rate,
\[
\lim_{t\to\infty}\dot{r}(t)=\lim_{t\to\infty}\ddot{r}(t)=0.
\]
The function $\eta^{*}$ in (\ref{xilong}) is the most interesting
one: it corresponds to internal deformation of $\OS$ while $u=0$,
and thus $\dot{\xi}|_{\dOS}=0$. We call it a \emph{pressure wave}.
For many geometries of $\OS$, $\eta^{*}\neq0$ cannot occur. If $\OS$
is the ball, there are examples of non-trivial time-periodic $\eta^{*}$.
So the answer to the question of long-term viscous damping is a mixed
one, with a global geometric aspect. 

Pressure waves also provide solutions to system \eqref{eq:fullsystem-1},
so they must be accounted for in a global analysis of this system
as well. Details are given in Section \ref{SecXi}.

\subsection{Related works and comparison to known results}

\subsubsection{Results on the full system \eqref{eq:fullsystem-1} and related works}

We refer to \cite{Boulakia2007,CS2005,KT2012,IKLT2017,AB2015,RV2014,KO2023,KO2024,BKS2024}
for results on the existence and uniqueness of weak and strong solutions
to the full system \eqref{eq:fullsystem-1} in two and three spatial
dimensions. To our knowledge, global existence of solutions of \eqref{eq:fullsystem-1}
(even for small data), is not known. 

Global existence for \eqref{eq:fullsystem-1} was obtained in the
case of additional damping of the elastic structure: In \cite{IKLT2017,KO2023,KO2024},
global existence of solutions for small data and exponential convergence
to the rest state is shown with the Lamé equations replaced by a wave
equation with damping $\alpha\dot{\xi}$, for $\alpha>0$. Due to
the fact that pressure waves also solve \eqref{eq:fullsystem-1},
see Remark \ref{rem:pwsolve11}, it is clear that no analogous result
holds for the undamped system. In \cite{BKS2024}, global existence
of weak solutions is established for a system correpsonding to \eqref{eq:fullsystem-1}
in the nonlinear viscoelastic case, using variational methods. 

In \cite{BGT2019}, see also \cite{RV2014} and \cite{GHL2019} for
related methods, it was shown that the full system \eqref{eq:fullsystem-1}
without additional damping admits local strong solutions that preserve
initial regularity over time. This gives rise to the hope of finding
global strong solutions for small data. System \eqref{eq:nonlin_system}
is simpler to analyse, in particular due to the fact that the dynamics
and regularity of the interface $\partial\OS$ are not an issue. At
the same time, the structural difficulty of a parabolic-hyperbolic-type
coupling, some of the non-linearity of System \eqref{eq:fullsystem-1}
and the non-trivial long-time behaviour remains. In this sense, our
results are a first step towards analysing the undamped system \eqref{eq:fullsystem-1}
for large times. 

Important related models concern the interaction with elastic shells,
elastic beams, compressible or inviscid fluids. We refer to \cite{CS2006,LR2014,MMNRT2022,GH2016,GHL2019,MC2015,KT2024,AKT2025,BKS2024compressible,MRR2020}
and references therein. 

\subsubsection{Results on the linearized system}

The linearization of systems \eqref{eq:fullsystem-1} and \eqref{eq:nonlin_system}
was analyzed by Avalos and Triggiani \cite{4AT2009,5AT,8AT2009,AB2015,Astrongstab,AT2,ATboundary,ATuniformstab},
see also \cite{GMZZ2014}. In particular, they established the existence
of a (non-analytic) semigroup associated to this problem and characterized
its long-time behaviour based on properties of the spectrum of its
generator. For a large class of \emph{good domains} $\OS$, as well
as in the case of additional damping, they showed strong stability
of the semigroup \@ \cite{5AT,ATuniformstab,ATboundary,8AT2009,4AT2009,AT2,Astrongstab}.
This is consistent with \eqref{ulong} and \eqref{xilong}, as the
term $\varphi_{N}^{0}$ corresponds to an one-dimensional invariant
subspace of states, and in the semigroup approach, the rigid motions
$r(t)$ are taken care of by shifting the Lamé operator, removing
its kernel. Moreover, \emph{good domains} are exatly the ones that
guarantee $\eta^{*}=0$, cf. \cite{5AT,ATuniformstab,ATboundary,8AT2009,4AT2009,AT2,Astrongstab}
and the discussion in Section \ref{SecXi}. Our characterization of
long-time asymptotic behaviour of solutions in \eqref{xilong} extends
these stability results to a related statement in the general case
of \emph{bad domains}, when pressure waves $\eta^{*}\neq0$ may occur.
It determines the attractor and provides convergence to a specific
solution in this case. 

\subsubsection{Results related to System \eqref{eq:nonlin_system} }

We refer to \cite{KTZ2011} for the proof of local existence of weak
solutions to system (\ref{eq:nonlin_system}). In \cite{DL2022},
we proved global existence for small data in the two-dimensional case. 

\subsection{Further discussion }

The two main results of this paper are the existence of a unique global
solution in the case of small initial data, Theorem \ref{thm:globale_existenz},
and the characterization of the long-time behaviour of the displacement
in Theorem \ref{thm:main result xi}. A first step in the analysis
is to establish a functional analytic setting that can handle the
mixed parabolic-hyperbolic character of the problem. In the existence
proof, we use the key techniques that have been established for these
mixed systems: optimal regularity of the mixed Dirichlet and Neumann-Stokes
system \cite{BP2007,GS1991}, hidden regularity for the Dirichlet-Lamé
system \cite{LLT1986,KTZ2011}, and additional approximation arguments
that shift between regularity levels, following the approach in \@
\cite{BGT2019} for System \eqref{eq:fullsystem-1}. Here, the setting
of \cite{BGT2019} is adapted to weaker norms, so that global energy
estimates and first-order energy estimates are sufficient for globally
extending the solution. On the other hand, the norms are stronger
than in the semigroup approach of Avalos and Triggiani, so that the
nonlinearity can be handled. 

Probably due to the geometric aspect of the long-time behaviour, we
do not have a direct proof of \eqref{xilong} using dynamical systems
theory. Instead, the method seems to be ad-hoc, essentially using
the same tools but in a technical way: the eigenmoode decomposition
of candidate limit pressure waves, energy estimates, continuous dependence
of the data, and compactness arguments. We hope that it can be generalized
to apply to other similar systems, e.g. with a mixed hyperbolic-parabolic
character and of transmission type, \cite{GMZZ2014}. 

In system (\ref{eq:nonlin_system}), domains do not change, so the
difficult situation of contact of structure and outer wall does not
arise explicitly. We refer to \cite{CGLMTW2019,MAA2018,CGH2021} for
the analysis of contact problems in this context. 

A full characterization of good and bad domains is, to our knowledge,
not available yet. The most complete results are in \cite{5AT,AT2,ATboundary},
cf. Section \ref{SecXi}. For corresponding considerations on the
influence of geometry on the controllability of the system, we refer
to \cite{OP1999}. 

The paper is organized as follows: In Section \ref{SecLocal}, local
existence of strong solutions is established. In Section \ref{SecGlobal},
it is shown that solutions extend to be global-in-time, if the initial
data are sufficiently small. In Section \ref{SecXi}, the main result
on long-time behaviour of $\xi$ is introduced. Sections \ref{SecTilde}
and \ref{SecProof} contain the proof.

\section{Local existence of solutions}

\label{SecLocal}

\label{locEx} The existence of strong solutions $(u,p,\xi)$ to (\ref{eq:nonlin_system})
up to time $T>0$ is shown in the spaces 
\begin{align*}
u\in X_{T}:= & \Lp^{2}(\H1^{2}(\OF))\cap\H1^{1}(\H1^{1}(\OF))\cap\C^{1}(\Lp^{2}(\OF)),\\
p\in Y_{T}:= & \Lp^{2}(\H1^{1}(\OF)),\\
\xi\in Z_{T}:= & \C^{0}(\H1^{2}(\OS))\cap\C^{1}(\H1^{1}(\OS))\cap\C^{2}(\Lp^{2}(\OS)),
\end{align*}
where the time interval $(0,T)$ for the Sobolev spaces and $[0,T]$
for the spaces of uniformly continuous functions is omitted whenever
possible. We also show that these solutions satisfy
\begin{align}
u & \in\C^{0}(\H1^{2}(\OF)),\label{addReg}\\
p & \in\C^{0}(\H1^{1}(\OF)).\nonumber 
\end{align}
In this regularity class, the initial data 
\begin{equation}
u_{0}\in\H1^{2}(\OF),\,\xi_{0}\in\H1^{2}(\OS),\,\xi_{1}\in\H1^{1}(\OS)\label{idata}
\end{equation}
must satisfy the following compatibility conditions: There are
\begin{align}
u_{1},\,p_{0}\in\H1^{1}(\OF)\text{ and }\xi_{2}\in\Lp^{2}(\OS)\label{idata_2}
\end{align}
such that the equations 
\begin{equation}
\begin{cases}
\begin{array}{rcll}
u_{1}-\Div(\sigma(u_{0},p_{0})) & = & -(u_{0}\cdot\nabla)u_{0} & \text{in }\OF,\\
\Div(u_{0}) & = & 0 & \text{in }\OF,\\
\Div(u_{1}) & = & 0 & \text{in }\OF,\\
\sigma(u_{0},p_{0})n & = & \Sigma(\xi_{0})n & \textrm{on }\ensuremath{\partial\Omega_{S}},\\
u_{0} & = & \xi_{1} & \text{on }\dOS,\\
u_{0} & = & 0 & \text{on }\dO,\\
u_{1} & = & 0 & \text{on }\dO,\\
\xi_{2}-\Div(\Sigma(\xi_{0})) & = & 0 & \text{in }\OS,
\end{array}\end{cases}\label{eq:compatibility_thm}
\end{equation}
hold. These conditions are also sufficient. The local existence result
is the following. 
\begin{thm}
\label{thm:lokale_existenz} Let the initial data $u_{0},\xi_{0}$
and $\xi_{1}$ be given such that (\ref{idata}) -- (\ref{eq:compatibility_thm})
are satisfied. Then there exists a time 
\[
T=T\left(\Vert u_{0}\Vert_{\H1^{1}(\OF)},\Vert u_{1}\Vert_{\Lp^{2}(\OF)},\Vert\varepsilon(\xi_{0})\Vert_{\Lp^{2}(\OS)},\Vert\xi_{1}\Vert_{\H1^{1}(\OS)},\Vert\xi_{2}\Vert_{\Lp^{2}(\OS)}\right)>0
\]
such that system (\ref{eq:nonlin_system}) admits a unique strong
solution 
\[
u\in\C^{0}(\H1^{2}(\OF))\cap X_{T},p\in\C^{0}(\H1^{1}(\OF))\text{ and }\xi\in Z_{T}.
\]
\end{thm}

We refer to \cite{DL2022} for the proof of a similar result in the
two-dimensional setting. Substantial differences particularly appear
in the non-linear estimates. The proof can be divided into four steps.
\smallskip{}

\textbf{Step 1: More regular solutions to the linearized system} \\
We consider the linearized system 
\begin{equation}
\begin{cases}
\begin{array}{rcll}
\dot{u}-\Div(\sigma(u,p)) & = & f & \textrm{in }(0,T)\times\OF,\\
\Div(u) & = & 0 & \textrm{in }(0,T)\times\OF,\\
\sigma(u,p)n & = & \Sigma(\xi)n & \textrm{on }(0,T)\times\dOS,\\
u & = & \dot{\xi} & \textrm{on }(0,T)\times\dOS,\\
u & = & 0 & \text{on }(0,T)\times\dO,\\
\ddot{\xi}-\Div(\Sigma(\xi)) & = & 0 & \textrm{in }(0,T)\times\OS,\\
u(0) & = & u_{0} & \textrm{in }\OF,\\
\xi(0) & = & \xi_{0} & \text{in }\OS,\\
\dot{\xi}(0) & = & \xi_{1} & \textrm{in }\OS.
\end{array}\end{cases}\label{eq:lin_system}
\end{equation}
and define the following auxiliary spaces of higher regularity: 
\begin{align*}
\tilde{X}_{T}:= & \Lp^{2}(\H1^{5/2+1/16}(\OF))\cap\H1^{1}(\H1^{2}(\OF))\cap\H1^{2}(\H1^{1}(\OF)),\\
\tilde{Y}_{T}:= & \Lp^{2}(\H1^{3/2+1/16}(\OF))\cap\H1^{1}(\H1^{1}(\OF)),\\
\tilde{Z}_{T}:= & \Lp^{2}(\H1^{5/2+1/16}(\OS))\cap\C^{1}(\H1^{3/2+1/16}(\OS))\cap\C^{2}(\H1^{1/2+1/16}(\OS))\cap\C^{3}(\H1^{-1/2+1/16}(\OS)).
\end{align*}
We refer to the Appendix for a precise definition and properties of
these Sobolev-Slobodecki and Bochner spaces. The following existence
result holds. 
\begin{thm}
\label{lem:lin_regulaer} Let 
\begin{align}
\begin{split}(u_{0},\,u_{1},\,p_{0},\,\xi_{0},\,\xi_{1},\,\xi_{2},\,f)\in & \,\H1^{5/2+1/16}(\OF)\times\H1^{1}(\OF)\times\H1^{3/2+1/16}(\OF)\\
 & \times\H1^{5/2+1/16}(\OS)\times\H1^{3/2+1/16}(\OS)\times\H1^{1/2+1/16}(\OS)\\
 & \times\left(\Lp^{2}(\H1^{1/2+1/16}(\OF))\cap\H1^{1}(\Lp^{2}(\OF))\right)
\end{split}
\end{align}
be such that the compatibility conditions in (\ref{eq:compatibility_thm})
are satisfied with $(u_{0}\cdot\nabla)u_{0}$ replaced by $-f(0)$.
Assume additionally that 
\begin{align}
\begin{array}{rcll}
u_{1} & = & \xi_{2} & \textrm{on }\dOS.\end{array}\label{eq:compatibility_step1}
\end{align}
Then for every $T>0$, the linear system (\ref{eq:lin_system}) admits
a unique solution $(u,p,\xi)\in\tilde{X}_{T}\times\tilde{Y}_{T}\times\tilde{Z}_{T}$. 
\end{thm}

The proof of Theorem \ref{lem:lin_regulaer} is based on the ideas
in the proof of a corresponding result for the linearization of system
\eqref{eq:fullsystem-1} in \cite{BGT2019} and \cite{RV2014}. In
particular, this includes tricks on how to deal with the parabolic-hyperbolic
coupling. Compared to \cite[Proposition 1.5]{BGT2019}, for our simpler
system, we additionally prove the existence of solutions on arbitrary
time intervals and for non-trivial initial displacements $\xi_{0}\neq0$.
\\
For estimating the transmission-type boundary conditions it is natural
to define the space for every $\theta\geq0$, 
\[
\mathrm{M}^{\theta}(\dOS):=\H1^{\theta}(0,T;\Lp^{2}(\dOS))\cap\Lp^{2}(0,T;\H1^{\theta}(\dOS)).
\]
We repeatedly use the following trace estimate.
\begin{lem}
\label{lem:auxiliaryLinear} Let $v\in\tilde{X}_{T}$. Then there
exist $C,\alpha>0$ such that 
\[
\Vert v\Vert_{\mathrm{M}^{3/2+1/16}(\dOS)}+\left\Vert \int_{0}^{t}v\,\mathrm{d}s\right\Vert _{\mathrm{M}^{2+1/16}(\dOS)}\leq CT^{\alpha}\left(\Vert v(0)\Vert_{\H1^{2+1/16}(\OF)}+\Vert\dot{v}(0)\Vert_{\H1^{3/4}(\OF)}+\Vert v\Vert_{\tilde{X}_{T}}\right).
\]
\end{lem}

\begin{proof}
By the usual trace estimates for $\dOS$ as part of the boundary of
$\OF$, using \eqref{eq:est_appendix_BGT} and Lemma \ref{lem:est_interpolation_BGT}a)
with $s=3/4$, $\sigma_{1}=2$ and $\sigma_{2}=5/2+1/16$, we obtain
\[
\Vert v\Vert_{\Lp^{2}(\H1^{3/2+1/16}(\Gamma))}\leq CT^{\alpha}\left(\Vert v\Vert_{\tilde{X}_{T}}+\Vert v(0)\Vert_{\H1^{2+1/16}(\OF)}\right).
\]
Analogously, using \eqref{eq:est_appendix_BGT} and Lemma \ref{lem:est_interpolation_BGT}a)
with $s=5/8$, $\sigma_{1}=2$ and $\sigma_{2}=5/2+1/16$, we obtain 

\[
\Vert v\Vert_{\H1^{3/2+1/16}(\Lp^{2}(\dOS))}\leq CT^{\alpha}\left(\Vert v\Vert_{\tilde{X}_{T}}+\Vert v(0)\Vert_{\H1^{3/4}(\OF)}+\Vert\dot{v}(0)\Vert_{\H1^{3/4}(\OF)}\right).
\]
Regarding the second term, note that 
\[
\left\Vert \int_{0}^{t}v\,\mathrm{d}s\right\Vert _{\Lp^{2}(\H1^{2+1/16}(\dOS))}\leq CT\Vert v\Vert_{\Lp^{2}(\H1^{5/2+1/16}(\dOS))},
\]
and that using \eqref{eq:est_appendix_BGT}, we obtain
\begin{align*}
\left\Vert \int_{0}^{t}v\,\mathrm{d}s\right\Vert _{\H1^{2+1/16}(\Lp^{2}(\dOS))} & \leq C\left((T+1)\Vert v\Vert_{\Lp^{2}(\H1^{3/4}(\OF))}+\Vert\dot{v}\Vert_{\H1^{1/16}(\H1^{3/4}(\OF))}\right)\\
 & \leq C\Big((T^{2}+T)\Vert v-v(0)\Vert_{\Lp^{2}(\H1^{3/4}(\OF))}+(T^{3/2}+T^{1/2})\Vert v(0)\Vert_{\H1^{3/4}(\OF)}\\
 & \qquad T^{1/2+1/16}\Vert\dot{v}-\dot{v}(0)\Vert_{\H1^{5/8}(\H1^{3/4}(\OF))}+T^{1/2}\Vert\dot{v}(0)\Vert_{\H1^{3/4}(\OF)}\Big).
\end{align*}
Now we can conclude by using Lemma \ref{lem:est_interpolation_BGT}a)
with $s=5/8$, $\sigma_{1}=0$ and $\sigma_{2}=2$.
\end{proof}
In order to prove Theorem \ref{lem:lin_regulaer}, let $\hat{u}\in\tilde{X}_{T}^{0}:=\left\{ v\in\tilde{X}_{T}:v(0)=u_{0}\text{ and }\dot{v}(0)=u_{1}\right\} $
be given and consider the linear auxiliary systems 
\begin{equation}
\begin{cases}
\begin{array}{rcll}
\ddot{\xi}-\Div(\Sigma(\xi)) & = & 0 & \text{in }(0,T)\times\OS,\\
\xi & = & \xi_{0}+\int_{0}^{t}\hat{u}(s)\,\mathrm{d}s & \text{on }(0,T)\times\Gamma,\\
\xi(0) & = & \xi_{0} & \text{in }\OS,\\
\dot{\xi}(0) & = & \xi_{1} & \text{in }\OS,
\end{array}\end{cases}\label{eq:xihelp}
\end{equation}
and 

\begin{equation}
\begin{cases}
\begin{array}{rcll}
\dot{u}-\Div(\sigma(u,p)) & = & f & \text{in }(0,T)\times\OF,\\
\Div(u) & = & 0 & \text{in }(0,T)\times\OF,\\
\sigma(u,p)n & = & \Sigma(\xi)n & \text{on }(0,T)\times\Gamma,\\
u & = & 0 & \text{on }(0,T)\times\dO,\\
u(0) & = & u_{0} & \text{in }\OF.
\end{array}\end{cases}\label{eq:uhelp}
\end{equation}
as well as their derivatives\\
\begin{equation}
\begin{cases}
\begin{array}{rcll}
\ddot{\Xi}-\Div(\Sigma(\Xi)) & = & 0 & \text{in }(0,T)\times\OS,\\
\Xi & = & \hat{u} & \text{on }(0,T)\times\Gamma,\\
\Xi(0) & = & \xi_{1} & \text{in }\OS,\\
\dot{\Xi}(0) & = & \xi_{2} & \text{in }\OS,
\end{array}\end{cases}\label{eq:Xihelp}
\end{equation}
and

\begin{equation}
\begin{cases}
\begin{array}{rcll}
\dot{U}-\Div(\sigma(U,P)) & = & \dot{f} & \text{in }(0,T)\times\OF,\\
\Div(U) & = & 0 & \text{in }(0,T)\times\OF,\\
\sigma(U,P)n & = & \Sigma(\Xi)n & \text{on }(0,T)\times\Gamma,\\
U & = & 0 & \text{on }(0,T)\times\dO,\\
U(0) & = & u_{1} & \text{in }\OF.
\end{array}\end{cases}\label{eq:Uhelp}
\end{equation}
By Theorem \ref{thm:lame_LLT} and Lemma \ref{lem:auxiliaryLinear},
there is a unique solution
\[
\Xi\in\hat{Z}_{T}:=\C^{0}(\H1^{3/2+1/16}(\OS))\cap\C^{1}(\H1^{1/2+1/16}(\OS))\cap\C^{2}(\H1^{-1/2+1/16}(\OS))
\]
of \eqref{eq:Xihelp} such that 
\begin{align}
 & \Vert\Xi\Vert_{\hat{Z}_{T}}+\Vert\Sigma(\Xi)n\Vert_{\mathrm{M}^{1/2+1/16}(\dOS)}\nonumber \\
 & \leq C\left(\Vert\xi_{1}\Vert_{\H1^{3/2+1/16}(\OS)}+\Vert\xi_{1}\Vert_{\H1^{1/2+1/16}(\OS)}+T^{\alpha}\left(\Vert u_{0}\Vert_{\H1^{1/2+1/16}(\OF)}+\Vert u_{1}\Vert_{\H1^{3/4}(\OF)}+\Vert\hat{u}\Vert_{\tilde{X}_{T}}\right)\right).\label{eq:Xiest}
\end{align}
By construction,
\[
\xi:=\xi_{0}+\int_{0}^{t}\Xi(s)\,\mathrm{d}s\in\tilde{Z}_{T}
\]
is then a unique solution of \eqref{eq:xihelp}, where we have used
Theorem \ref{thm:elliptic_lame} and Lemma \ref{lem:auxiliaryLinear}
to prove the additional regularity
\begin{align*}
\Vert\xi\Vert_{\Lp^{2}(\H1^{5/2+1/16}(\OS))} & \leq C\left(\Vert\ddot{\xi}\Vert_{\Lp^{2}(\H1^{1/2+1/16}(\OS))}+\left\Vert \xi_{0}+\int_{0}^{t}\hat{u}(s)\,\mathrm{d}s\right\Vert _{\Lp^{2}(\H1^{2+1/16}(\Gamma))}\right)\\
 & \leq CT^{\alpha}\Big(\Vert\xi_{0}\Vert_{\H1^{5/2+1/16}(\OS)}+\Vert\xi_{1}\Vert_{\H1^{3/2+1/16}(\OS)}+\Vert\xi_{2}\Vert_{\H1^{1/2+1/16}(\OS)}\\
 & \qquad+\Vert u_{0}\Vert_{\H1^{2+1/16}(\OF)}+\Vert u_{1}\Vert_{\H1^{3/4}(\OF)}+\Vert\hat{u}\Vert_{\tilde{X}_{T}}\Big).
\end{align*}
Theorem \ref{thm:Stokes_BP} provides a unique solution $(U,P)$ of
system \eqref{eq:Uhelp}. Additonally using \eqref{eq:Xiest}, we
obtain the estimate
\begin{align}
 & \Vert U\Vert_{\Lp^{2}(\H1^{2}(\OF))\cap\H1^{1}(\Lp^{2}(\OF))}+\Vert P\Vert_{\Lp^{2}(\H1^{1}(\OF))}\nonumber \\
 & \leq C\Big(\Vert u_{1}\Vert_{\H1^{1}(\OF)}+\Vert\dot{f}\Vert_{\Lp^{2}(\Lp^{2}(\OF))}+\Vert\xi_{1}\Vert_{\H1^{3/2+1/16}(\OS)}+\Vert\xi_{2}\Vert_{\H1^{1/2+1/16}(\OS)}\label{eq:UPest}\\
 & \qquad+T^{\alpha}\left(\Vert u_{0}\Vert_{\H1^{2+1/16}(\OF)}+\Vert u_{1}\Vert_{\H1^{3/4}(\OF)}+\Vert\hat{u}\Vert_{\tilde{X}_{T}}\right)\Big).\nonumber 
\end{align}
Again by construction, the time integrals 
\[
u:=u_{0}+\int_{0}^{t}U(s)\,\mathrm{d}s,\qquad p:=p_{0}+\int_{0}^{t}P(s)\,\mathrm{d}s
\]
solve \eqref{eq:uhelp}. By applying Theorem \ref{thm:Stokes_BP}
to \eqref{eq:uhelp} with $s=1/2+1/16$, applying Theorem \ref{thm:lame_LLT}
to \eqref{eq:xihelp} with $\theta=1+1/16$, and using Lemma \ref{lem:auxiliaryLinear}
and \eqref{eq:UPest}, we obtain 
\begin{align}
 & \Vert u\Vert_{\Lp^{2}(\H1^{5/2+1/16}(\OF))}+\Vert p\Vert_{\Lp^{2}(\H1^{3/2+1/16}(\OF))}\nonumber \\
 & \leq C\Big((1+T^{1/2})\Vert\xi_{0}\Vert_{\H1^{5/2+1/16}(\OS)}+\Vert\xi_{1}\Vert_{\H1^{3/2+1/16}(\OS)}+\Vert\xi_{2}\Vert_{\H1^{1/2+1/16}(\OS)}\label{eq:ucontraction}\\
 & \qquad+\Vert u_{1}\Vert_{\H1^{1}(\OF)}+\Vert f\Vert_{\Lp^{2}(\H1^{1/2+1/16}(\OF))\cap\H1^{1}(\Lp^{2}(\OF))}\nonumber \\
 & \qquad+T^{\alpha}\left(\Vert u_{0}\Vert_{\H1^{2+1/16}(\OF)}+\Vert u_{1}\Vert_{\H1^{3/4}(\OF)}+\Vert\hat{u}\Vert_{\tilde{X}_{T}}\right)\Big).\nonumber 
\end{align}
Consequently, for given $\hat{u}\in\tilde{X}_{T}^{0}$, there are
unique solutions $\xi\in\tilde{Z}_{T}$ and $(u,p)\in\tilde{X}_{T}^{0}\times\tilde{Y}_{T}$
to systems $\eqref{eq:xihelp}$ and \eqref{eq:uhelp}. The affine
map $\tilde{X}_{T}^{0}\ni\hat{u}\mapsto u\in\tilde{X}_{T}^{0}$ is
well-defined and continuous and \eqref{eq:ucontraction} shows that
the operator norm satisfies $\leq CT^{\alpha}$ for $\alpha\geq0$,
so that we obtain a unique fixed point $u\in\tilde{X}_{T_{0}}^{0}$
and corresponding $p\in\tilde{Y}_{T_{0}},\xi\in\tilde{Z}_{T_{0}}$
if $T_{0}$ is chosen sufficiently small. As the size of $T_{0}$
does not depend on the initial data and the compatibility conditions
\eqref{eq:compatibility_thm} are preserved, we can extend these solutions
to any $T>0$ and have proved Theorem \ref{lem:lin_regulaer}. 

\textbf{Step 2: A-priori estimates and approximation of corresponding
solutions.}\\
 The next step in the proof of Theorem \ref{thm:lokale_existenz}
is to reduce the regularity in Theorem~\ref{lem:lin_regulaer} to
norms that fit to the global a-priori estimates associated to system
\eqref{eq:nonlin_system}. To establich global estimates, we define
the total of kinetic and elastic energies for system (\ref{eq:nonlin_system})
by 
\[
E(t):=\frac{1}{2}\Vert u(t)\Vert_{\Lp^{2}(\OF)}^{2}+\frac{1}{2}\Vert\dot{\xi}(t)\Vert_{\Lp^{2}(\OS)}^{2}+\frac{1}{2}\int_{\OS}\Sigma(\xi):\varepsilon(\xi)(t)\,\dx y,
\]
and a corresponding higher-order quantity by 
\[
K(t):=\frac{1}{2}\Vert\dot{u}(t)\Vert_{\Lp^{2}(\OF)}^{2}+\frac{1}{2}\Vert\ddot{\xi}(t)\Vert_{\Lp^{2}(\OS)}^{2}+\frac{1}{2}\int_{\OS}\Sigma(\dot{\xi}):\varepsilon(\dot{\xi})(t)\,\dx y.
\]

\begin{thm}
\label{lem:lin} Let 
\begin{align}
\begin{split}\label{eq:AW}\left(u_{0},\,u_{1},\,p_{0},\,\xi_{0},\,\xi_{1},\,\xi_{2},\,f\right)\in & \,\H1^{2}(\OF)\times\H1^{1}(\OF)\times\H1^{1}(\OF)\times\H1^{2}(\OS)\times\H1^{1}(\OS)\times\Lp^{2}(\OS)\\
 & \times\left(\Lp^{2}(\H1^{1/2+1/16}(\OF))\cap\H1^{1}(\H1^{-1/2+1/16}(\OF))\right)
\end{split}
\end{align}
be given such that (\ref{eq:compatibility_thm}) is satisfied, with
$-(u_{0}\cdot\nabla)u_{0}$ replaced by $f(0)$. Then the linear system
(\ref{eq:lin_system}) admits a unique solution $(u,p,\xi)\in X_{T}\times Y_{T}\times Z_{T}$
that satisfies 
\begin{align}
E(t)+\int_{0}^{t}2\nu\Vert\varepsilon(u(s))\Vert_{\Lp^{2}(\OF)}^{2}\,\dx s & =E(0)+\int_{0}^{t}\int_{\OF}f\cdot u\,\dx y\dx s\label{EE}
\end{align}
and 
\begin{align}
K(t)+\int_{0}^{t}2\nu\Vert\varepsilon(\dot{u}(s))\Vert_{\Lp^{2}(\OF)}^{2}\,\dx s & =K(0)+\int_{0}^{t}\langle\dot{u}(s),\dot{f}(s)\rangle_{1/2-1/16}\,\dx s.\label{eq:Kgleichung}
\end{align}
\end{thm}

\begin{proof}
Given existence of a solution $(u,p,\xi)\in X_{T}\times Y_{T}\times Z_{T}$,
the energy equality (\ref{EE}) is obtained by testing with $u$ and
$\dot{\xi}$, respectively. The corresponding equality for $K$ in
\eqref{eq:Kgleichung} is obtained by testing the time derivative
of the system with $\dot{u},\ddot{\xi}$ , respectively. Note that
using Korn's second inequality and the relation
\[
\int_{\OS}\Sigma(\xi(t)):\varepsilon(\xi(t))\,\dx y=2\lambda_{1}\Vert\varepsilon(\xi(t))\Vert_{\Lp^{2}(\OS)}^{2}+\int_{\OS}\lambda_{2}\Div(\xi(t))^{2}\,\dx y,
\]
we have constants $c,C>0$ depending only on $\OS$ such that 

\[
c\Vert\varepsilon(\xi(t))\Vert_{\Lp^{2}(\OS)}^{2}\leq\int_{\OS}\Sigma(\xi(t)):\varepsilon(\xi(t))\,\dx y\leq C\Vert\varepsilon(\xi(t))\Vert_{\Lp^{2}(\OS)}^{2},
\]
and 
\[
c\Vert\dot{\xi}(t)\Vert_{\H1^{1}(\OS)}^{2}\leq\Vert\dot{\xi}(t)\Vert_{\Lp^{2}(\OS)}^{2}+\int_{\OS}\Sigma(\dot{\xi}(t)):\varepsilon(\dot{\xi}(t))\,\dx y\leq C\Vert\dot{\xi}(t)\Vert_{\H1^{1}(\OS)}^{2}.
\]
Combining (\ref{EE}) and \eqref{eq:Kgleichung} gives the a-priori
estimate
\begin{align}
\begin{split} & \Vert u\Vert_{\H1^{1}(\H1^{1}(\OF))\cap\C^{1}(\Lp^{2}(\OF))}^{2}+\Vert\varepsilon(\xi)\Vert_{\C^{0}(\Lp^{2}(\OS))}^{2}+\Vert\dot{\xi}\Vert_{\C^{0}(\H1^{1}(\OS))\cap\C^{1}(\Lp^{2}(\OF))}^{2}\\
\leq & \,C\Big(\Vert u_{0}\Vert_{\Lp^{2}(\OF)}^{2}+\Vert u_{1}\Vert_{\Lp^{2}(\OF)}^{2}+\Vert\varepsilon(\xi_{0})\Vert_{\Lp^{2}(\OS)}^{2}+\Vert\xi_{1}\Vert_{\H1^{1}(\OS)}^{2}+\Vert\xi_{2}\Vert_{\Lp^{2}(\OS)}^{2}\\
 & \hspace{0.5cm}+\Vert f\Vert_{\Lp^{2}(\Lp^{2}(\OF))\cap\H1^{1}(\H1^{-1/2+1/16}(\OF))}^{2}\Big).
\end{split}
\label{eq:est_H1H1_step3}
\end{align}
The existence part of Theorem \ref{lem:lin} can be proved by using
this estimate and approximation with more regular data. We define
\begin{align*}
A:=\H1^{2}(\OF)\times\H1^{1}(\OF)\times\H1^{1}(\OF)\times\H1^{2}(\OS)\times\H1^{1}(\OS)\times\Lp^{2}(\OS)
\end{align*}
and 
\begin{align*}
\tilde{A}:= & \H1^{5/2+1/16}(\OF)\times\H1^{1}(\OF)\times\H1^{3/2+1/16}(\OF)\\
 & \times\H1^{5/2+1/16}(\OS)\times\H1^{3/2+1/16}(\OS)\times\H1^{1/2+1/16}(\OS).
\end{align*}
Given are
\[
d:=(u_{0},\,u_{1},\,p_{0},\,\xi_{0},\,\xi_{1},\,\xi_{2},\,f)\in A\times\left(\Lp^{2}(\H1^{1/2+1/16}(\OF))\cap\H1^{1}(\H1^{-1/2+1/16}(\OF))\right)
\]
that satisfy (\ref{eq:compatibility_thm}) with $-(u_{0}\cdot\nabla)u_{0}$
replaced by $f(0)$.Through the following six small steps, we construct
a sequence 
\begin{align*}
d_{n}:=(u_{0}^{n},u_{1}^{n},p_{0}^{n},\xi_{0}^{n},\xi_{1}^{n},\xi_{2}^{n},f^{n})\in\tilde{A}\times\left(\C^{\infty}(\H1^{1/2+1/16}(\OF))\right)
\end{align*}
that satisfies (\ref{eq:compatibility_thm}) and (\ref{eq:compatibility_step1})
for all $n\in\mathbb{N}$ such that 
\begin{equation}
d_{n}\to d\text{ in the norms of }A\times\left(\Lp^{2}(\H1^{1/2+1/16}(\OF))\cap\H1^{1}(\H1^{-1/2+1/16}(\OF))\right):\label{eq:approxinitialData}
\end{equation}
 1.) Set $u_{1}^{n}:=u_{1}$ for all $n\in\mathbb{N}$. \\
 2.) Choose a sequence $(\hat{\xi}_{2}^{n})\subset\C_{0}^{\infty}(\OS)$
such that $\lim_{n\to\infty}\hat{\xi}_{2}^{n}=\xi_{2}$ in $\Lp^{2}(\OS)$.
To modify this sequence such that it satisfies the compatibility condition
on $\dOS$, we define 
\begin{align*}
(\dOS)^{n}:=\left\{ y\in\OS:\text{dist}(y,\dOS)<\frac{1}{2^{n}}\right\} 
\end{align*}
for $n\in\mathbb{N}$ and choose a sequence $(\varphi^{n})\subset\C^{\infty}(\OS)$
such that 
\begin{align*}
\varphi^{n}(y)=\begin{cases}
1 & \text{if }y\in(\dOS)^{n+1},\\
0 & \text{if }y\in\Omega_{S}\setminus(\dOS)^{n}.
\end{cases}
\end{align*}
Now let $u_{1}^{E}\in\H1^{1}(\Omega)$ denote an extension of $u_{1}$
to $\Omega$ and set 
\begin{align*}
\xi_{2}^{n}:=\hat{\xi}_{2}^{n}+\varphi^{n}u_{1}^{E}\in\H1^{1}(\OS).
\end{align*}
Then $\xi_{2}^{n}\vert_{\dOS}=u_{1}\vert_{\dOS}$ and 
\begin{align*}
\Vert\varphi^{n}u_{1}^{E}\Vert_{\Lp^{2}(\OS)}\leq C\Vert\varphi^{n}\Vert_{\Lp^{3}(\OS)}\Vert u_{1}^{E}\Vert_{\Lp^{6}(\OS)}\leq C\vert(\dOS)^{n}\vert^{1/3}\Vert u_{1}^{E}\Vert_{\H1^{1}(\OS)}\to0,
\end{align*}
so $\xi_{2}^{n}\to\xi_{2}$ in $\Lp^{2}(\OS)$. \\
 3.) Choose a sequence $(g^{n})\subset\H1^{2+1/16}(\dOS)$ such that
$\lim_{n\to\infty}g^{n}=\xi_{0}\vert_{\dOS}$ in $\H1^{3/2}(\dOS)$.
Because of $\Div(\Sigma(\xi_{0}))=\xi_{2},$ a sequence $(\xi_{0}^{n})\subset\H1^{5/2+1/16}(\OS)$
that satisfies $\lim_{n\to\infty}\xi_{0}^{n}=\xi_{0}$ in $\H1^{2}(\OS)$
is then given by the solutions of the Dirichlet problem 
\[
\begin{cases}
\begin{array}{rcll}
\Div(\Sigma(\xi_{0}^{n})) & = & \xi_{2}^{n} & \text{in }\OS,\\
\xi_{0}^{n} & = & g^{n} & \text{on }\dOS,
\end{array}\end{cases}
\]
and using Theorem~\ref{thm:elliptic_lame} for both $s=3/2+1/16$
and $s=1$. \\
 4.) Since $\H1^{1/2+1/16}(\OF)\hookrightarrow\H1^{-1/2+1/16}(\OF)$
is dense, \cite[Theorem 2.1]{LionsMagenes} implies that we find a
sequence $(f^{n})\subset\C^{\infty}(\H1^{1/2+1/16}(\OF))$ such that
$\lim_{n\to\infty}f^{n}=f$ in $\Lp^{2}(\H1^{1/2+1/16}(\OF))\cap\H1^{1}(\H1^{-1/2+1/16}(\OF))$.
Moreover, the embedding 
\[
\left(\H1^{1/2+1/16}(\OF),\H1^{-1/2+1/16}(\OF)\right)_{1/2}\hookrightarrow\Lp^{2}(\OF)
\]
and \cite[Theorem 3.1]{LionsMagenes} imply that 
\begin{align*}
\Vert f-f^{n}\Vert_{\C^{0}(\Lp^{2}(\OF))}\leq C\Vert f-f^{n}\Vert_{\Lp^{2}(\H1^{1/2+1/16}(\OF))\cap\H1^{1}(\H1^{-1/2+1/16}(\OF)}\to0,
\end{align*}
and, in particular, $f^{n}(0)\to f(0)$ in $\Lp^{2}(\OF)$. \\
 5.) For all $n\in\mathbb{N}$, consider the Stokes problem 
\[
\begin{cases}
\begin{array}{rcll}
\Div(\sigma(u_{0}^{n},p_{0}^{n})) & = & u_{1}^{n}-f^{n}(0) & \text{in }\OF,\\
\Div(u_{0}^{n}) & = & 0 & \text{in }\OF,\\
\sigma(u_{0}^{n},p_{0}^{n})n & = & \Sigma(\xi_{0}^{n})n & \text{on }\dOS,\\
u_{0}^{n} & = & 0 & \text{on }\dO.
\end{array}\end{cases}
\]
Because of $u_{1}^{n}\in\H1^{1}(\OF)$ and $\Sigma(\xi_{0}^{n})n\in\H1^{1+1/16}(\partial\Omega_{S})$,
there is a sequence of solutions $(u_{0}^{n},p_{0}^{n})\subset\H1^{5/2+1/16}(\OF)\times\H1^{3/2+1/16}(\OF)$
by using Theorem~\ref{thm:elliptic_stokes} for $s=1/2+1/16$. Due
to
\begin{align*}
 & \lim_{n\to\infty}u_{1}^{n}=u_{1}\text{ in }\Lp^{2}(\OF),\hspace{0.2cm}\lim_{n\to\infty}\Sigma(\xi_{0}^{n})n=\Sigma(\xi_{0})n\text{ in }\H1^{1/2}(\partial\Omega_{S})\\
 & \text{and }\lim_{n\to\infty}f^{n}(0)=f(0)\text{ in }\Lp^{2}(\OF),
\end{align*}
Theorem~\ref{thm:elliptic_stokes} for $s=0$ implies $\lim_{n\to\infty}u_{0}^{n}=u_{0}$
in $\H1^{2}(\OF)$ and $\lim_{n\to\infty}p_{0}^{n}=p_{0}$ in $\H1^{1}(\OF)$.
\\
 6.) Finally, set $h:=\Div(\Sigma(\xi_{1}))\in\H1^{-1}(\OS))$ and
consider the elliptic problem 
\[
\begin{cases}
\begin{array}{rcll}
\Div(\Sigma(\xi_{1})) & = & h & \text{in }\H1^{-1}(\OS),\\
\xi_{1} & = & u_{0} & \text{on }\dOS.
\end{array}\end{cases}
\]
Now choose some sequence $(h^{n})\subset\Lp^{2}(\OS)$ such that $\lim_{n\to\infty}h^{n}=h$
in $\H1^{-1}(\OS)$ and consider the elliptic problems 
\[
\begin{cases}
\begin{array}{rcll}
\Div(\Sigma(\xi_{1}^{n})) & = & h^{n} & \text{in }\OS,\\
\xi_{1}^{n} & = & u_{0}^{n} & \text{on }\dOS.
\end{array}\end{cases}
\]
Since Step 5 implies that $(u_{0}^{n}\vert_{\dOS})\subset\H1^{2+1/16}(\dOS)$
and $\lim_{n\to\infty}u_{0}^{n}\vert_{\dOS}=u_{0}\vert_{\dOS}$ in
$\H1^{3/2}(\dOS)$, we can use Theorem~\ref{thm:elliptic_lame} for
both $s=1$ and $s=0$ and obtain a sequence of solutions $(\xi_{1}^{n})\subset\H1^{2}(\OS)$
such that $\lim_{n\to\infty}\xi_{1}^{n}=\xi_{1}$ in $\H1^{1}(\OS)$. 

Through Steps 1.) - 6.), we have constructed compatible data 
\[
d_{n}=(u_{0}^{n},\,u_{1}^{n},\,p_{0}^{n},\,\xi_{0}^{n},\,\xi_{1}^{n},\,\xi_{2}^{n},\,f^{n})\in\tilde{A}\times\left(\C^{\infty}(\H1^{1/2+1/16}(\OF))\right)
\]
that satisfy \eqref{eq:approxinitialData}. By Theorem \ref{lem:lin_regulaer},
we find solutions $(u^{n},p^{n},\xi^{n})\in\tilde{X}_{T}\times\tilde{Y}_{T}\times\tilde{Z}_{T}$
to the linear system (\ref{eq:lin_system}) corresponding to $d_{n}$.
Estimate (\ref{eq:est_H1H1_step3}) for the difference $(u^{n}-u^{m},p^{n}-p^{m},\xi^{n}-\xi^{m})$
of any such solutions shows that $(u^{n},\varepsilon(\xi^{n}),\dot{\xi}^{n})$
is a Cauchy sequence in 
\begin{align*}
\left(\H1^{1}(\H1^{1}(\OF))\cap\C^{1}(\Lp^{2}(\OF))\right)\times\C^{0}(\Lp^{2}(\OS))\times\left(\C^{0}(\H1^{1}(\OS))\cap\C^{1}(\Lp^{2}(\OS))\right).
\end{align*}
Moreover, by elliptic regularity of the Stokes problem, Theorem~\ref{thm:elliptic_stokes},
and Theorem~\ref{thm:elliptic_lame} with $s=1$, 
\begin{align*}
 & \Vert u^{n}-u^{m}\Vert_{\Lp^{2}(\H1^{2}(\OF))}+\Vert p^{n}-p^{m}\Vert_{\Lp^{2}(\H1^{1}(\OF))}+\Vert\xi^{n}-\xi^{m}\Vert_{\C^{0}(\H1^{2}(\OS))}\\
\leq & \,C\Big(T^{1/2}\Vert\dot{u}^{n}-\dot{u}^{m}\Vert_{\C^{0}(\Lp^{2}(\OF))}+\Vert f^{n}-f^{m}\Vert_{\Lp^{2}(\Lp^{2}(\OF))}+(1+T^{1/2})\Vert\ddot{\xi}_{n}-\ddot{\xi}_{m}\Vert_{\C^{0}(\Lp^{2}(\OS))}\\
 & \hspace{0.3cm}+(1+T^{1/2})\Vert\xi_{0}^{n}-\xi_{0}^{m}\Vert_{\H1^{2}(\OS)}+(T^{1/2}+T)\Vert u^{n}-u^{m}\Vert_{\Lp^{2}(\H1^{2}(\OF))}\Big).
\end{align*}
Consequently, for $T>0$ such that $C(T^{1/2}+T)<1$, the last term
on the right-hand side can be absorbed, so that $(u^{n},p^{n},\xi^{n})$
is also a Cauchy sequence in 
\begin{align*}
\Lp^{2}(\H1^{2}(\OF))\times\Lp^{2}(\H1^{1}(\OF))\times\C^{0}(\H1^{2}(\OS))
\end{align*}
and therefore in $X_{T}\times Y_{T}\times Z_{T}$. By construction,
the limit $(u,p,\xi)\in X_{T}\times Y_{T}\times Z_{T}$ is a strong
solution to the linear system (\ref{eq:lin_system}) with initial
data $(u_{0},u_{1},p_{0},\xi_{0},\xi_{1},\xi_{2},f)$. Uniqueness
follows from a Gronwall argument. This concludes the proof of Theorem
\ref{lem:lin}.

\textbf{Step 3: Non-linear problem.} \\
 The existence and uniqueness of solutions to (\ref{eq:nonlin_system})
is proved by using a fixed point argument. The starting point are
estimates on the non-linear term $(u\cdot\nabla)u$. They are given
in some detail, as the choice of norms is special. 
\end{proof}
\begin{lem}
\label{lem:utilde_statt_f-1} Let $u,v,w\in X_{T}$. 
\begin{itemize}
\item[a)] Then 
\begin{align*}
(u\cdot\nabla)u\in\Lp^{2}(\H1^{1/2+1/16}(\OF))\cap\H1^{1}(\H1^{-1/2+1/16}(\OF)).
\end{align*}
\item[b)] There exist some $C,\,\alpha>0$ such that 
\begin{align*}
 & \Vert(u\cdot\nabla)v\Vert_{\Lp^{2}(\Lp^{2}(\OF))}\\
\leq & \,CT^{\alpha}\min\Big\{\left(\Vert u\Vert_{\H1^{1}(\H1^{1}(\OF))}+\Vert u(0)\Vert_{\H1^{1}(\OF)}\right)\left(\Vert v\Vert_{\H1^{1}(\H1^{1}(\OF))}+\Vert v(0)\Vert_{\H1^{1}(\OF)}\right)^{1/3}\Vert v\Vert_{X_{T}}^{2/3},\\
\\
 & \Vert u\Vert_{X_{T}}^{2/3}\left(\Vert u\Vert_{\H1^{1}(\H1^{1}(\OF))}+\Vert u(0)\Vert_{\H1^{1}(\OF)}\right)^{1/3}\left(\Vert v\Vert_{\H1^{1}(\H1^{1}(\OF))}+\Vert v(0)\Vert_{\H1^{1}(\OF)}\right)\Big\}.\\
\end{align*}
\item[c)]  There exist some $C,\,\alpha>0$ such that 
\begin{align*}
 & \int_{0}^{t}\int_{\OF}\vert(u\cdot\nabla)v\cdot w\vert+\vert(\dot{u}\cdot\nabla)v\cdot\dot{w}\vert+\vert(u\cdot\nabla)\dot{v}\cdot\dot{w}\vert\,\dx y\dx s\\
\leq & \,CT^{\alpha}\Big(\Vert u\Vert_{\H1^{1}(\H1^{1}(\OF))\cap\C^{1}(\Lp^{2}(\OF))}\Vert v\Vert_{\H1^{1}(\H1^{1}(\OF))}+\Vert u(0)\Vert_{\H1^{1}(\OF)}\Vert v\Vert_{\H1^{1}(\H1^{1}(\OF))}\\
 & \hspace{1.5cm}+\Vert v(0)\Vert_{\H1^{1}(\OF)}\Vert u\Vert_{\H1^{1}(\H1^{1}(\OF))\cap\C^{1}(\Lp^{2}(\OF))}\Big)\Vert w\Vert_{\H1^{1}(\H1^{1}(\OF))\cap\C^{1}(\Lp^{2}(\OF))}.
\end{align*}
\end{itemize}
\end{lem}

\begin{proof}
a) To show $(u\cdot\nabla)u\in\Lp^{2}(\H1^{1/2+1/16}(\OF))$, we use
interpolation, Hölder's inequality and (fractional) Sobolev embeddings
to get
\begin{align*}
 & \Vert(u\cdot\nabla)u\Vert_{\H1^{1/2+1/16}(\OF)}\\
\leq\, & C\Vert(u\cdot\nabla)u\Vert_{\H1^{1}(\OF)}^{1/2+1/16}\Vert(u\cdot\nabla)u\Vert_{\Lp^{2}(\OF)}^{1/2-1/16}\\
\leq\, & C\big(\Vert(u\cdot\nabla)u\Vert_{\Lp^{2}(\OF)}+\Vert\vert\nabla u\vert\vert\nabla u\vert\Vert_{\Lp^{2}(\OF)}^{1/2+1/16}\Vert(u\cdot\nabla)u\Vert_{\Lp^{2}(\OF)}^{1/2-1/16}\\
 & \hspace{0.6cm}+\Vert\vert u\vert\vert\nabla^{2}u\vert\Vert_{\Lp^{2}(\OF)}^{1/2+1/16}\Vert(u\cdot\nabla)u\Vert_{\Lp^{2}(\OF)}^{1/2-1/16}\big)\\
\leq & C\big(\Vert u\Vert_{\H1^{1}(\OF)}\Vert u\Vert_{\H1^{3/2}(\OF)}+\Vert\nabla u\Vert_{\Lp^{3}(\OF)}^{1/2+1/16}\Vert\nabla u\Vert_{\Lp^{6}(\OF)}^{1/2+1/16}\Vert u\Vert_{\Lp^{6}(\OF)}^{1/2-1/16}\Vert\nabla u\Vert_{\Lp^{3}(\OF)}^{1/2-1/16}\\
 & \hspace{0.6cm}+\Vert u\Vert_{\Lp^{\infty}(\OF)}^{1/2+1/16}\Vert\nabla^{2}u\Vert_{\Lp^{2}(\OF)}^{1/2+1/16}\Vert u\Vert_{\Lp^{6}(\OF)}^{1/2-1/16}\Vert\nabla u\Vert_{\Lp^{3}(\OF)}^{1/2-1/16}\big)\\
\leq & C\Vert u\Vert_{\H1^{13/8}(\OF)}\Vert u\Vert_{\H1^{2}(\OF)}^{1/2+1/16}\Vert u\Vert_{\H1^{1}(\OF)}^{1/2-1/16}.
\end{align*}
Hölder's inequality on $(0,T)$ and more embeddings combined with
Lemma \ref{lem:est_interpolation_BGT} provide
\begin{align*}
 & \Vert(u\cdot\nabla)u\Vert_{\Lp^{2}(\H1^{1/2+1/16}(\OF))}\\
\leq\, & C\left\Vert \Vert u\Vert_{\H1^{13/8}(\OF)}\right\Vert _{\Lp^{8}(0,T)}\left\Vert \Vert u\Vert_{\H1^{2}(\OF)}\right\Vert _{\Lp^{2}(0,T)}^{9/16}\left\Vert \Vert u\Vert_{\H1^{1}(\OF)}\right\Vert _{\Lp^{14/3}(0,T)}^{7/16}\\
\leq\, & C(T)\Vert u\Vert_{\H1^{3/8}(\H1^{13/8}(\OF))}\Vert u\Vert_{\Lp^{2}(\H1^{2}(\OF))}^{9/16}\Vert u\Vert_{\H1^{4/14}(\H1^{1}(\OF))}^{7/16}\leq C(T)\Vert u\Vert_{X_{T}}^{2}.
\end{align*}
In order to show $(u\cdot\nabla)u\in\H1^{1}(\H1^{-1/2+1/16}(\OF))$,
note that
\begin{align*}
\H1^{-1/2+1/16}(\OF)=(\H1^{1/2-1/16}(\OF))^{*}
\end{align*}
(see \cite[Theorem 4.8.2]{Triebel}). Now for $v\in\H1^{1/2-1/16}(\OF)$,
Hölder's inequality and embeddings give the estimate
\begin{align*}
 & \int_{\OF}\left((\dot{u}\cdot\nabla)u+(u\cdot\nabla)\dot{u}\right)\cdot v\,\dx y\\
\leq\, & C\left(\Vert\dot{u}\Vert_{\Lp^{4}(\OF)}\Vert\nabla u\Vert_{\Lp^{8/3}(\OF)}+\Vert u\Vert_{\Lp^{8}(\OF)}\Vert\nabla\dot{u}\Vert_{\Lp^{2}(\OF)}\right)\Vert v\Vert_{\Lp^{8/3}(\OF)}\\
\leq\, & C\left(\Vert\dot{u}\Vert_{\H1^{3/4}(\OF)}\Vert u\Vert_{\H1^{11/8}(\OF)}+\Vert u\Vert_{\H1^{9/8}(\OF)}\Vert\dot{u}\Vert_{\H1^{1}(\OF)}\right)\Vert v\Vert_{\H1^{1/2-1/16}(\OF)}.
\end{align*}
Again, Hölder's inequality on $(0,T)$ and further embeddings combined
with Lemma \ref{lem:est_interpolation_BGT} lead to
\begin{align*}
 & \left\Vert \int_{\OF}\left((\dot{u}\cdot\nabla)u+(u\cdot\nabla)\dot{u}\right)\cdot v\,\dx y\right\Vert _{\Lp^{2}(0,T)}\\
 & \leq\,C\Big(\left\Vert \Vert\dot{u}\Vert_{\H1^{3/4}(\OF)}\right\Vert _{\Lp^{4}(0,T)}\left\Vert \Vert u\Vert_{\H1^{11/8}(\OF)}\right\Vert _{\Lp^{4}(0,T)}\\
 & \hspace{0.3cm}+\left\Vert \Vert u\Vert_{\H1^{9/8}(\OF)}\right\Vert _{\Lp^{\infty}(0,T)}\left\Vert \Vert\dot{u}\Vert_{\H1^{1}(\OF)}\right\Vert _{\Lp^{2}(0,T)}\Big)\Vert v\Vert_{\H1^{1/2-1/16}(\OF)}\\
 & \leq\,C(T)\Vert u\Vert_{X_{T}}^{2}\Vert v\Vert_{\H1^{1/2-1/16}(\OF)}.
\end{align*}
Similarly, we can estimate 
\begin{align*}
\left\Vert \int_{\OF}(u\cdot\nabla)u\cdot v\,\dx y\right\Vert _{\Lp^{2}(0,T)}\leq C(T)\Vert u\Vert_{X_{T}}^{2}\Vert v\Vert_{\H1^{1/2-1/16}(\OF)},
\end{align*}
so we conclude that $(u\cdot\nabla)u\in\H1^{1}(\H1^{-1/2+1/16}(\OF))$.
\\
 b) By embedding and interpolation, 
\[
\Vert(u\cdot\nabla)v\Vert_{\Lp^{2}(\OF)}\leq C\Vert u\Vert_{\H1^{5/3}(\OF)}\Vert v\Vert_{\H1^{1}(\OF)}\leq C\Vert u\Vert_{\H1^{2}(\OF)}^{2/3}\Vert u\Vert_{\H1^{1}(\OF)}^{1/3}\Vert v\Vert_{\H1^{1}(\OF)},
\]
so Hölder's inequality on $(0,T)$ together with (\ref{eq:sobolev_inequality_T-1})
for $q=12$, $s=5/12$, $\sigma=1$ implies
\begin{align*}
 & \Vert(u\cdot\nabla)v\Vert_{\Lp^{2}(\Lp^{2}(\OF))}\\
\leq\, & C\left\Vert \Vert u\Vert_{\H1^{2}(\OF)}^{2/3}\right\Vert _{\Lp^{3}(0,T)}\left\Vert \Vert u\Vert_{\H1^{1}(\OF)}^{1/3}\Vert\right\Vert _{\Lp^{12}(0,T)}\left\Vert \Vert v\Vert_{\H1^{1}(\OF)}\right\Vert _{\Lp^{12}(0,T)}\\
\leq\, & C\left\Vert \Vert u\Vert_{\H1^{2}(\OF)}\right\Vert _{\Lp^{2}(0,T)}^{2/3}\left\Vert \Vert u\Vert_{\H1^{1}(\OF)}\Vert\right\Vert _{\Lp^{4}(0,T)}^{1/3}\left\Vert \Vert v\Vert_{\H1^{1}(\OF)}\right\Vert _{\Lp^{12}(0,T)}\\
\leq\, & CT^{\alpha}\Vert u\Vert_{X_{T}}^{2/3}\left(\Vert u\Vert_{\H1^{1}(\H1^{1}((\OF))}+\Vert u(0)\Vert_{\H1^{1}(\OF)}\Vert\right)^{1/3}\left(\Vert v\Vert_{\H1^{1}(\H1^{1}(\OF))}+\Vert v(0)\Vert_{\H1^{1}(\OF)}\right).\\
\end{align*}
Similarly, due to

\[
\Vert(u\cdot\nabla)v\Vert_{\Lp^{2}(\OF)}\leq C\Vert u\Vert_{\Lp^{6}(\OF)}\Vert\nabla v\Vert_{\Lp^{3}(\OF)}\leq C\Vert u\Vert_{\H1^{1}(\OF)}\Vert v\Vert_{\H1^{1}(\OF)}^{1/3}\Vert v\Vert_{\H1^{2}(\OF)}^{2/3},
\]
we obtain
\begin{align*}
 & \Vert(u\cdot\nabla)v\Vert_{\Lp^{2}(\Lp^{2}(\OF))}\\
\leq\, & CT^{\alpha}\left(\Vert u\Vert_{\H1^{1}(\H1^{1}((\OF))}+\Vert u(0)\Vert_{\H1^{1}(\OF)}\right)\left(\Vert v\Vert_{\H1^{1}(\H1^{1}(\OF))}+\Vert v(0)\Vert_{\H1^{1}(\OF)}\right)^{1/3}\Vert v\Vert_{X_{T}}^{2/3}.\\
\end{align*}
 c) For the second term, by Hölder's inequality, embeddings and interpolation,
\begin{align*}
\int_{\OF}\vert(\dot{u}\cdot\nabla)v\cdot\dot{w}\vert\,\dx y & \leq C\Vert\dot{u}\Vert_{\Lp^{3}(\OF)}\Vert\nabla v\Vert_{\Lp^{2}(\OF)}\Vert\dot{w}\Vert_{\H1^{1}(\OF)}\\
 & \leq C\Vert\dot{u}\Vert_{\H1^{1}(\OF)}^{1/2}\Vert\dot{u}\Vert_{\Lp^{2}(\OF)}^{1/2}\Vert v\Vert_{\H1^{1}(\OF)}\Vert\dot{w}\Vert_{\H1^{1}(\OF)}.
\end{align*}
Applying again Hölder's inequality on $(0,T)$, embeddings, and (\ref{eq:sobolev_inequality_T-1})
for $q=8$, $s=3/8$, $\sigma=1$, we obtain
\begin{align*}
 & \int_{0}^{t}\int_{\OF}\vert(\dot{u}\cdot\nabla)v\cdot\dot{w}\vert\,\dx y\dx s\\
\leq\, & C\left\Vert \Vert\dot{u}\Vert_{\H1^{1}(\OF)}^{1/2}\right\Vert _{\Lp^{4}(0,T)}\left\Vert \Vert\dot{u}\Vert_{\Lp^{2}(\OF)}^{1/2}\right\Vert _{\Lp^{8}(0,T)}\left\Vert \Vert v\Vert_{\H1^{1}(\OF)}\right\Vert _{\Lp^{8}(0,T)}\left\Vert \Vert\dot{w}\Vert_{\H1^{1}(\OF)}\right\Vert _{\Lp^{2}(0,T)}\\
\leq\, & C\left\Vert \Vert\dot{u}\Vert_{\H1^{1}(\OF)}\right\Vert _{\Lp^{2}(0,T)}^{1/2}\left\Vert \Vert\dot{u}\Vert_{\Lp^{2}(\OF)}\right\Vert _{\Lp^{4}(0,T)}^{1/2}\left\Vert \Vert v\Vert_{\H1^{1}(\OF)}\right\Vert _{\Lp^{8}(0,T)}\left\Vert w\right\Vert _{\H1^{1}(\H1^{1}(\OF))}\\
\leq\, & CT^{\alpha}\Vert u\Vert_{\H1^{1}(\H1^{1}(\OF))}^{1/2}\Vert u\Vert_{\C^{1}(\Lp^{2}(\OF))}^{1/2}\left(\Vert v\Vert_{\H1^{1}(\H1^{1}(\OF))}+\Vert v(0)\Vert_{\H1^{1}(\OF)}\right)\left\Vert w\right\Vert _{\H1^{1}(\H1^{1}(\OF))}.
\end{align*}
The first term can be estimated similarly. For the third term, we
make use of the same tools to estimate 
\begin{align*}
 & \int_{0}^{t}\int_{\OF}\vert(u\cdot\nabla)\dot{v}\cdot\dot{w}\vert\,\dx y\dx s\\
\leq\, & C\int_{0}^{t}\Vert u\Vert_{\H1^{1}(\OF)}\Vert\nabla\dot{v}\Vert_{\Lp^{2}(\OF)}\Vert\dot{w}\Vert_{\H1^{1}(\OF)}^{1/2}\Vert\dot{w}\Vert_{\Lp^{2}(\OF)}^{1/2}\,\dx s\\
\leq\, & C\left\Vert \Vert u\Vert_{\H1^{1}(\OF)}\right\Vert _{\Lp^{8}(0,T)}\left\Vert \Vert\nabla\dot{v}\Vert_{\Lp^{2}(\OF)}\right\Vert _{\Lp^{2}(0,T)}\left\Vert \Vert\dot{w}\Vert_{\H1^{1}(\OF)}^{1/2}\right\Vert _{\Lp^{4}(0,T)}\left\Vert \Vert\dot{w}\Vert_{\Lp^{2}(\OF)}^{1/2}\right\Vert _{\Lp^{8}(0,T)}\\
\leq\, & CT^{\alpha}\left(\Vert u\Vert_{\H1^{1}(\H1^{1}(\OF))}+\Vert u(0)\Vert_{\H1^{1}(\OF)}\right)\Vert v\Vert_{\H1^{1}(\H1^{1}((\OF))}\Vert w\Vert_{\H1^{1}(\H1^{1}(\OF))}^{1/2}\Vert w\Vert_{\C^{1}(\Lp^{2}(\OF))}^{1/2}.\qedhere
\end{align*}
Now consider given data 
\begin{align*}
(u_{0},u_{1},p_{0},\xi_{0},\xi_{1},\xi_{2})\in\H1^{2}(\OF)\times\H1^{1}(\OF)\times\H1^{1}(\OF)\times\H1^{2}(\OS)\times\H1^{1}(\OS)\times\Lp^{2}(\OS)
\end{align*}
such that the compatibility conditions (\ref{eq:compatibility_thm})
are satisfied. For some 
\begin{align*}
M=M\left(\Vert u_{0}\Vert_{\H1^{1}(\OF)},\Vert u_{1}\Vert_{\Lp^{2}(\OF)},\Vert\varepsilon(\xi_{0})\Vert_{\Lp^{2}(\OS)},\Vert\xi_{1}\Vert_{\H1^{1}(\OS)},\Vert\xi_{2}\Vert_{\Lp^{2}(\OS)}\right)>0,
\end{align*}
we set 
\begin{align*}
X_{T}^{0,M}:=\left\{ v\in X_{T}:v(0)=u_{0},\,\dot{v}(0)=u_{1},\,\Vert v\Vert_{\H1^{1}(\H1^{1}(\OF))\cap\C^{1}(\Lp^{2}(\OF))}^{2}\leq M\right\} .
\end{align*}
Note that if $u_{0}$ is extended constantly in time, then on any
time interval $(0,T)$,
\[
f:=(u_{0}\cdot\nabla)u_{0}\in\Lp^{2}(\H1^{1/2+1/16}(\OF))\cap\H1^{1}(\H1^{-1/2+1/16}(\OF)),
\]
and Theorem \ref{lem:lin} provides a solution of the corresponding
linear system \eqref{eq:lin_system}. In particular, the estimate
(\ref{eq:est_H1H1_step3}) shows that $X_{T}^{0,M}\neq\emptyset$
if $M>0$ is chosen sufficiently large. Given any $\tilde{u}\in X_{T}^{0,M}$,
Lemma \ref{lem:utilde_statt_f-1} a) shows that 
\begin{align*}
\tilde{f}:=(\tilde{u}\cdot\nabla)\tilde{u}\in\Lp^{2}(\H1^{1/2+1/16}(\OF))\cap\H1^{1}(\H1^{-1/2+1/16}(\OF)),
\end{align*}
so Theorem~\ref{lem:lin} provides a solution $(u,p,\xi)$ of the
linear system (\ref{eq:lin_system}) with $f=\tilde{f}$. It remains
to prove that the map $S\colon\tilde{u}\to u$ is well-defined and
a contraction from $X_{T}^{0,M}$ to $X_{T}^{0,M}$. Theorem~\ref{lem:lin}
shows that $u\in X_{T}$ attains the correct initial values. To get
\begin{align*}
\Vert u\Vert_{\H1^{1}(\H1^{1}(\OF))\cap\C^{1}(\Lp^{2}(\OF))}^{2}\leq M
\end{align*}
for $T>0$ sufficiently small, we use estimate \eqref{eq:est_H1H1_step3}
to get
\begin{align*}
 & \Vert u\Vert_{\H1^{1}(\H1^{1}(\OF))\cap\C^{1}(\Lp^{2}(\OF))}^{2}+\Vert\varepsilon(\xi)\Vert_{\C^{0}(\Lp^{2}(\OS))}^{2}+\Vert\dot{\xi}\Vert_{\C^{0}(\H1^{1}(\OS))\cap\C^{1}(\Lp^{2}(\OS))}^{2}\\
\leq & \,C\Bigg(\Vert u_{0}\Vert_{\Lp^{2}(\OF)}^{2}+\Vert u_{1}\Vert_{\Lp^{2}(\OF)}^{2}+\Vert\varepsilon(\xi_{0})\Vert_{\Lp^{2}(\OS)}^{2}+\Vert\xi_{1}\Vert_{\H1^{1}(\OS)}^{2}+\Vert\xi_{2}\Vert_{\Lp^{2}(\OS)}^{2}\\
 & \hspace{0.5cm}+\int_{0}^{T}\int_{\OF}(\tilde{u}\cdot\nabla)\tilde{u}\cdot u\,\dx y\dx s+\int_{0}^{T}\int_{\OF}\left((\dot{\tilde{u}}\cdot\nabla)\tilde{u}+(\tilde{u}\cdot\nabla)\dot{\tilde{u}}\right)\cdot\dot{u}\,\dx y\dx s\Bigg).
\end{align*}
Now Lemma \ref{lem:utilde_statt_f-1} implies
\begin{align*}
 & \int_{0}^{T}\int_{\OF}(\tilde{u}\cdot\nabla)\tilde{u}\cdot u\,\dx y\dx s+\int_{0}^{T}\int_{\OF}\left((\dot{\tilde{u}}\cdot\nabla)\tilde{u}+(\tilde{u}\cdot\nabla)\dot{\tilde{u}}\right)\cdot\dot{u}\,\dx y\dx s\\
\leq & \,CT^{\alpha}\left(\Vert\tilde{u}\Vert_{\H1^{1}(\H1^{1}(\OF))}+\Vert u_{0}\Vert_{\H1^{1}(\OF)}\right)\Vert\tilde{u}\Vert_{\H1^{1}(\H1^{1}(\OF))\cap\C^{1}(\Lp^{2}(\OF))}\Vert u\Vert_{\H1^{1}(\H1^{1}(\OF))\cap\C^{1}(\Lp^{2}(\OF))}
\end{align*}
and hence 
\begin{align}
\begin{split} & \Vert u\Vert_{\H1^{1}(\H1^{1}(\OF))\cap\C^{1}(\Lp^{2}(\OF))}^{2}+\Vert\varepsilon(\xi)\Vert_{\C^{0}(\Lp^{2}(\OS))}^{2}+\Vert\dot{\xi}\Vert_{\C^{0}(\H1^{1}(\OS))\cap\C^{1}(\Lp^{2}(\OS))}^{2}\\
\leq & \,C\Big(\Vert u_{0}\Vert_{\Lp^{2}(\OF)}^{2}+\Vert u_{1}\Vert_{\Lp^{2}(\OF)}^{2}+\Vert\varepsilon(\xi_{0})\Vert_{\Lp^{2}(\OS)}^{2}+\Vert\xi_{1}\Vert_{\H1^{1}(\OS)}^{2}+\Vert\xi_{2}\Vert_{\Lp^{2}(\OS)}^{2}\\
 & \hspace{0.5cm}+T^{\alpha}\left(M^{2}+\Vert u_{0}\Vert_{\H1^{1}(\OF)}^{2}\right)M^{2}\Big).
\end{split}
\label{eq:estnonlinear}
\end{align}
Thus for $M>0$ sufficiently large and $T=T(M,\Vert u_{0}\Vert_{\H1^{1}(\OF)})>0$
sufficiently small, we obtain 
\begin{align*}
\Vert u\Vert_{\H1^{1}(\H1^{1}(\OF))\cap\C^{1}(\Lp^{2}(\OF))}^{2}\leq M
\end{align*}
and consequently $u\in X_{T}^{0,M}$, such that $S:X_{T}^{0,M}\to X_{T}^{0,M}$
is well-defined. To prove that $S$ is also a contraction, let $(u^{1},p^{1},\xi^{1}),\,(u^{2},p^{2},\xi^{2})\in X_{T}\times Y_{T}\times Z_{T}$
denote the solutions of (\ref{eq:lin_system}) corresponding to some
$\tilde{u}^{1},\,\tilde{u}^{2}\in X_{T}^{0,M}$, respectively. By
repeating the calculations from \eqref{eq:estnonlinear} for the difference
$u^{1}-u^{2}$, we obtain that 
\begin{align}
\begin{split} & \Vert u^{1}-u^{2}\Vert_{\H1^{1}(\H1^{1}(\OF))\cap\C^{1}(\Lp^{2}(\OF))}^{2}+\Vert\varepsilon(\xi^{1}-\xi^{2})\Vert_{\C^{0}(\Lp^{2}(\OS))}^{2}+\Vert\dot{\xi}^{1}-\dot{\xi}^{2}\Vert_{\C^{0}(\H1^{1}(\OS))\cap\C^{1}(\Lp^{2}(\OS))}^{2}\\
\leq & \,CT^{\alpha}\left(M^{2}+\Vert u_{0}\Vert_{\H1^{1}(\OF)}^{2}\right)\Vert\tilde{u}^{1}-\tilde{u}^{2}\Vert_{X_{T}}^{2}.
\end{split}
\label{eq:step4a}
\end{align}
Furthermore, using Theorem~\ref{thm:elliptic_stokes} with $s=0$,
Theorem~\ref{thm:elliptic_lame} with $s=1$ and Lemma~\ref{lem:utilde_statt_f-1}
b) yields 
\begin{align}
\Vert u^{1}-u^{2}\Vert_{\Lp^{2}(\H1^{2}(\OF))} & \leq\,C\Big(\Vert\dot{u}^{1}-\dot{u}^{2}\Vert_{\Lp^{2}(\Lp^{2}(\OF))}+\Vert\tilde{u}^{1}\cdot\nabla(\tilde{u}^{1}-\tilde{u}^{2})\Vert_{\Lp^{2}(\Lp^{2}(\OF))}\label{eq:step4}\\
 & \hspace{0.3cm}+\Vert(\tilde{u}^{1}-\tilde{u}^{2})\cdot\nabla\tilde{u}^{2}\Vert_{\Lp^{2}(\Lp^{2}(\OF))}+\Vert\Sigma(\xi^{1}-\xi^{2})\Vert_{\Lp^{2}(\H1^{1/2}(\dOS))}\Big)\nonumber \\
 & \leq\,C\Big(T^{1/2}\Vert u^{1}-u^{2}\Vert_{\C^{1}(\Lp^{2}(\OF))}+T^{\alpha}(M+\Vert u_{0}\Vert_{\H1^{1}(\OF)})\Vert\tilde{u}^{1}-\tilde{u}^{2}\Vert_{X_{T}}\nonumber \\
 & \hspace{0.3cm}+T^{1/2}\Vert\dot{\xi}^{1}-\dot{\xi}^{2}\Vert_{\C^{1}(\Lp^{2}(\OS))}+T\Vert u^{1}-u^{2}\Vert_{\Lp^{2}(\H1^{2}(\OF))}\Big).\nonumber 
\end{align}
Combining \eqref{eq:step4} and \eqref{eq:step4a} for $T=T\left(M,\,\Vert u_{0}\Vert_{\H1^{1}(\OF)}\right)>0$
sufficiently small, we obtain that 
\begin{align*}
\Vert u^{1}-u^{2}\Vert_{X_{T}}\leq\frac{1}{2}\Vert\tilde{u}^{1}-\tilde{u}^{2}\Vert_{X_{T}}.
\end{align*}
This shows that $S$ admits a unique fixed point $u\in X_{T}^{0,M}\subseteq X_{T}$.
Together with the corresponding $p\in Y_{T},\,\xi\in Z_{T}$, we have
obtained the unique solution to the non-linear system (\ref{eq:nonlin_system}).
\\
 \textbf{Step 4: Additional regularity.}\\
 Finally, we show that by boot-strapping, $(u,p,\xi)\in X_{T}\times Y_{T}\times Z_{T}$
also implies that 
\[
u\in\C^{0}(\H1^{2}(\OF))\text{ and }p\in\C^{0}(\H1^{1}(\OF)).
\]
By Lemma \ref{lem:utilde_statt_f-1}, 
\[
(u\cdot\nabla)u\in\Lp^{2}(\H1^{1/2+1/16}(\OF))\cap\H1^{1}(\H1^{-1/2+1/16}(\OF)),
\]
so $(u\cdot\nabla)u\in\C^{0}(\Lp^{2}(\OF))$, see \cite[Theorem 3.1]{LionsMagenes}.
Hence (\ref{addReg}) follows from Theorem~\ref{thm:elliptic_stokes}
with right-hand side $f=-\dot{u}-(u\cdot\nabla)u\in\C^{0}(\Lp^{2}(\OF))$
and Neumann boundary data $\Sigma(\xi)n\in\C^{0}(\H1^{3/2}(\dOS))$.
This concludes the proof of Theorem~\ref{thm:lokale_existenz}. 
\end{proof}

\section{Existence of global solutions for small data}

\label{SecGlobal}

\label{globEx} We show that for small initial data, the local solution
extends to a global one. Recall the energy 
\[
E(t):=\Vert u(t)\Vert_{\Lp^{2}(\OF)}^{2}+\Vert\dot{\xi}(t)\Vert_{\Lp^{2}(\OS)}^{2}+\int_{\OS}\Sigma(\xi):\varepsilon(\xi)(t)\,\dx y,
\]
associated to system (\ref{eq:nonlin_system}), and the corresponding
higher-order quantity 
\[
K(t):=\Vert\dot{u}(t)\Vert_{\Lp^{2}(\OF)}^{2}+\Vert\ddot{\xi}(t)\Vert_{\Lp^{2}(\OS)}^{2}+\int_{\OS}\Sigma(\dot{\xi}):\varepsilon(\dot{\xi})(t)\,\dx y.
\]

\begin{thm}
\label{thm:globale_existenz} There exist constants $C_{u}>0,\,C_{E}>0,\,C_{K}>0$
such that for any initial data 
\begin{align*}
d:=(u_{0},u_{1},p_{0},\xi_{0},\xi_{1},\xi_{2})\in\H1^{2}(\OF)\times\H1^{1}(\OF)\times\H1^{1}(\OF)\times\H1^{2}(\OS)\times\H1^{1}(\OS)\times\Lp^{2}(\OS)
\end{align*}
satisfying the compatibility conditions (\ref{eq:compatibility_thm})
and the bounds 
\begin{align}
\Vert u_{0}\Vert_{\H1^{1}(\OF)}\leq C_{u},\hspace{0.5cm}E(0)\leq C_{E},\hspace{0.5cm}K(0)\leq C_{K},\label{inbounds}
\end{align}
the corresponding unique solution $(u,p,\xi)$ to (\ref{eq:nonlin_system})
exists up to any time $T>0$. 
\end{thm}

\begin{proof}
From Theorem \ref{thm:lokale_existenz} we obtain a unique solution
\begin{align*}
(u,p,\xi)\in & \left(\C^{0}(0,T;\H1^{2}(\OF))\cap\H1^{1}(0,T;\H1^{1}(\OF))\cap\C^{1}(0,T;\Lp^{2}(\OF))\right)\\
 & \times\C^{0}(0,T;\H1^{1}(\OF))\\
 & \times\left(\C^{0}(0,T;\H1^{2}(\OS))\cap\C^{1}(0,T;\H1^{1}(\OS))\cap\C^{2}(0,T;\Lp^{2}(\OS))\right)
\end{align*}
up to some time 
\begin{align*}
T\left(\Vert u_{0}\Vert_{\H1^{1}(\OF)},\Vert u_{1}\Vert_{\Lp^{2}(\OF)},\Vert\varepsilon(\xi_{0})\Vert_{\Lp^{2}(\OS)},\Vert\xi_{1}\Vert_{\H1^{1}(\OS)},\Vert\xi_{2}\Vert_{\Lp^{2}(\OS)}\right)>0.
\end{align*}
We show that if condition (\ref{inbounds}) is satisfied for suitable
$C_{E},C_{K},C_{u}>0$, then 
\[
\Vert u(t)\Vert_{\H1^{1}(\OF)}\leq1,\,E(t)\leq E(0)\text{ and }K(t)\leq K(0)\text{ for all }t\in[0,T).
\]
We find some $0<\delta<1$ such that we can choose $M:=\frac{c_{1}}{2c_{2}}$
in the proof of Theorem \ref{thm:lokale_existenz}, if the initial
data satisfy 
\[
\max\left\{ \Vert u_{0}\Vert_{\H1^{1}(\OF)}^{2},\Vert u_{1}\Vert_{\Lp^{2}(\OF)}^{2},\Vert\varepsilon(\xi_{0})\Vert_{\Lp^{2}(\OS)}^{2},\Vert\xi_{1}\Vert_{\H1^{1}(\OS)}^{2},\Vert\xi_{2}\Vert_{\Lp^{2}(\OS)}^{2}\right\} \leq\frac{\delta}{2}.
\]
There exists some constant $C_{0}>0$ such that 
\[
\Vert u_{0}\Vert_{\Lp^{2}(\OF)}^{2}+\Vert u_{1}\Vert_{\Lp^{2}(\OF)}^{2}+\Vert\varepsilon(\xi_{0})\Vert_{\Lp^{2}(\OS)}^{2}+\Vert\xi_{1}\Vert_{\H1^{1}(\OS)}^{2}+\Vert\xi_{2}\Vert_{\Lp^{2}(\OS)}^{2}\leq C_{0}(E(0)+K(0)),
\]
so we consider initial data satisfying $C_{0}(E(0)+K(0))\leq\frac{\delta}{2}$
and assume that $\Vert u_{0}\Vert_{\H1^{1}(\OF)}^{2}\leq\min\left\{ \frac{\delta}{2},\frac{M\delta}{2}\right\} ,$so
that $\Vert u\Vert_{\H1^{1}(0,T;\H1^{1}(\OF))}^{2}\leq M=\frac{c_{1}}{2c_{2}}.$
By Korn's and Poincaré's inequalities, there exist constants $c_{1},c_{2}>0$
such that 
\begin{align}
c_{1}\Vert\varepsilon(v)\Vert_{\Lp^{2}(\OF)}^{2} & \leq\Vert v\Vert_{\H1^{1}(\OF)}^{2}\leq c_{2}\Vert\varepsilon(v)\Vert_{\Lp^{2}(\OF)}^{2}\label{eq:Korn}
\end{align}
holds for all $v\in\H1^{1}(\OF)$ with partially vanishing trace at
the boundary $\dOF$. In particular, this is true for $u(t)$ and
$\dot{u}(t)$. Hence 
\begin{align}
\Vert u(t)\Vert_{\H1^{1}(\OF)}^{2} & \leq c_{2}\Vert\varepsilon(u_{0})\Vert_{\Lp^{2}(\OF)}^{2}+c_{2}\int_{0}^{t}\int_{\OF}2\varepsilon(u):\varepsilon(\dot{u})\,\dx y\dx s\leq\frac{c_{2}M\delta}{2c_{1}}+\frac{c_{2}M}{c_{1}}<1\label{eq:H1_u_1}
\end{align}
for all $t\in[0,T)$. Next, we want to bound $E$. As in Theorem~\ref{lem:lin},
we obtain the energy equality 
\begin{align}
E(t)+\int_{0}^{t}2\nu\Vert\varepsilon(u(s))\Vert_{\Lp^{2}(\OF)}^{2}\,\dx s=E(0)-\int_{0}^{t}\int_{\OF}(u\cdot\nabla)u\cdot u\,\dx y\dx s.\label{eq:energy_eq_extended}
\end{align}
For the second term on the right-hand side of (\ref{eq:energy_eq_extended}),
using the usual embeddings and interpolation and the $1$-bound in
(\ref{eq:H1_u_1}) to decrease exponents, we estimate 
\begin{align}
\int_{\OF}[(u\cdot\nabla)u\cdot u](s)\,\dx y & \leq C\Vert u(s)\Vert_{\Lp^{2}(\OF)}^{1/2}\Vert u(s)\Vert_{\H1^{1}(\OF)}^{2}\leq\hat{C}E(s)^{1/2}\Vert\varepsilon(u(s))\Vert_{\Lp^{2}(\OF)}^{2}\label{eq:graduu}
\end{align}
for some fixed $\hat{C}>0$ that depends only on $\OF$. This leads
to 
\begin{align}
E(t)+\int_{0}^{t}(2\nu-\hat{C}E(s)^{1/2})\Vert\varepsilon(u(s))\Vert_{\Lp^{2}(\OF)}^{2}\,\dx s\leq E(0).\label{eq:zwischenschritt_E}
\end{align}
If the initial data is chosen so small that 
\begin{align}
2\nu-\hat{C}E(0)^{1/2}>0,\label{eq:cond_for_E0}
\end{align}
then we can show that 
\begin{align}
E(t)\leq E(0)\hspace{0.3cm}\text{for all }t\in[0,T):\label{eq:bound_E}
\end{align}
Assume to the contrary that there is a time $t_{0}\in(0,T)$ such
that $E(t_{0})>E(0)$. Because of (\ref{eq:cond_for_E0}), there is
some $\tilde{E}>E(0)$ which still satisfies $2\nu-\hat{C}\tilde{E}^{1/2}>0$.
As $E$ is continuous, choose a time $t_{1}\in(0,T)$ such that $E(t_{1})>E(0)$
and $E(t)\leq\tilde{E}\hspace{0.3cm}\text{for all }t\leq t_{1}.$But
then 
\begin{align*}
E(t_{1})+\underbrace{\int_{0}^{t_{1}}\left(2\nu-\hat{C}\tilde{E})^{1/2}\right)\Vert\varepsilon(u(s))\Vert_{\Lp^{2}(\OF)}^{2}\,\dx s}_{\geq0}\leq E(0),
\end{align*}
a contradiction. Moreover, we will need later that by (\ref{eq:zwischenschritt_E})
and (\ref{eq:bound_E}), 
\begin{align}
\int_{0}^{t}\Vert\varepsilon(u(s))\Vert_{\Lp^{2}(\OF)}\,\dx s\leq\frac{E(0)}{2\nu-\hat{C}E(0)^{1/2}}.\label{eq:bound_int_eps_u}
\end{align}
Next, we derive a similar result for the higher-order quantity $K$.
By Theorem~\ref{lem:lin}, 
\[
K(t)+\int_{0}^{t}2\nu\Vert\varepsilon(\dot{u}(s))\Vert_{\Lp^{2}(\OF)}^{2}\,\dx s=K(0)-\int_{0}^{t}\int_{\OF}\left((\dot{u}\cdot\nabla)u+(u\cdot\nabla)\dot{u}\right)\cdot\dot{u}\,\dx y\dx s.
\]
First, using \eqref{eq:Korn}, we estimate 
\begin{align}
\int_{\OF}[(\dot{u}\cdot\nabla)u\cdot\dot{u}](s)\,\dx y & \leq C\Vert\varepsilon(u(s))\Vert_{\Lp^{2}(\OF)}\Vert\varepsilon(\dot{u}(s))\Vert_{\Lp^{2}(\OF)}^{2},\label{eq:est_dtugradudtu_holder}
\end{align}
and similarly 
\begin{align}
\begin{split}\int_{\OF}[(u\cdot\nabla)\dot{u}\cdot\dot{u}](s)\,\dx y\leq C\Vert\varepsilon(u(s))\Vert_{\Lp^{2}(\OF)}\Vert\varepsilon(\dot{u}(s))\Vert_{\Lp^{2}(\OF)}^{2}.\end{split}
\label{eq:est_dtugradudtu_holder_2}
\end{align}
Moreover, using the differentiated energy balance, the definition
of $E$ and $K,$ \eqref{eq:graduu} and (\ref{eq:bound_E}), we have
\begin{align*}
\begin{split}2\nu\Vert\varepsilon(u(s))\Vert_{\Lp^{2}(\OF)}^{2} & =-\dot{E}(s)-\int_{\OF}[(u\cdot\nabla)u\cdot u](s)\,\dx y\\
 & \leq C\left(E(s)+K(s)\right)+\hat{C}E(s)^{1/2}\Vert\varepsilon(u(s))\Vert_{\Lp^{2}(\OF)}^{2}\\
 & \leq C\left(E(0)+K(s)\right)+\hat{C}E(0)^{1/2}\Vert\varepsilon(u(s))\Vert_{\Lp^{2}(\OF)}^{2}
\end{split}
\end{align*}
for all $s\in[0,T)$. Hence, (\ref{eq:cond_for_E0}) implies 
\begin{align}
\begin{split}\Vert\varepsilon(u(s))\Vert_{\Lp^{2}(\OF)}^{2}\leq C\left(E(0)+K(s)\right).\end{split}
\label{eq:est_epsu}
\end{align}
Combining (\ref{eq:est_dtugradudtu_holder}), (\ref{eq:est_dtugradudtu_holder_2})
and (\ref{eq:est_epsu}) leads to 
\begin{align}
K(t)+\int_{0}^{t}\left(2\nu-\tilde{C}\left(E(0)+K(s)\right)^{1/2}\right)\Vert\varepsilon(\dot{u}(s))\Vert_{\Lp^{2}(\OF)}^{2}\,\dx s\leq K(0)\label{eq:zwischenschritt_K}
\end{align}
for some fixed $\tilde{C}>0$. If the initial data is chosen so small
that $2\nu-\tilde{C}\left(E(0)+K(0)\right)^{1/2}>0$, then a similar
argument as before shows that 
\begin{align}
K(t)\leq K(0)\hspace{0.3cm}\text{for all }t\in[0,T).\label{eq:bound_K}
\end{align}
Moreover by using (\ref{eq:zwischenschritt_K}) and (\ref{eq:bound_K}),
we obtain that 
\begin{align}
\int_{0}^{t}\Vert\varepsilon(\dot{u})(s)\Vert_{\Lp^{2}(\OF)}^{2}\,\dx s\leq\frac{K(0)}{2\nu-\tilde{C}\left(E(0)+K(0)\right)^{1/2}}.\label{eq:bound_int_eps_dotu}
\end{align}
This proves uniform and even decreasing bounds for $\Vert u(t)\Vert_{\H1^{1}(\OF)}$,
$E(t)$ and $K(t)$ on $[0,T)$. 
 Since the lifespan $T$ of the local solution given in Theorem \ref{thm:lokale_existenz}
only depends decreasingly on the corresponding norms of the initial
data, and compatibility is conserved along solutions, their existence
extends to any finite time interval. 
\end{proof}
\begin{cor}
\label{cor:ConvofFluidvel}Consider initial data $(u_{0},u_{1},p_{0},\xi_{0},\xi_{1},\xi_{2})$
satisfying the conditions of Theorem \ref{thm:globale_existenz}.
Then the corresponding global solution $(u,p,\xi)$ to (\ref{eq:nonlin_system})
satisfies 
\begin{align*}
\lim_{t\to\infty}\Vert u(t)\Vert_{\H1^{1}(\OF)}=0.
\end{align*}
\end{cor}

\begin{proof}
The global solution $(u,p,\xi)$ satisfies (\ref{eq:bound_int_eps_u})
and (\ref{eq:bound_int_eps_dotu}) for all $t>0$, so the dissipation
rates are globally integrable, 
\begin{align}
\int_{0}^{\infty}\Vert\varepsilon(u(t))\Vert_{\Lp^{2}(\OF)}^{2}\,\dx t<\infty,\hspace{0.2cm}\int_{0}^{\infty}\Vert\varepsilon(\dot{u}(t))\Vert_{\Lp^{2}(\OF)}^{2}\,\dx t<\infty.\label{eq:est_Linfty}
\end{align}
For every $\delta>0$, there is thus a time $T_{\delta}>0$ such that
\[
\int_{T_{\delta}}^{\infty}\Vert\varepsilon(u(t))\Vert_{\Lp^{2}(\OF)}^{2}\,\dx t<\frac{\delta}{2},\hspace{0.2cm}\int_{T_{\delta}}^{\infty}\Vert\varepsilon(\dot{u}(t))\Vert_{\Lp^{2}(\OF)}^{2}\,\dx t<\frac{\delta}{2},
\]
and a time $t_{\delta}^{1}\geq T_{\delta}$ such that 
\[
\Vert\varepsilon(u(t_{\delta}^{1}))\Vert_{\Lp^{2}(\OF)}^{2}\leq\delta.
\]
For $t\geq t_{\delta}^{1}$, by (\ref{eq:Korn}) and the fundamental
theorem of calculus, we thus obtain
\begin{align*}
\Vert u(t)\Vert_{\H1^{1}(\OF)}^{2} & \leq c_{2}\Vert\varepsilon(u(t_{\delta}^{1}))\Vert_{\Lp^{2}(\OF)}^{2}+c_{2}\int_{t_{\delta}^{1}}^{t}\int_{\OF}2\varepsilon(u(s)):\varepsilon(\dot{u}(s))\,\dx y\dx s\\
 & \leq c_{2}\Vert\varepsilon(u(t_{\delta}^{1}))\Vert_{\Lp^{2}(\OF)}^{2}+c_{2}\int_{t_{\delta}^{1}}^{t}\Vert\varepsilon(u(s))\Vert_{\Lp^{2}(\OF)}^{2}\,\dx s+c_{2}\int_{t_{\delta}^{1}}^{t}\Vert\varepsilon(\dot{u}(s))\Vert_{\Lp^{2}(\OF)}^{2}\,\dx s\\
 & \leq2c_{2}\delta.
\end{align*}
\end{proof}
A consequence of $u(t)\to0$ as $t\to\infty$ is that the fluid pressure
$p(t)$ becomes spatially constant as $t\to\infty$. The precise statement
is the following 
\begin{cor}
\label{cor:convofpressure}Consider initial data $(u_{0},u_{1},p_{0},\xi_{0},\xi_{1},\xi_{2})$
satisfying the conditions of Theorem \ref{thm:globale_existenz} and
the corresponding global solution $(u,p,\xi)$ to (\ref{eq:nonlin_system}).
Define 
\begin{align*}
q(t):=\frac{1}{\vert\OF\vert}\int_{\OF}p(t,y)\dx y\quad\text{and }\qquad\hat{p}(t,y):=p(t,y)-q(t).
\end{align*}
Then
\begin{align*}
\lim_{t\to\infty}\Vert\hat{p}(\cdot+t)\Vert_{\Lp^{2}((0,T/2)\times\OF)}=0\quad\text{and}\quad\lim_{t\to\infty}\Vert\Sigma(\xi(\cdot+t)n-q(\cdot+t)n\Vert_{\Lp^{2}(\H1^{-1/2}(\dOS))}=0.
\end{align*}
\end{cor}

\begin{proof}
For general $f\in\Lp^{2}((0,T/2)\times\OF)$ define 
\begin{align*}
\hat{f}(t,y):=f(t,y)-\frac{1}{\vert\OF\vert}\int_{\OF}f(t,y)\dx y
\end{align*}
as above. By {\cite[Theorem~2.5]{GHH06}}, i.e. applying the Bogovskii
operator, there exists some $g\in\Lp^{2}(\H1_{0}^{1}(\OF))$ such
that $\Div(g)=\hat{f}$ and $\Vert g\Vert_{\Lp^{2}(\H1^{1}(\OF))}\leq C\Vert f\Vert_{\Lp^{2}((0,T/2)\times\OF)}$.
Hence 
\begin{align*}
\int_{0}^{T/2}\int_{\OF}\hat{p}(s+t,y)f(s,y)\,\dx y & =\,\int_{0}^{T/2}\int_{\OF}\hat{p}(s+t,y)\hat{f}(s,y)\,\dx y\dx s\\
 & =\,-\int_{0}^{T/2}\int_{\OF}\nabla\hat{p}(s+t,y)\cdot g(s,y)\,\dx y\dx s=\\
 & =\,-\int_{0}^{T/2}\int_{\OF}(\Delta u-\dot{u}-(u\cdot\nabla)u)(s+t,y)\cdot g(s,y)\,\dx y\dx s\\
 & \leq\,C\left(\int_{0}^{T/2}\Vert(\varepsilon(u(s+t))\Vert_{\Lp^{2}(\OF)}^{2}\,\dx y\right)\Vert g\Vert_{\Lp^{2}(\H1^{1}(\OF))}^{2}\\
 & +C\left(\int_{0}^{T/2}\left(\Vert\dot{u}(s+t)\Vert_{\Lp^{2}(\OF)}^{2}+\Vert u(s+t)\Vert_{\H1^{1}(\OF)}^{4}\right)\,\dx s\right)\Vert g\Vert_{\Lp^{2}(\H1^{1}(\OF)).}^{2}
\end{align*}
By \eqref{eq:est_Linfty}, the terms on the right-hand side vanish
as $t\to\infty$. Now let 
\[
E_{F}:\Lp^{2}(\H1^{1/2}(\dOS))\to\Lp^{2}(\H1^{1}(\OF))
\]
denote a bounded linear extension operator that satisfies $\text{supp}(E_{F}h(t))\subset\subset\Omega$
for all $h\in\Lp^{2}(\H1^{1/2}(\dOS))$ and $t\in(0,T/2)$. Then similarly,
\begin{align*}
 & \int_{t}^{t+T/2}\int_{\dOS}(\Sigma(\xi)n-qn)\cdot h\,\dx S(y)\dx s\\
=\, & \int_{t}^{t+T/2}\int_{\dOS}\sigma(u,\hat{p})n\cdot h\,\dx S(y)\dx s\\
=\, & \int_{t}^{t+T/2}\int_{\OF}\sigma(u,\hat{p}):\nabla(E_{F}h)\,\dx y\dx s+\int_{0}^{T/2}\int_{\OF}\Div(\sigma(u,\hat{p}))\cdot E_{F}h\,\dx y\dx s\\
\leq\, & \left(\int_{t}^{t+T/2}\Vert\varepsilon(u)\Vert_{\Lp^{2}(\OF)}^{2}\,\dx s+C\Vert\hat{p}(\cdot+t)\Vert_{\Lp^{2}((0,T/2)\times\OF)}^{2}\right)\Vert E_{F}h\Vert_{\Lp^{2}(\H1^{1}(\OF))}^{2}\\
 & +C\left(\int_{t}^{t+T/2}\Vert\dot{u}\Vert_{\Lp^{2}(\OF)}^{2}+\Vert u\Vert_{\H1^{1}(\OF)}^{4}\,\dx s\right)\Vert E_{F}h\Vert_{\Lp^{2}(\H1^{1}(\OF))}^{2}\to0,\,\text{as }t\to\infty.
\end{align*}
\end{proof}

\section{Long-Time Behaviour of $\xi$}

\label{SecXi}By Corollaries \ref{cor:ConvofFluidvel} and \ref{cor:convofpressure},
the fluid velocity vanishes, $u(t)\to0$, and the fluid pressure $p(t)$
becomes spatially constant as $t\to\infty$. The results of \cite{LLT1986,KTZ2011}
show continuous dependence on the data for regular solutions of the
Lamé system. Therefore, after taking care of initial deformations
$\xi_{0}|_{\dOS}$ at the interface in a suitable way, it is natural
to consider solutions $(\eta,q)$ of the system
\begin{align}
\begin{split}\begin{cases}
\begin{array}{rcll}
\ddot{\eta}-\Div(\Sigma(\eta)) & = & 0 & \text{in }(0,T)\times\OS,\\
\eta & = & 0 & \text{on }(0,T)\times\dOS,\\
\Sigma(\eta)n & = & qn & \text{on }(0,T)\times\dOS,
\end{array}\end{cases}\end{split}
\label{eq:xi_system_A}
\end{align}
with scalar $q(t)\in\mathbb{R}$ as the candidate limit dynamics for
system \eqref{eq:nonlin_system}. Note the presence of two boundary
conditions, but no initial condition in this problem. We denote the
corresponding set $\omega_{T}$ of suitably regular solutions on $(0,T)$
by 
\begin{align}
\begin{split}\omega_{T}:=A_{T}\times Q_{T}:=\big\{(\eta,qn)\in & \big(\C^{0}(\H1^{1}(\OS))\cap\C^{1}(\Lp^{2}(\OS))\cap\C^{2}(\H1^{-1}(\OS))\cap\H1^{2}(\H1^{1}(\OS)^{*})\\
 & \hspace{0.1cm}\times\Lp^{2}(\H1^{-1/2}(\dOS)):\\
 & (\eta,q)\text{ solve (\ref{eq:xi_system_A})}\big\},
\end{split}
\label{eq:defomega}
\end{align}
With a slight abuse of notation we call non-trivial solutions $\omega_{T}\ni(\eta,q)\neq(0,0)$
\emph{pressure waves}. The set $\omega_{T}$ is characterized by the
corresponding overdetermined eigenvalue problem 
\begin{align}
\begin{split}\begin{cases}
\begin{array}{rcll}
-\Div(\psi) & = & \mu\psi & \text{in }\OS,\\
\psi & = & 0 & \text{on }\dOS,\\
\Sigma(\psi)n & = & qn & \text{on }\dOS.
\end{array}\end{cases}\end{split}
\label{eq:xi_overdet_EP-1}
\end{align}
More precisely, by definition, regular solutions $\psi$ of \eqref{eq:xi_overdet_EP-1}
are contained in the set of eigenfunctions of the Dirichlet-Lamé operator
\begin{align*}
\mathcal{L}(\eta) & =-\Div(\Sigma(\eta))\qquad\text{with domain}\\
D(\mathcal{L}) & :=\H1_{0}^{1}(\OS)\cap\H1^{2}(\OS)\subset\Lp^{2}(\OS).
\end{align*}
This operator is self-adjoint positive definite and has compact resolvent.
It admits countably many eigenvalues $(\mu_{k})\subset(0,\infty)$
that have one-dimensional eigenspaces and can only cluster at infinity.
The corresponding eigenfunctions $(\psi_{k})\subset D(\mathcal{L})$
with indices $k\in K\subseteq\mathbb{N}$ solving 
\begin{align*}
\begin{split}\begin{cases}
\begin{array}{rcll}
\mathcal{L}(\psi_{k}) & = & \mu_{k}\psi_{k} & \text{in }\OS,\\
\psi_{k} & = & 0 & \text{on }\dOS,
\end{array}\end{cases}\end{split}
\end{align*}
can be chosen to form an orthonormal basis. Now by separation of variables,
it is straightforward to show that
\begin{align}
\omega_{T}= & \left\{ \sum_{k\in I}(a_{k}\sin(\sqrt{\mu_{k}}t)+b_{k}\cos(\sqrt{\mu_{k}}t))(\psi_{k}(y),q_{k}):I\subset\mathbb{N},\,(\psi_{k},q_{k})\text{ solve }(\ref{eq:xi_overdet_EP-1}),\,(a_{k}),(b_{k})\in\mathbb{R}\right\} \label{eq:charomega}\\
 & \hspace{0.5cm}\cap(\C^{0}(\H1_{0}^{1}(\OS))\cap\C^{1}(\Lp^{2}(\OS))\cap\C^{2}(\H1^{-1}(\OS))\cap\H1^{2}(\H1^{1}(\OS)^{*})))\times\C^{0}(0,T),\nonumber 
\end{align}
where %
\[
I:=\{i\in\mathbb{N}\colon\exists(\psi_{i},q_{i})\neq(0,0)\text{ solution of }\eqref{eq:xi_overdet_EP-1}_{\mu_{i}}\}
\]
is the set of indices $i\in\mathbb{N}$ of non-trivial eigenfunctions
$\psi_{i}$ that solve \eqref{eq:xi_overdet_EP-1}. Clearly $I\subseteq K$,
but no further general characterization of $I$ is known yet. 

Following the work of Avalos and Triggiani (cf.\@ \cite{5AT,ATuniformstab,ATboundary,8AT2009,4AT2009,AT2,Astrongstab}),
we use the following definition. 
\begin{defn}
The domain $\OS$ is called
\end{defn}

\begin{itemize}
\item a \emph{good} domain, if $\psi_{k}=0$, $q_{k}=0$ is the only solution
of (\ref{eq:xi_overdet_EP-1}), i.e. $I=\emptyset$. 
\item a \emph{bad} domain, if (\ref{eq:xi_overdet_EP-1}) admits a non-zero
solution $(\psi_{k},q_{k})$ for some $\mu_{k}>0$, i.e. $I\neq\emptyset$.
\end{itemize}
\begin{prop}
It was shown in \cite{5AT,AT2} for a similar wave-type system that
every domain which is partially flat, partially spherical, partially
elliptic, partially hyperbolic or partially parabolic is a \emph{good}
domain, and this result was transferred to the Lamé system in \cite[Remark~1.1]{ATboundary}. 
\end{prop}

\begin{example}
\label{exa:ball}A known \emph{bad} domain is the ball \cite{5AT}.
For $\OS=B_{r}(0)$, examples of non-zero solutions to (\ref{eq:xi_overdet_EP-1})
are given by 
\begin{align*}
\psi_{i}(y) & :=\left(\frac{r^{2}\sin\left(\frac{r_{i}}{r}\vert y\vert\right)}{r_{i}^{2}\vert y\vert^{3}}-\frac{r\cos\left(\frac{r_{i}}{r}\vert y\vert\right)}{r_{i}\vert y\vert^{2}}\right)y,\\
q_{i} & :=(2\lambda_{1}+\lambda_{2})\sin(r_{i}),
\end{align*}
with eigenvalues 
\begin{align*}
\mu_{i}=\frac{(2\lambda_{1}+\lambda_{2})r_{i}^{2}}{r^{2}},
\end{align*}
where $r_{i}\in(0,\infty)$ is the $i$-th positive root of the spherical
Bessel function 
\begin{align*}
j_{1}(r)=\frac{\sin(r)}{r^{2}}-\frac{\cos(r)}{r}.
\end{align*}
\end{example}

We need a few more definitions in order to state the main result.
For $d=2,3$, the set of all skew-symmetric matrices in $\mathbb{R}^{d\times d}$
is denoted by 
\begin{align*}
\frak{so}(d):=\{A\in\mathbb{R}^{d\times d}:A^{T}=-A^{T}\},
\end{align*}
and the kernel of the symmetrized gradient $\varepsilon$ is denoted
by 
\begin{align*}
\mathcal{R}:=\text{ker}(\varepsilon)=\{r:\mathbb{R}^{d}\to\mathbb{R}^{d}:r(y)=Ay+b,\,A\in\frak{so}(d),\,b\in\mathbb{R}^{d}\}.
\end{align*}
The corresponding orthogonal projection $P_{\mathcal{R}}:\H1^{1}(\OS)\to\H1^{1}(\OS)$
is given by 
\begin{align*}
P_{\mathcal{R}}(f):=skew\left(\frac{1}{\vert\OS\vert}\int_{\OS}\nabla f\,\dx y\right)\left(y-\frac{1}{\vert\OS\vert}\int_{\OS}y\,\dx y\right)+\frac{1}{\vert\OS\vert}\int_{\OS}f\,\dx y,
\end{align*}
where 
\begin{align*}
skew(A):=\frac{1}{2}(A-A^{T}),\hspace{0.5cm}A\in\mathbb{R}^{d\times d}.
\end{align*}
Due to Korn's inequality, we have 
\begin{lem}[{{\cite[Appendix~A.2]{Korn}}}]
\label{lem:xi_korn-1} There exist constants $c,C>0$ such that 
\begin{align*}
c\Vert\varepsilon(f)\Vert_{\Lp^{2}(\OS)}\leq\Vert f-P_{\mathcal{R}}f\Vert_{\H1^{1}(\OS)}\leq C\Vert\varepsilon(f)\Vert_{\Lp^{2}(\OS)}
\end{align*}
for all $f\in\H1^{1}(\OS)$. 
\end{lem}

We denote by $\varphi_{N}$ the solution of the stationary Neumann
system 
\begin{align*}
\begin{cases}
\begin{array}{rcll}
\Div(\Sigma(\varphi)) & = & 0 & \text{ in }\OS,\\
\Sigma(\varphi)n & = & n & \text{ on }\dOS,
\end{array}\end{cases}
\end{align*}
with $P_{\mathcal{R}}\varphi_{N}=0$. For the $q$-parameterized stationary
Neumann system 
\begin{align}
\begin{cases}
\begin{array}{rcll}
\Div(\Sigma(\eta)) & = & 0 & \text{ in }\OS,\\
\Sigma(\eta)n & = & qn & \text{ on }\dOS,
\end{array}\end{cases}\label{eq:xi_stat_system-1}
\end{align}
we define the space 
\begin{align*}
\mathcal{E}:=\{\eta\in\H1^{1}(\OS):\exists q\in\mathbb{R}\text{ s.t. }(\eta,q)\text{ solves }(\ref{eq:xi_stat_system-1})\}=\Span\{\varphi_{N}\}+\mathcal{R}.
\end{align*}
The main result is the following. 
\begin{thm}
\label{thm:main result xi}Let $(u,p,\xi)$ be a global solution to
(\ref{eq:nonlin_system}) as in Theorem \ref{thm:globale_existenz}.
Then 

\begin{align}
\lim_{t\to\infty}\Vert\xi(t)-\eta^{*}(t)-\varphi_{N}^{0}-r(t)\Vert_{\H1^{1}(\OS)}=0,\label{eq:xi_conv_baddomain-1}
\end{align}
where the time-constant displacement $\varphi_{N}^{0}$ is determined
from $\xi_{0}$ via
\begin{align}
\varphi_{N}^{0}:=\frac{\int_{\OS}\Sigma(\xi_{0}):\varepsilon(\varphi_{N})\,\dx y}{\int_{\OS}\Sigma(\varphi_{N}):\varepsilon(\varphi_{N})\,\dx y}\varphi_{N}\in\Span\{\varphi_{N}\}\subset\mathcal{E},\label{eq:def_phiN0}
\end{align}
either $\eta^{*}=0$ or $\eta^{*}\in A$ is a pressure wave and 
\begin{align*}
r(t):=P_{\mathcal{R}}(\xi-\eta^{*})(t)\in\mathcal{R}
\end{align*}
is a rigid motion. In addition, $r$ disappears in rates in the sense
that 
\begin{align}
\lim_{t\to\infty}\Vert\dot{\xi}(t)-\dot{\eta}^{*}(t)\Vert_{\Lp^{2}(\OS)}=0\hspace{0.5cm}\text{ and }\hspace{0.5cm}\lim_{t\to\infty}\Vert\ddot{\xi}(t)-\ddot{\eta}^{*}(t)\Vert_{\H1^{-1}(\OS)}=0.\label{eq:xi_conv_deriv_baddomain-1}
\end{align}
\end{thm}

\begin{rem}
Our interpretation of Theorem \ref{thm:main result xi} is that
\end{rem}

\begin{itemize}
\item up to pressure waves, fluid-viscous damping of the elastic displacement
is established,
\item in specific geometric constellations, pressure waves may persist.
In this case, we have proved convergence of the solution to a fixed
pressure wave. 
\end{itemize}
\begin{rem}
\label{rem:pwsolve11}It is important to note and straightforward
to check that $u=0,(\xi,p)=(\eta,q)\in\omega_{T}$ also provide solutions
to system \eqref{eq:fullsystem-1}. In particular, the non-trivial
long-term behaviour of solutions in Theorem \ref{thm:main result xi}
with $\eta^{*}\neq0$ is not due to the fact that we freeze the fluid
domain $\OF$ and thus lose a potential dissipation mechanism. Regarding
the rigid motions $r$, the situation is less clear. We do not know
whether they can be shown to converge to a rest state. If the motion
may persist here, it can probably still be shown to disappear for
solutions of \eqref{eq:fullsystem-1}.
\end{rem}

\section{Long-time dynamics of $\tilde{\xi}_{t_{0}}$}

\label{SecTilde} Given a global solution $(u,p,\xi)$ of \eqref{eq:nonlin_system}
and $t_{0}>0$, we first analyse the \emph{time differences}
\[
\tilde{\xi}_{t_{0}}(t):=\xi(t_{0}+t)-\xi(t),\hspace{0.2cm}\tilde{p}_{t_{0}}(t):=p(t_{0}+t)-p(t)
\]
and prove \emph{their} convergence to a pressure wave in a suitable
sense. This trick solves some technical issues with missing compactness
in $\xi$ itself and the transmission boundary condition which involves
$\dot{\xi}$ instead of $\xi$. The uniformity with respect to the
shift $t_{0}$ in the following result will then allow us to extend
the analysis to $\xi,p$. The most relevant function space is 
\[
U_{T}:=\C^{0}(0,T;\H1^{1}(\OS))\cap\C^{1}(0,T;\Lp^{2}(\OS))\cap\C^{2}(0,T;\H1^{-1}(\OS))
\]
and $d_{T}:=d^{\,U_{T}}\times d^{\,\Lp^{2}(0,T;\H1^{-1/2}(\dOS))}$
is used as a notation for a corresponding metric. Note that the following
main result of this section is like the characterization of an attractor
for $\tilde{\xi}_{t_{0}},\tilde{p}_{t_{0}}$, but stronger in the
sense that the metric acts on trajectories instead of states. Convergence
to a fixed pressure wave will be proved in Step 2 in Section \ref{SecProof}
below. 
\begin{thm}
\label{thm:xi_approx_A} Let $(u,p,\xi)$ be the global solution of
(\ref{eq:nonlin_system}) given in Theorem \ref{thm:globale_existenz}.
Moreover, let $T>0$ and $t_{0}^{*}\in(0,T/2)$. Then
\begin{align}
\lim_{t\to\infty}\sup_{t_{0}\in(0,t_{0}^{*}]}d_{T/2}\left((\tilde{\xi}_{t_{0}}(\cdot+t),\Sigma(\tilde{\xi}_{t_{0}}(\cdot+t))n),\,\omega_{T/2}\right)=0.\label{eq:xi_thm_approx_A}
\end{align}
\end{thm}

\begin{proof}
Note that it suffices to prove that for all sequences $t_{n}\to\infty$,
there exists a subsequence $(t_{n_{k}})$ of $(t_{n})$ and a pair
of functions $(\eta,qn):=(\eta,qn)((t_{n_{k}}))\in\omega_{T/2}$ such
that 
\begin{align*}
\lim_{k\to\infty}\sup_{t_{0}\in(0,t_{0}^{*}]}\Vert & (\tilde{\xi}_{t_{0}}(\cdot+t_{n_{k}}),\Sigma(\tilde{\xi}_{t_{0}}(\cdot+t_{n_{k}}))n)-(\eta,qn)\Vert_{U_{T/2}\times\Lp^{2}(\H1^{-1/2}(\dOS))}=0.
\end{align*}
The existence of this subsequence will follow from a compactness argument.
We use the energy estimates 
\begin{align}
\Vert\dot{\xi}(t)\Vert_{\H1^{1}(\OS)}+\Vert\ddot{\xi}(t)\Vert_{\Lp^{2}(\OS)}\leq C(E(0)+K(0))\label{eq:apriorixidot}
\end{align}
for all $t\geq0$, and prove the convergence of a subsequence of $(\dot{\xi}(t_{n}))$
to the set $B_{T}$ which consists of all weak solutions $\eta\in\C^{0}(0,T,\H1^{1}(\OS))\cap\C^{1}(0,T,\Lp^{2}(\OS))$
of the homogeneous Dirichlet problem 
\begin{align}
\begin{cases}
\ddot{\eta}-\Div(\Sigma(\eta))=0 & \text{in }(0,T)\times\OS,\\
\eta=0 & \text{on }(0,T)\times\dOS.
\end{cases}\label{eq:xi_dirichletsys}
\end{align}
 Recall that the target set $A_{T}$ defined in \eqref{eq:defomega}
is a subset of $B_{T}$.

Let $E_{S}:\H1^{1/2}(\dOS)\to\H1^{1}(\OS)$ denote a bounded linear
extension operator and set 
\begin{align*}
\varphi_{0}^{n}:=\dot{\xi}(t_{n})-E_{S}\left(u(t_{n})\vert_{\dOS}\right)\in\H1^{1}(\OS)\hspace{0.2cm}\text{ and }\hspace{0.2cm}\varphi_{1}^{n}:=\ddot{\xi}(t_{n})\in\Lp^{2}(\OS).
\end{align*}
Then Theorem \ref{thm:lame_LLT} implies that system (\ref{eq:xi_dirichletsys})
with compatible inital data 
\begin{align*}
\varphi^{n}(0)=\varphi_{0}^{n},\hspace{0.2cm}\dot{\varphi}^{n}(0)=\varphi_{1}^{n}\hspace{0.2cm}\text{ in }\OS
\end{align*}
admits a unique solution $\varphi^{n}\in B_{T}$. The difference $\psi^{n}=\dot{\xi}(\cdot+t_{n})-\varphi^{n}$
solves 
\begin{align*}
\begin{cases}
\begin{array}{rcll}
\ddot{\psi}-\Div(\Sigma(\psi)) & = & 0 & \text{in }(0,T)\times\OS,\\
\psi & = & u(\cdot+t_{n}) & \text{on }(0,T)\times\dOS,\\
\psi(0) & = & E_{S}\left(u(t_{n})\vert_{\dOS}\right) & \text{in }\OS,\\
\dot{\psi}(0) & = & 0 & \text{in }\OS,
\end{array}\end{cases}
\end{align*}
with 
\[
\Vert\psi^{n}(0)\Vert_{\H1^{1}(\OS)}\leq C\Vert u(t_{n})\Vert_{\H1^{1}(\OF)}\to0\hspace{0.2cm}\text{ for }n\to\infty
\]
and 
\[
\Vert u(\cdot+t_{n})\Vert_{\C^{0}(\H1^{1/2}(\dOS))\cap\H1^{1}(\Lp^{2}(\dOS))}\leq C\left(\Vert u(\cdot+t_{n})\Vert_{\C^{0}(\H1^{1}(\OF))}+\left(\int_{t_{n}}^{T+t_{n}}\Vert\varepsilon(\dot{u})\Vert_{2}^{2}\dx s\right)^{1/2}\right)\overset{n\to\infty}{\to}0
\]
due to \eqref{eq:bound_int_eps_dotu}. Hence Theorem \ref{thm:lame_LLT_interior}
implies
\begin{align}
\psi^{n}\to0\qquad\text{in }\C^{0}(\H1^{1}(\OS))\cap\C^{1}(\Lp^{2}(\OS)).\label{eq:xi_phi_dotxi}
\end{align}
Due to \eqref{eq:apriorixidot}, $(\varphi^{n})$ is bounded in $\C^{0}(\H1^{1}(\OS))\cap\C^{1}(\Lp^{2}(\OS))$.
Moreover, for every $\psi\in\H1_{0}^{1}(\OS)$, 
\[
\langle\ddot{\varphi}^{n},\psi\rangle_{\H1^{-1}(\OS),\H1_{0}^{1}(\OS)}=-\int_{\OS}\Sigma(\varphi^{n}):\nabla\psi\,\dx y\leq C\Vert\varphi^{n}\Vert_{\C^{0}(\H1^{1}(\OS))}\Vert\psi\Vert_{\H1^{1}(\OS)},
\]
so that $(\varphi^{n})$ is also bounded in $\C^{2}(\H1^{-1}(\OS))$
and hence in $U_{T}$. 
The compact embedding 
\[
U_{T}\hookrightarrow^{c}\C^{0}(\Lp^{2}(\OS))\cap\C^{1}(\H1^{-1}(\OS))=:V_{T}
\]
implies the existence of a subsequence $(\varphi^{n_{k}})$ re-denoted
by $(\varphi^{n})$, with limit $\varphi\in V_{T}$ that solves the
homogeneous Dirichlet system (\ref{eq:xi_dirichletsys}) in a weaker
sense. Hence, 
\begin{align}
\lim_{n\to\infty}\Vert\dot{\xi}(\cdot+t_{n})-\varphi\Vert_{V_{T}}=0\label{eq:xi_phi_conv}
\end{align}
follows from \eqref{eq:xi_phi_dotxi} and \cite[Theorem~2.3]{LLT1986}.

Next, we derive uniform convergence in $t_{0}\in(0,t_{0}^{*}]$ of
the sequence $\left(\tilde{\xi}_{t_{0}}(\cdot+t_{n})\right)$ in $U_{T/2}$.
The candidate limit is 
\begin{align*}
\eta_{t_{0}}(s):=\int_{s}^{s+t_{0}}\varphi(r)\,\dx r\in\C^{1}(0,T/2;\Lp^{2}(\OS))\cap\C^{2}(0,T/2;\H1^{-1}(\OS)).
\end{align*}
We also define 
\begin{align*}
\eta_{t_{0}}^{n}(s):=\int_{s}^{s+t_{0}}\varphi^{n}(r)\,\dx r\in\C^{1}(0,T/2;\H1^{1}(\OS))\cap\C^{2}(0,T/2;\Lp^{2}(\OS)).
\end{align*}
We need to improve the spatial regularity of $\eta_{t_{0}}$. Since
$\eta_{t_{0}}(0)=\int_{0}^{t_{0}}\varphi(r)\,\dx r$ solves the elliptic
system 
\begin{align*}
\begin{cases}
\Div(\Sigma(\eta_{t_{0}}(0)))=\dot{\varphi}(t_{0})-\dot{\varphi}(0) & \text{in }\OS,\\
\eta_{t_{0}}(0)=0 & \text{on }\dOS,
\end{cases}
\end{align*}
for all $t_{0}\in(0,t_{0}^{*}]$, elliptic regularity of the Dirichlet-Lamé
operator implies $\eta_{t_{0}}(0)\in\H1^{1}(\OS)$. A similar argument
shows that the convergence of $(\varphi^{n})$ to $\varphi$ in $V_{T/2}$
implies the convergence of $\eta_{t_{0}}^{n}(0)$ to $\eta_{t_{0}}(0)$
in $\H1^{1}(\OS)$ uniformly for all $t_{0}\in(0,t_{0}^{*}]$. The
difference $\eta_{t_{0}}^{n}-\eta_{t_{0}}$ solves the homogeneous
Dirichlet system (\ref{eq:xi_dirichletsys}) up to time $T/2$ with
initial data 
\begin{align*}
\eta_{t_{0}}^{n}(0)-\eta_{t_{0}}(0)\in\H1^{1}(\OS)\hspace{0.2cm}\text{ and }\hspace{0.2cm}\dot{\eta}_{t_{0}}^{n}(0)-\dot{\eta}_{t_{0}}(0)=\varphi^{n}(t_{0})-\varphi(t_{0})+\varphi(0)-\varphi^{n}(0)\in\Lp^{2}(\OF).
\end{align*}
Hence, due to 
\begin{align*}
\Vert\varphi^{n}(t_{0})-\varphi(t_{0})+\varphi(0)-\varphi^{n}(0)\Vert_{\Lp^{2}(\OS)}\leq2\Vert\varphi^{n}-\varphi\Vert_{\C^{0}(\Lp^{2}(\OS))}\to0\hspace{0.2cm}\text{ for }n\to\infty,
\end{align*}
and Theorem \ref{thm:lame_LLT}, 
\begin{align}
\lim_{n\to\infty}\sup_{t_{0}\in(0,t_{0}^{*}]}\left\Vert \eta_{t_{0}}^{n}-\eta_{t_{0}}\right\Vert _{U_{T/2}}=0\label{eq:xi_conv_etan_eta}
\end{align}
and 
\begin{align}
\lim_{n\to\infty}\sup_{t_{0}\in(0,t_{0}^{*}]}\left\Vert \Sigma(\eta_{t_{0}}^{n})n-\Sigma(\eta_{t_{0}})n\right\Vert _{\Lp^{2}((0,T)\times\dOS)}=0.\label{eq:xi_conv_etan_eta_neumann}
\end{align}
For any $n\in\mathbb{N}$ and $t_{0}\in(0,t_{0}^{*}]$, the difference
$\psi_{t_{0}}^{n}:=\tilde{\xi}_{t_{0}}(\cdot+t_{n})-\eta_{t_{0}}^{n}$
solves 
\begin{align*}
\begin{cases}
\begin{array}{rcll}
\ddot{\psi}_{t_{0}}^{n}-\Div(\Sigma(\psi_{t_{0}}^{n})) & = & 0 & \text{in }(0,T/2)\times\OS,\\
\psi_{t_{0}}^{n} & = & \int_{\cdot+t_{n}}^{\cdot+t_{0}+t_{n}}u(r)\dx r & \text{on }(0,T/2)\times\dOS,\\
\psi_{t_{0}}^{n}(0) & = & \xi(t_{0}+t_{n})-\xi(t_{n})-\eta^{n}(0) & \text{in }\OS,\\
\dot{\psi_{t_{0}}}^{n}(0) & = & \psi^{n}(t_{0})-\psi^{n}(0) & \text{in }\OS,
\end{array}\end{cases}
\end{align*}
where 
\begin{equation}
\left\Vert \int_{\cdot+t_{n}}^{\cdot+t_{0}+t_{n}}u(r)\dx r\right\Vert _{\C^{0}(0,T/2;\H1^{1/2}(\dOS))\cap\H1^{1}(0,T/2;\Lp^{2}(\dOS))}\to0\qquad\text{for }n\to\infty,\label{eq:int_of_u}
\end{equation}
and clearly
\[
\Vert\dot{\psi}_{t_{0}}^{n}(0)\Vert_{\Lp^{2}(\OS)}\leq2\Vert\psi^{n}\Vert_{\C^{0}(\H1^{1}(\OS))\cap\C^{1}(\Lp^{2}(\OS))}\to0\hspace{0.2cm}\text{ for }n\to\infty.
\]
Due to
\begin{align*}
\begin{cases}
\begin{array}{rcll}
\Div(\Sigma(\psi_{t_{0}}^{n}(0))) & = & \dot{\psi}^{n}(t_{0})-\dot{\psi}^{n}(0) & \text{in }\OS,\\
\psi_{t_{0}}^{n}(0) & = & \int_{t_{n}}^{t_{0}+t_{n}}u(r)\dx r & \text{on }\dOS,
\end{array}\end{cases}
\end{align*}
\eqref{eq:int_of_u} and \eqref{eq:xi_phi_conv} show that 
\[
\lim_{n\to\infty}\Vert\psi_{t_{0}}^{n}(0)\Vert_{\H1^{1}(\OS)}=0.
\]
Hence Theorem \ref{thm:lame_LLT_interior} implies 
\begin{align}
\lim_{n\to\infty}\sup_{t_{0}\in(0,t_{0}^{*}]}\left\Vert \psi_{t_{0}}^{n}\right\Vert _{U_{T/2}}=0.\label{eq:xi_conv_xi_etan}
\end{align}
At the boundary, for every $h\in\H1^{1/2}(\dOS)$, we obtain 
\begin{align*}
\int_{\dOS}\left(\Sigma(\psi_{t_{0}}^{n})n\right)\cdot h\,\dx S(y) & =\,\int_{\OS}\left(\Sigma(\psi_{t_{0}}^{n})\right):\nabla(E_{S}h)\,\dx y\\
 & +\int_{\OS}\left(\ddot{\tilde{\xi}}_{t_{0}}(\cdot+t_{n})-\dot{\varphi}^{n}(\cdot+t_{0})+\dot{\varphi}^{n}(\cdot)\right)\cdot E_{S}h\,\dx y\\
 & \leq C\left(\left\Vert \psi_{t_{0}}^{n}\right\Vert _{\C^{0}(\H1^{1}(\OS))}+\Vert\psi^{n}\Vert_{\C^{1}(\Lp^{2}(\OS))}\right)\Vert E_{S}h\Vert_{\H1^{1}(\OS)}.
\end{align*}
Thus, $\lim_{n\to\infty}\sup_{t_{0}\in(0,t_{0}^{*}]}\left\Vert \Sigma(\psi_{t_{0}}^{n})n\right\Vert _{\C^{0}(\H1^{-1/2}(\dOS))}=0$
follows from (\ref{eq:xi_conv_xi_etan}) and (\ref{eq:xi_phi_dotxi}).
We conclude from (\ref{eq:xi_conv_etan_eta}) and (\ref{eq:xi_conv_etan_eta_neumann})
that 
\begin{align}
\lim_{n\to\infty}\sup_{t_{0}\in(0,t_{0}^{*}]}\left\Vert \tilde{\xi}_{t_{0}}(\cdot+t_{n})-\eta_{t_{0}}\right\Vert _{U_{T/2}}=0\label{eq:xi_conv_eta}
\end{align}
and that 
\begin{align}
\lim_{n\to\infty}\sup_{t_{0}\in(0,t_{0}^{*}]}\left\Vert \Sigma(\tilde{\xi}_{t_{0}}(\cdot+t_{n}))n-\Sigma(\eta_{t_{0}})n\right\Vert _{\Lp^{2}(\H1^{-1/2}(\dOS))}=0.\label{eq:xi_conv_eta_neumann}
\end{align}
It remains to show that $\eta_{t_{0}}\in A_{T/2}$,$\ $i.e.$\ $that
$\Sigma(\eta_{t_{0}})n=\tilde{q}n$ for some scalar function $\tilde{q}.$
By Corollary \ref{cor:convofpressure},
\begin{align}
\lim_{t\to\infty}\Vert\Sigma(\xi(\cdot+t)n-q(\cdot+t)n\Vert_{\Lp^{2}(\H1^{-1/2}(\dOS))}=0.\label{eq:xi_conv_Sigma_q}
\end{align}
Combined with (\ref{eq:xi_conv_eta_neumann}), this implies that 
\begin{align*}
\lim_{n\to\infty}\sup_{t_{0}\in(0,t_{0}^{*}]}\left\Vert \tilde{q}_{t_{0}}(\cdot+t_{n})n-\Sigma(\eta_{t_{0}})n\right\Vert _{\Lp^{2}(\H1^{-1/2}(\dOS))}=0.
\end{align*}
Together with (\ref{eq:xi_conv_eta}) and (\ref{eq:xi_conv_eta_neumann}),
this concludes the proof of (\ref{eq:xi_thm_approx_A}). 
\end{proof}

\section{Proof of the Main Result}

\label{SecProof}

The proof of the main result Theorem \ref{thm:main result xi} is
divided into six steps. \\
 \textbf{Step 1: Convergence in rate on good domains}
\begin{cor}
\label{cor:xi_derivatives_gooddomain} Let $\OS$ be a \emph{good}
domain. Then 
\begin{align*}
\lim_{t\to\infty}\Vert\dot{\xi}(t)\Vert_{\Lp^{2}(\OS)}=0\hspace{0.5cm}\text{ and }\hspace{0.5cm}\lim_{t\to\infty}\Vert\ddot{\xi}(t)\Vert_{\H1^{-1}(\OS)}=0.
\end{align*}
Moreover, 
\begin{align}
\lim_{t\to\infty}\Vert\ddot{\xi}(t+\cdot)\Vert_{\Lp^{2}(0,T/2;\H1^{1}(\OS)^{*})}=0.\label{eq:xi_ddotxi_L2H1*}
\end{align}
\end{cor}

\begin{proof}
We start by proving the convergence of $\dot{\xi}$. Since $\omega_{T}=\{(0,0)\}$
if $\OS$ is a good domain, Theorem \ref{thm:xi_approx_A} shows that
\begin{align}
\lim_{t\to\infty}\sup_{t_{0}\in(0,t_{0}^{*}]}\Vert\xi(\cdot+t_{0}+t)-\xi(\cdot+t)\Vert_{U_{T/2}}=0\label{eq:xi_conv_difference_xi}
\end{align}
and 
\[
\lim_{t\to\infty}\sup_{t_{0}\in(0,t_{0}^{*}]}\Vert\Sigma(\xi(\cdot+t_{0}+t)-\xi(\cdot+t))n\Vert_{\Lp^{2}(\H1^{1/2}(\dOS))}=0.
\]
Here, the uniformity for small $t_{0}$ is essential. By the fundamental
theorem of calculus, 
\begin{align*}
\left\Vert \frac{\xi(t+t_{0}^{*})-\xi(t)}{t_{0}^{*}}-\dot{\xi}(t)\right\Vert _{\Lp^{2}(\OS)} & =\left\Vert \frac{1}{t_{0}^{*}}\int_{0}^{t_{0}^{*}}\dot{\xi}(t+s)-\dot{\xi}(t)\,\dx s\right\Vert _{\Lp^{2}(\OS)}\\
 & \leq\sup_{t_{0}\in(0,t_{0}^{*}]}\Vert\tilde{\xi}_{t_{0}}(\cdot+t)\Vert_{\C^{1}(\Lp^{2}(\OS))}\to0\hspace{0.2cm}\text{ as }t\to\infty.
\end{align*}
In combination with 
\begin{align*}
\left\Vert \xi(t+t_{0}^{*})-\xi(t)\right\Vert _{\Lp^{2}(\OS)}\to0\;\text{ for }t\to\infty
\end{align*}
due to \eqref{eq:xi_conv_difference_xi}, we obtain $\lim_{t\to\infty}\Vert\dot{\xi}(t)\Vert_{\Lp^{2}(\OS)}=0$.
The convergence $\lim_{t\to\infty}\Vert\ddot{\xi}(t)\Vert_{\H1^{-1}(\OS)}=0$
and (\ref{eq:xi_ddotxi_L2H1*}) follow analogously from \eqref{eq:xi_conv_difference_xi}. 
\end{proof}
\textbf{Step 2: Construction of a limit for $\tilde{\xi}_{t_{0}}$
and convergence in energy norm. }

Given $(u,p,\xi)$ a global solution to \eqref{eq:nonlin_system}
corresponding to compatible initial data $(u_{0},\xi_{0},\xi_{1})$,
for every sequence $t_{n}\to\infty$, define the time-shifted differences
by 
\begin{align*}
\tilde{u}_{t_{0},n}(t) & :=u(t+t_{0}+t_{n})-u(t+t_{n})=\tilde{u}_{t_{0}}(t+t_{n}),\\
\tilde{p}_{t_{0},n}(t) & :=p(t+t_{0}+t_{n})-p(t+t_{n})=\tilde{p}_{t_{0}}(t+t_{n}),\\
\tilde{\xi}_{t_{0},n}(t) & :=\xi(t+t_{0}+t_{n})-\xi(t+t_{n})=\tilde{\xi}_{t_{0}}(t+t_{n}).
\end{align*}
Let now $(t_{n})$ be the subsequence of $(t_{n})$ constructed in
the proof of Theorem \ref{thm:xi_approx_A} and let $\varphi\in B_{T}$
be such that (\ref{eq:xi_phi_conv}) holds. Moreover, choose $t_{0}\in(0,T/2)\setminus\{2\pi/\sqrt{\mu_{i}}:i\in I\}$
and let $(\eta_{t_{0}},q_{t_{0}})\in\omega_{T/2}$ be the pair such
that (\ref{eq:xi_conv_eta}) and (\ref{eq:xi_conv_eta_neumann}) hold.
Recall that
\begin{align}
\eta_{t_{0}}(t)=\int_{t}^{t+t_{0}}\varphi(s)\,\dx s.\label{eq:xi_def_eta_phi}
\end{align}
Due to the structure of $\omega=\omega_{T}$, we can extend its elements
and hence $(\eta_{t_{0}},q_{t_{0}})$ globally in time. Then we can
associate the energy 
\[
E_{t_{0},n}(t)=\Vert\tilde{u}_{t_{0},n}(t)\Vert_{\Lp^{2}(\OF)}^{2}+\Vert\dot{\tilde{\xi}}_{t_{0},n}(t)-\dot{\eta}_{t_{0}}(t)\Vert_{\Lp^{2}(\OS)}^{2}+\int_{\OS}\Sigma(\tilde{\xi}_{t_{0},n}(t)-\eta_{t_{0}}(t)):\varepsilon(\tilde{\xi}_{t_{0},n}(t)-\eta_{t_{0}}(t))\,\dx y
\]
to the triple $(\tilde{u}_{t_{0},n},\tilde{p}_{t_{0},n}-q_{t_{0}},\tilde{\xi}_{t_{0},n}-\eta_{t_{0}}).$
Now let $\varepsilon>0$. Because of (\ref{eq:xi_conv_eta}) and $\Vert u(t)\Vert_{\H1^{1}(\OF)}\to0$
for $t\to\infty$, we find some $n_{1}\in\mathbb{N}$ such that $E_{t_{0},n}(0)<\frac{\varepsilon}{2}$
for all $n\geq n_{1}$ and all $t_{0}\in(0,t_{0}^{*}]$. Since 
\begin{align*}
 & E_{t_{0},n}(t)+\int_{0}^{t}\Vert\varepsilon(\tilde{u}_{t_{0},n})\Vert_{\Lp^{2}(\OF)}^{2}\,\dx s\\
=\, & E_{t_{0},n}(0)-2\int_{0}^{t}\int_{\OF}((\tilde{u}_{t_{0},n}(s)\cdot\nabla)u(s+t_{0}+t_{n_{k}})+(u(s+t_{n})\cdot\nabla)\tilde{u}_{t_{0},n}(s))\cdot\tilde{u}_{t_{0},n}(s)\,\dx y\dx s\\
\leq\, & E_{t_{0},n}(0)+2C\int_{0}^{t}\left(\Vert u(s+t_{0}+t_{n})\Vert_{\H1^{1}(\OS)}+\Vert u(s+t_{n}\right)\Vert_{\H1^{1}(\OS)})\Vert\varepsilon(\tilde{u}_{t_{0},n})(s)\Vert_{\Lp^{2}(\OS)}^{2}\,\dx s,
\end{align*}
we can use $\Vert u(t)\Vert_{\H1^{1}(\OF)}\to0$ and follow the proof
of \eqref{eq:bound_E} in Theorem \ref{thm:globale_existenz} to find
some $n_{2}\geq n_{1}$ such that 
\begin{align}
E_{t_{0},n}(t)\leq E_{t_{0},n}(0)<\frac{\varepsilon}{2}\hspace{0.2cm}\text{ for all }t\geq0\label{eq:xi_bound_Ek}
\end{align}
and all $n\geq n_{2}$. Now for any $n\geq n_{2}$, Korn's inequality
implies that 
\begin{align}
 & \Vert\eta_{t_{0}}(t+t_{n+1}-t_{n})-\eta_{t_{0}}(t)\Vert_{\H1^{1}(\OS)}\nonumber \\
 & \leq C\Vert\Sigma(\eta_{t_{0}}(t+t_{n+1}-t_{n})-\eta_{t_{0}}(t))\Vert_{\Lp^{2}(\OS)}\nonumber \\
 & \leq C\big(\Vert\Sigma(\eta_{t_{0}}(t+t_{n+1}-t_{n})-\tilde{\xi}_{t_{0},n+1}(t))\Vert_{\Lp^{2}(\OS)}+\Vert\Sigma(\tilde{\xi}_{t_{0},n+1}-\eta_{t_{0}})(t))\Vert_{\Lp^{2}(\OS)}\big)\nonumber \\
 & =C\big(\Vert\Sigma(\eta_{t_{0}}-\tilde{\xi}_{t_{0},n})(t+t_{n+1}-t_{n}))\Vert_{\Lp^{2}(\OS)}+\Vert\Sigma(\tilde{\xi}_{t_{0},n+1}-\eta_{t_{0}})(t))\Vert_{\Lp^{2}(\OS)}\big)<C\varepsilon\label{eq:xi_period_eta}
\end{align}
for all $t\geq0$. At the same time, the coefficients 
\begin{align}
\eta_{t_{0},i}(t):=\int_{\OS}\eta_{t_{0}}(t)\cdot\psi_{i}\,\dx y=a_{t_{0},i}\sin(\sqrt{\mu_{i}}t)+b_{t_{0},i}\cos(\sqrt{\mu_{i}}t)\label{eq:etat0i}
\end{align}
of $\eta_{t_{0}}$ are either constantly zero or $P_{i}$-periodic
for $P_{i}:=\frac{2\pi}{\sqrt{\mu_{i}}}$ and $i\in I$. We denote
by $I_{\eta_{t_{0}}}$ the set of all $i\in I$ where $\eta_{t_{0},i}$
does not vanish and set $I_{\eta}:=\cup_{t_{0}\in(0,t_{0}^{*}]}I_{\eta_{t_{0}}}$.
If $I_{\eta}=\emptyset$, then $\eta_{t_{0}}=0$ and $q_{t_{0}}=0$
for all $t_{0}\in(0,t_{0}^{*}]$ and we can skip to Step 5 in the
proof, with $\eta^{*}=0$. For the moment, we assume that $I_{\eta}\neq\emptyset$.
The smallness in (\ref{eq:xi_period_eta}) implies that 
\begin{align*}
 & \lim_{n\to\infty}\eta_{t_{0},i}(t+(t_{n+1}-t_{n})\,\text{mod}\,P_{i}))-\eta_{t_{0},i}(t)=\lim_{n\to\infty}\eta_{t_{0},i}(t+(t_{n+1}-t_{n}))-\eta_{t_{0},i}(t)=0
\end{align*}
for all $t\geq0$. Using the decomposition into sine and cosine functions
of period $\leq P_{i}$ in (\ref{eq:etat0i}) and arguing by contradiction,
we conclude that$\lim_{n\to\infty}(t_{n+1}-t_{n})\,\text{mod}\,P_{i}=0$for
all $i\in I_{\eta}$ and hence $\lim_{n\to\infty}t_{n}\,\text{mod}\,P_{i}=t^{i}$for
some $t^{i}\in[0,P_{i})$. We thus define 
\begin{align*}
\hat{\eta}_{t_{0}}(t,y):=\sum_{i\in I_{\eta}}\eta_{t_{0},i}(t-t^{i})\psi_{i}(y).
\end{align*}
Applying the angle sum identities and the characterization in \eqref{eq:charomega},
we see that $\hat{\eta}_{t_{0}}\in A$. It remains to show that $\tilde{\xi}_{t_{0}}\to\hat{\eta}_{t_{0}}$
in a suitable sense. By definition, for every $t\geq0$, 
\[
\eta_{t_{0}}(t-t_{n_{k}}),\dot{\eta}_{t_{0}}(t-t_{n_{k}})\overset{k\to\infty}{\to}\hat{\eta}_{t_{0}}(t),\dot{\hat{\eta}}_{t_{0}}(t)\in\Lp^{2}(\OS).
\]
Let therefore $n_{2}\geq n_{1}$ be such that for all $t\geq0$ and
all $n\geq n_{2}$, 
\[
\Vert\dot{\eta}_{t_{0}}(t-t_{n})-\dot{\hat{\eta}}_{t_{0}}(t)\Vert_{\Lp^{2}(\OS)}^{2}+\int_{\OS}\Sigma(\eta_{t_{0}}(t-t_{n})-\hat{\eta}_{t_{0}}(t)):\varepsilon(\eta_{t_{0}}(t-t_{n})-\hat{\eta}_{t_{0}}(t))\,\dx y<\frac{\varepsilon}{2}.
\]
By shifting the functions appearing in (\ref{eq:xi_bound_Ek}) by
$t_{n}$, we obtain that for all $t\geq t_{n_{2}}$,
\begin{align}
 & \Vert(\dot{\tilde{\xi}}_{t_{0}}-\dot{\hat{\eta}}_{t_{0}})(t)\Vert_{\Lp^{2}(\OS)}^{2}+\int_{\OS}\Sigma(\tilde{\xi}_{t_{0}}-\hat{\eta}_{t_{0}}):\varepsilon(\tilde{\xi}_{t_{0}}-\hat{\eta}_{t_{0}})(t)\,\dx y\nonumber \\
 & \leq\Vert\dot{\tilde{\xi}}_{t_{0}}(t)-\dot{\eta}_{t_{0}}(t-t_{n_{2}})\Vert_{\Lp^{2}(\OS)}^{2}+\Vert\dot{\eta}_{t_{0}}(t-t_{n_{2}})-\dot{\hat{\eta}}_{t_{0}}(t)\Vert_{\Lp^{2}(\OS)}^{2}\nonumber \\
 & +2\int_{\OS}\Sigma(\tilde{\xi}_{t_{0}}(t)-\eta_{t_{0}}(t-t_{n_{2}})):\varepsilon(\tilde{\xi}_{t_{0}}(t)-\eta_{t_{0}}(t-t_{n_{2}}))\,\dx y\nonumber \\
 & +2\int_{\OS}\Sigma(\eta_{t_{0}}(t-t_{n_{2}})-\hat{\eta}_{t_{0}}(t)):\varepsilon(\eta_{t_{0}}(t-t_{n_{2}})-\hat{\eta}_{t_{0}}(t))\,\dx y<2\varepsilon.\label{eq:xi_tnk_eta}
\end{align}
\textbf{Step 3: Convergence of $\tilde{\xi}_{t_{0}}$ in $U_{T}$. }

Here, we improve the convergence in energy in \eqref{eq:xi_tnk_eta}
to stronger norms. By Theorem \ref{thm:xi_approx_A}, for $t\geq0$
sufficiently large, there exists a pair $(\eta_{t_{0}}^{t},q_{t_{0}}^{t})\in\omega_{T}$
such that 
\[
\Vert\tilde{\xi}_{t_{0}}(\cdot+t)-\eta_{t_{0}}^{t}\Vert_{U_{T/2}}^{2}<\varepsilon
\]
and 
\begin{align}
\Vert\Sigma(\tilde{\xi}_{t_{0}}(\cdot+t)-\eta_{t_{0}}^{t})n\Vert_{\Lp^{2}(\H1^{-1/2}(\dOS))}^{2}<\varepsilon\label{eq:xi_dist_etat_sigma}
\end{align}
for all $t_{0}\in(0,t_{0}^{*}]$. Using $\eta_{t_{0}}^{t}-\hat{\eta}_{t_{0}}\vert_{\dOS}=0$,
Korn's inequality and (\ref{eq:xi_tnk_eta}), we obtain 
\[
\Vert(\eta_{t_{0}}^{t}-\hat{\eta}_{t_{0}})(t)\Vert_{\H1^{1}(\OS)}^{2}\leq C\bigg(\int_{\OS}\Sigma(\eta_{t_{0}}^{t}-\tilde{\xi}_{t_{0}}):\varepsilon(\eta_{t_{0}}^{t}-\tilde{\xi}_{t_{0}})\,\dx y+\int_{\OS}\Sigma(\tilde{\xi}_{t_{0}}-\hat{\eta}_{t_{0}}):\varepsilon(\tilde{\xi}_{t_{0}}-\hat{\eta}_{t_{0}})\,\dx y\bigg)(t)<C\varepsilon.
\]
Hence 
\[
\Vert(\tilde{\xi}_{t_{0}}-\hat{\eta}_{t_{0}})(t)\Vert_{\H1^{1}(\OS)}^{2}\leq\Vert(\tilde{\xi}_{t_{0}}-\eta_{t_{0}}^{t})(t)\Vert_{\H1^{1}(\OS)}^{2}+\Vert(\eta_{t_{0}}^{t}-\hat{\eta}_{t_{0}})(t)\Vert_{\H1^{1}(\OS)}^{2}<C\varepsilon.
\]
Together with (\ref{eq:xi_tnk_eta}), we conclude that 
\begin{align*}
\lim_{t\to\infty}\Vert(\tilde{\xi}_{t_{0}}-\hat{\eta}_{t_{0}})(\cdot+t)\Vert_{\C^{0}(\H1^{1}(\OS))\cap\C^{1}(\Lp^{2}(\OS))}=0.
\end{align*}
Since $\xi(t+t_{0})-\xi(t)-\hat{\eta}_{t_{0}}(t)$ solves a weak Dirichlet
problem, we obtain for all $f\in\H1_{0}^{1}(\OS)$ that 
\begin{align*}
\langle\ddot{\xi}(t+t_{0})-\ddot{\xi}(t)-\ddot{\hat{\eta}}_{t_{0}}(t),f\rangle_{\H1^{-1}(\OS),\H1_{0}^{1}(\OS)} & =-\int_{\OS}\Sigma(\xi(t+t_{0})-\xi(t)-\hat{\eta}_{t_{0}}(t)):\varepsilon(f)\,\dx y\\
 & \leq\Vert\xi(t+t_{0})-\xi(t)-\hat{\eta}_{t_{0}}(t)\Vert_{\H1^{1}(\OS)}\Vert f\Vert_{\H1^{1}(\OS)}\overset{t\to\infty}{\to}0.
\end{align*}
Hence even 
\begin{align}
\lim_{t\to\infty}\Vert\tilde{\xi}_{t_{0}}(\cdot+t)-\hat{\eta}_{t_{0}}(\cdot+t)\Vert_{\C^{0}(\H1^{1}(\OS))\cap\C^{1}(\Lp^{2}(\OS))\cap\C^{2}(\H1^{-1}(\OS))}=0.\label{eq:limit_etahat}
\end{align}
Moreover, the hidden regularity result in Theorem \ref{thm:lame_LLT}
applied to $\eta_{t_{0}}^{t}-\hat{\eta}_{t_{0}}(\cdot+t)$ implies
that 
\begin{align*}
\Vert\Sigma(\eta_{t_{0}}^{t}-\hat{\eta}_{t_{0}}(\cdot+t))n\Vert_{\Lp^{2}((0,T/2)\times\dOS)}\leq C\Vert\eta_{t_{0}}^{t}-\hat{\eta}_{t_{0}}(\cdot+t)\Vert_{\C^{0}(\H1^{1}(\OS))\cap\C^{1}(\Lp^{2}(\OS))}<C\varepsilon.
\end{align*}

Together with (\ref{eq:xi_dist_etat_sigma}), we get
\begin{align}
\lim_{t\to\infty}\Vert\Sigma(\tilde{\xi}_{t_{0}}(\cdot+t)-\hat{\eta}_{t_{0}}(\cdot+t))n\Vert_{\Lp^{2}(0,T/2;\H1^{-1/2}(\OS))}=0.\label{eq:limit_etahat_Sigma}
\end{align}
Then using the weak Neumann formulation yields
\begin{align}
\lim_{t\to\infty}\Vert\ddot{\tilde{\xi}}_{t_{0}}(\cdot+t)-\ddot{\hat{\eta}}_{t_{0}}(\cdot+t)\Vert_{\Lp^{2}(0,T/2;\H1^{1}(\OS)^{*})}=0.\label{eq:limit_etahat_H1*}
\end{align}
\textbf{Step 4: Construction of a limit for $\xi$.} \\
 Now we want to get rid of the difference $\tilde{\xi}_{t_{0}}$ and
study $\xi$ instead. For this, first recall (\ref{eq:xi_def_eta_phi})
and note that therefore
\begin{align*}
\varphi(t,y)=\sum_{i\in I_{\eta}}\varphi_{i}(t)\psi_{i}(y)
\end{align*}
with coefficient functions $\varphi_{i}\in\C^{0}(\mathbb{R})$. This
is due to the fact that we can choose any $(0,T/2)\ni t_{0}\neq P_{i}$
for all $i\in\mathbb{N}$ and obtain 
\begin{align*}
\varphi(t+t_{0})-\varphi(t)=\dot{\eta}_{t_{0}}(t)\perp\psi_{i}\hspace{0.2cm}\text{ for all }i\notin I_{\eta},
\end{align*}
so that also $\varphi\perp\psi_{i}$ for all $i\notin I_{\eta}$.
Next, we show that $\varphi_{i}$ is of the form 
\begin{align}
\varphi_{i}(t)=\tilde{a}_{i}\sin(\sqrt{\mu_{i}}t)+\tilde{b}_{i}\cos(\sqrt{\mu_{i}}t).\label{eq:xi_structure_phii}
\end{align}
Since $\eta_{t_{0},i}$ is $P_{i}$-periodic, 
\begin{align*}
0=\eta_{t_{0},i}(t+P_{i})-\eta_{t_{0},i}(t)=\int_{t}^{t+t_{0}}\varphi_{i}(s+P_{i})-\varphi_{i}(s)\,\dx s
\end{align*}
holds for all $t_{0}\in(0,T/2)$. Since $\varphi_{i}$ is continuous,
it must thus also be $P_{i}$-periodic. Hence we can consider its
Fourier series decomposition 
\begin{align*}
\varphi_{i}(t)=\sum_{k\in\mathbb{N}}\tilde{a}_{i}^{k}\sin(k\sqrt{\mu_{i}}t)+\tilde{b}_{i}^{k}\cos(k\sqrt{\mu_{i}}t).
\end{align*}
Due to $\varphi_{i}(t+t_{0})-\varphi_{i}(t)=\dot{\eta}_{t_{0},i}(t)$,
$\tilde{a}_{i}^{k}=\tilde{b}_{i}^{k}=0$ for $k\geq2$. Moreover,
$\tilde{b}_{i}^{0}=0$ follows from 
\begin{align*}
\eta_{t_{0},i}(t)=\int_{t}^{t+t_{0}}\varphi_{i}(s)\,\dx s=\int_{t}^{t+t_{0}}\tilde{a}_{i}^{1}\sin(\sqrt{\mu_{i}}t)+\tilde{b}_{i}^{1}\cos(\sqrt{\mu_{i}}t)\,\dx s+t_{0}\tilde{b}_{i}^{0}.
\end{align*}
This implies (\ref{eq:xi_structure_phii}) with $\tilde{a}_{i}:=\tilde{a}_{i}^{1}$,
$\tilde{b}_{i}:=\tilde{b}_{i}^{1}$. Next, similarly as for $\hat{\eta}_{t_{0}}$,
we can define 
\begin{align*}
\hat{\varphi}(t,y):=\sum_{i\in I_{\eta}}\hat{\varphi}_{i}(t)\psi_{i}(y):=\sum_{i\in I_{\eta}}\varphi_{i}(t-t_{i})\psi_{i}(y)
\end{align*}
satisfying $\hat{\varphi}\in A$. %
{} Now we are in a position to construct the limit pressure wave $\eta^{*}$
for the displacement $\xi$: Integration yields 
\[
\int_{t_{i}}^{t}\hat{\varphi}_{i}(s)\,\dx s=-\frac{\tilde{a}_{i}}{\sqrt{\mu_{i}}}\cos(\sqrt{\mu_{i}}(t-t_{i}))+\frac{\tilde{b}_{i}}{\sqrt{\mu_{i}}}\sin(\sqrt{\mu_{i}}(t-t_{i}))+\frac{\tilde{a}_{i}}{\sqrt{\mu_{i}}}.
\]
Set $\eta_{0}^{*}:=\sum_{i\in I_{\eta}}-\frac{\tilde{a}_{i}}{\sqrt{\mu_{i}}}\psi_{i}$and
\[
\eta^{*}(t,y):=\eta_{0}^{*}+\sum_{i\in I_{\eta}}\int_{t_{i}}^{t}\hat{\varphi}_{i}(s)\psi_{i}(y)\,\dx s=\sum_{i\in I_{\eta}}\left(-\frac{\tilde{a}_{i}}{\sqrt{\mu_{i}}}\cos(\sqrt{\mu_{i}}(t-t_{i}))+\frac{\tilde{b}_{i}}{\sqrt{\mu_{i}}}\sin(\sqrt{\mu_{i}}(t-t_{i}))\right)\psi_{i}.
\]
Again, the angle sum theorems show that $\eta^{*}\in A.$ \\
 \textbf{Step 5: Proof of (\ref{eq:xi_conv_deriv_baddomain-1})} \\
 First, we observe that 
\begin{align*}
\eta^{*}(t+t_{0})-\eta^{*}(t)=\int_{t}^{t+t_{0}}\hat{\varphi}(s)\,\dx s=\hat{\eta}_{t_{0}}(t)
\end{align*}
for all $t\in\mathbb{R}$ and all $t_{0}\in(0,t_{0}^{*}]$. As the
convergences in (\ref{eq:limit_etahat}), (\ref{eq:limit_etahat_H1*})
and (\ref{eq:limit_etahat_Sigma}) can be obtained uniformly in $t_{0}\in(0,t_{0}^{*}]$
for the corresponding $\hat{\eta}_{t_{0}}=\eta^{*}(\cdot+t_{0})-\eta^{*}(\cdot)$,
we conclude that 
\begin{align*}
\lim_{t\to\infty}\sup_{t_{0}\in(0,t_{0}^{*}]}\Vert\xi(t+t_{0})-\xi(t)-\eta^{*}(t+t_{0})+\eta^{*}(t)\Vert_{\C^{0}(\H1^{1}(\OS))\cap\C^{1}(\Lp^{2}(\OS))\cap\C^{2}(\H1^{-1}(\OS))}=0,
\end{align*}
\begin{align*}
\lim_{t\to\infty}\sup_{t_{0}\in(0,t_{0}^{*}]}\Vert\ddot{\xi}(\cdot+t+t_{0})-\ddot{\xi}(\cdot+t)-\ddot{\eta}^{*}(t+t_{0})+\ddot{\eta}^{*}(t)\Vert_{\Lp^{2}(0,T/2;\H1^{1}(\OS)^{*})}=0\hspace{0.2cm}
\end{align*}
and 
\begin{align*}
\lim_{t\to\infty}\sup_{t_{0}\in(0,t_{0}^{*}]}\Vert\Sigma(\xi(\cdot+t+t_{0})-\xi(\cdot+t)-\eta^{*}(t+t_{0})+\eta^{*}(t))n\Vert_{\Lp^{2}(0,T/2;\H1^{-1/2}(\OS))}=0.
\end{align*}
Consequently, we can apply the techniques of the proof of Corollary
\ref{cor:xi_derivatives_gooddomain} to $\xi-\eta^{*}$ to show (\ref{eq:xi_conv_deriv_baddomain-1})
and 
\begin{align}
\lim_{t\to\infty}\Vert\ddot{\xi}(t+\cdot)-\ddot{\eta}^{*}(\cdot+t)\Vert_{\Lp^{2}(0,T/2;\H1^{1}(\OS)^{*})}=0.\label{eq:xi_eta_H1*}
\end{align}
\textbf{Step 6: Proof of (\ref{eq:xi_conv_baddomain-1})}\\
 Define $\oxi=\xi-\eta^{*}$ and $\oq=q-q^{*}$ and recall that $r=P_{\mathcal{R}}(\oxi)$.
Set 
\[
\xi_{N}(t):=\oq(t)\varphi_{N}+r(t).
\]
Then for every $t\geq0$, $\xi_{N}(t)$ solves the stationary Neumann
system (\ref{eq:xi_stat_system-1}) with $q=\oq(t)$. We use $P_{\mathcal{R}}((\oxi-\xi_{N})(t))=0$,
Lemma \ref{lem:xi_korn-1} and Young's inequality to estimate 
\begin{align*}
 & \int_{\OS}\Sigma(\oxi-\xi_{N}):\varepsilon(\oxi-\xi_{N})(t)\,\dx y\\
=\, & \int_{\OS}\ddot{\oxi}\cdot(\oxi-\xi_{N})(t)\,\dx y+\int_{\dOS}\Sigma(\oxi-\xi_{N})n\cdot(\oxi-\xi_{N})(t)\,\dx S(y)\\
\leq\, & C\left(\Vert\ddot{\oxi}(t)\Vert_{\H1^{1}(\OS)^{*}}\Vert\varepsilon(\oxi-\xi_{N})(t))\Vert_{\Lp^{2}(\OS)}+\Vert(\Sigma(\oxi)n-\oq n)(t)\Vert_{\H1^{-1/2}(\dOS)}\Vert\varepsilon(\oxi-\xi_{N})(t)\Vert_{\Lp^{2}(\OS)}\right)\\
\leq\, & C\left(\Vert\ddot{\oxi}(t)\Vert_{\H1^{1}(\OS)^{*}}^{2}+\Vert(\Sigma(\oxi)n-\oq)n(t)\Vert_{\H1^{-1/2}(\dOS)}^{2}\right)+\frac{1}{2}\int_{\OS}\Sigma(\oxi-\xi_{N}):\varepsilon(\oxi-\xi_{N})(t)\,\dx y
\end{align*}
for any $t\geq0$. Together with (\ref{eq:xi_ddotxi_L2H1*}), (\ref{eq:xi_conv_Sigma_q})
and (\ref{eq:xi_eta_H1*}), this implies that 
\begin{align*}
\Vert(\oxi-\xi_{N})(\cdot+t)\Vert_{\Lp^{2}(\H1^{1}(\OS))}^{2} & \leq C\int_{0}^{T/2}\int_{\OS}\Sigma(\oxi-\xi_{N}):\varepsilon(\oxi-\xi_{N})(s+t)\,\dx y\dx s\\
 & \leq C\left(\Vert\ddot{\oxi}(\cdot+t)\Vert_{\Lp^{2}(\H1^{1}(\OS)^{*})}^{2}+\Vert(\Sigma(\oxi)n-\oq n)(\cdot+t)\Vert_{\Lp^{2}(\H1^{-1/2}(\dOS))}^{2}\right)\overset{t\to\infty}{\to}0.
\end{align*}
Define the intervals $I_{n}:=[n\frac{T}{2},(n+1)\frac{T}{2})$. Then
for all $n\in\mathbb{N}$, there exists $t_{n}\in I_{n}$ such that
\begin{align*}
\Vert(\oxi-\xi_{N})(t_{n})\Vert_{\H1^{1}(\OS)}^{2}\leq\frac{2}{T}\int_{0}^{T/2}\Vert(\oxi-\xi_{N})(s+nT/2)\Vert_{\H1^{1}(\OS)}^{2}\,\dx s\to0\text{ for }n\to\infty.
\end{align*}
Now define $\bar{\xi}_{N}$ through $\bar{\xi}_{N}(t):=\xi_{N}(t_{n})$
if $t\in I_{n-1}$. Then by (\ref{eq:xi_conv_difference_xi}), 
\begin{align*}
\Vert(\oxi-\bar{\xi}_{N})(t)\Vert_{\H1^{1}(\OS)} & \leq\Vert\oxi(t)-\oxi(t+(tn-t)/2)\Vert_{\H1^{1}(\OS)}+\Vert\oxi(t+(t_{n}-t)/2)-\oxi(t_{n})\Vert_{\H1^{1}(\OS)}\\
 & \quad+\Vert(\oxi-\xi_{N})(t_{n})\Vert_{\H1^{1}(\OS)}\to0\text{ as }t\to\infty.
\end{align*}
Next, note that the definition of $\varphi_{N}^{0}$ yields $\int_{\OS}\Sigma(\oxi_{0}-\varphi_{N}^{0}):\varepsilon(\varphi_{N})\,\dx y=0.$Due
to $\mathcal{R}=\ker(\varepsilon)$ and $\mathcal{E}=\Span\{\varphi_{N}\}+\mathcal{R}$,
even for any $\tilde{\varphi}\in\mathcal{E}$, 
\begin{align}
\int_{\OS}\Sigma(\oxi_{0}-\varphi_{N}^{0}):\varepsilon(\tilde{\varphi})\,\dx y=0.\label{eq:xi_xiN0-1}
\end{align}
We want to show that also 
\begin{align}
\int_{\OS}\Sigma(\oxi(t)-\varphi_{N}^{0}):\varepsilon(\tilde{\varphi})\,\dx y=0\hspace{0.2cm}\text{ for all }t\geq0.\label{eq:xi_orthogonal-1}
\end{align}
Subtracting any $\hat{\varphi}\in\mathcal{E}$ from $\oxi$ does not
influence the fluid part of the coupled system because the elements
of $\mathcal{E}$ solve the stationary Neumann problem (\ref{eq:xi_stat_system-1})
for some $\hat{q}\in\mathbb{R}$. For $\varphi_{N}^{0}\in\Span\{\varphi_{N}\}\subset\mathcal{E}$
defined in (\ref{eq:def_phiN0}), the corresponding pressure term
is given by 
\[
q_{N}^{0}:=\frac{\int_{\OS}\Sigma(\oxi_{0}):\varepsilon(\varphi_{N})\,\dx y}{\int_{\OS}\Sigma(\varphi_{N}):\varepsilon(\varphi_{N})\,\dx y}.
\]
Consequently $(u,p-q^{*},\oxi)$, $(u,p-q^{*}-q_{N}^{0},\oxi-\varphi_{N}^{0})$
and also $(u,p-q^{*}-q_{N}^{0}-\hat{q},\oxi-\varphi_{N}^{0}-\hat{\varphi})$
are global solutions for the same fluid motion $u$. Hence the energies
\begin{align*}
E_{\oxi-\varphi_{N}^{0}}(t) & :=\Vert u(t)\Vert_{\Lp^{2}(\OF)}^{2}+\Vert\dot{\oxi}(t)\Vert_{\Lp^{2}(\OS)}^{2}+\int_{\OS}\Sigma(\oxi(t)-\varphi_{N}^{0}):\varepsilon(\oxi(t)-\varphi_{N}^{0})\,\dx y,\text{ and}\\
E_{\oxi-\varphi_{N}^{0}-\bar{\varphi}}(t) & :=\Vert u(t)\Vert_{\Lp^{2}(\OF)}^{2}+\Vert\dot{\oxi}(t)\Vert_{\Lp^{2}(\OS)}^{2}+\int_{\OS}\Sigma(\oxi(t)-\varphi_{N}^{0}-\hat{\varphi}):\varepsilon(\oxi(t)-\varphi_{N}^{0}-\hat{\varphi})\,\dx y
\end{align*}
satisf
\begin{align*}
E_{\oxi^{-}\varphi_{N}^{0}-\hat{\varphi}}(t)-E_{\oxi-\varphi_{N}^{0}}(t)=E_{\oxi-\varphi_{N}^{0}-\hat{\varphi}}(0)-E_{\oxi-\varphi_{N}^{0}}(0)
\end{align*}
for all $t\geq0$. Since it follows from (\ref{eq:xi_xiN0-1}) that
\[
E_{\oxi-\varphi_{N}^{0}-\hat{\varphi}}(0)-E_{\oxi-\varphi_{N}^{0}}(0)=\int_{\OS}\Sigma(\hat{\varphi}):\varepsilon(\hat{\varphi})\,\dx y,
\]
we obtain 
\begin{align*}
\int_{\OS}\Sigma(\hat{\varphi}):\varepsilon(\hat{\varphi})\,\dx y & =E_{\oxi-\varphi_{N}^{0}-\hat{\varphi}}(0)-E_{\oxi-\varphi_{N}^{0}}(0)=E_{\oxi-\varphi_{N}^{0}-\hat{\varphi}}(t)-E_{\oxi-\varphi_{N}^{0}}(t)\\
 & =-2\int_{\OS}\Sigma(\oxi(t)-\varphi_{N}^{0}):\varepsilon(\hat{\varphi})\,\dx y+\int_{\OS}\Sigma(\hat{\varphi}):\varepsilon(\hat{\varphi})\,\dx y
\end{align*}
for all $t\geq0$. Hence, we conclude (\ref{eq:xi_orthogonal-1}).
Using this for $\hat{\varphi}=\bar{\xi}_{N}(t)-\varphi_{N}^{0}\in\mathcal{E}$
and using (\ref{eq:xi_orthogonal-1})
yields 
\[
\Vert\varepsilon(\oxi(t)-\varphi_{N}^{0})\Vert_{\Lp^{2}(\OS)}^{2}+\Vert\varepsilon(\bar{\xi}_{N}(t)-\varphi_{N}^{0})\Vert_{\Lp^{2}(\OS)}^{2}=\Vert\varepsilon(\oxi-\bar{\xi}_{N})(t)\Vert_{\Lp^{2}(\OS)}^{2}\overset{t\to\infty}{\to}0.
\]
Finally, $P_{\mathcal{R}}(\oxi(t)-r(t)-\varphi_{N}^{0})=0$ and Lemma
\ref{lem:xi_korn-1} imply that 
\[
\lim_{t\to\infty}\Vert\oxi(t)-\varphi_{N}^{*,0}-r(t)\Vert_{\H1^{1}(\OS)}\leq\lim_{t\to\infty}C\Vert\varepsilon(\xi(t)-\varphi_{N}^{0})\Vert_{\Lp^{2}(\OS)}=0
\]
for 
\begin{align*}
\varphi_{N}^{*,0}:=\frac{\int_{\OS}\Sigma(\xi_{0}-\eta^{*}(0)):\varepsilon(\varphi_{N})\,\dx y}{\int_{\OS}\Sigma(\varphi_{N}):\varepsilon(\varphi_{N})\,\dx y}\varphi_{N}.
\end{align*}
Due to $\eta^{*}\vert_{\dOS}=0$ and $\Div(\Sigma(\varphi_{N}))=0$,
also 
\begin{align*}
\int_{\OS}\Sigma(\eta^{*}(0)):\varepsilon(\varphi_{N})\,\dx y=-\int_{\OS}\eta^{*}(0)\cdot\Div(\Sigma(\varphi_{N}))\,\dx y+\int_{\dOS}\eta^{*}(0)\cdot\Sigma(\varphi_{N})n\,\dx S(y)=0
\end{align*}
holds, and hence $\varphi_{N}^{0}=\varphi_{N}^{*,0}$. This implies
(\ref{eq:xi_conv_baddomain-1}) and Theorem \ref{thm:main result xi}
is proved. 

\section{Appendix}

\label{SecApp}

\subsection{Definition of spaces and auxiliary estimates}

Given a Banach space $X$, $T>0$ and $0<s<1$, for $f\in\Lp^{2}(0,T;X)$
we define 
\[
\left[f\right]_{s,(0,T),X}:=\left(\int_{0}^{T}\int_{0}^{T}\frac{\Vert f(t_{1},\cdot)-f(t_{2},\cdot)\Vert_{X}^{2}}{\vert t_{1}-t_{2}\vert^{2s+1}}\,\textrm{d}t_{1}\textrm{d}t_{2}\right)^{1/2}.
\]
We denote by $\H1^{s}(0,T;X)$ the Sobolev-Slobodeckii spaces with
norms 
\begin{align*}
\Vert f\Vert_{\H1^{s}(0,T;X)} & :=\begin{cases}
\left(\Vert f\Vert_{\Lp^{2}(0,T;X)}^{2}+\left[f\right]_{s,(0,T),X}^{2}\right)^{1/2} & \text{if }0<s<1,\\
\left(\Vert f\Vert_{\H1^{1}(0,T;X)}^{2}+\left[\dot{f}\right]_{s,(0,T),X}^{2}\right)^{1/2} & \text{if }1<s<2.
\end{cases}\\
\end{align*}

\begin{lem}[{{{\cite[Corollary A.3]{BGT2019}}}}]
Let $\frac{1}{2}<\sigma\leq1$ and $0<s<\sigma$. Then there exists
a constant $C>0$ independent of $T$ such that 
\[
\Vert f\Vert_{\H1^{s}(0,T;X)}\leq CT^{\sigma-s}\Vert f\Vert_{\H1^{\sigma}(0,T;X)}
\]
holds for all $f\in\H1^{\sigma}(0,T;X)$ with $f(0,\cdot)=0$.
\end{lem}

For general $f\in\H1^{\sigma}(0,T;X)$, the preceeding Lemma implies
that 
\begin{align}
\begin{split}\Vert f\Vert_{\H1^{s}(0,T;X)} & \leq\Vert f-f(0,\cdot)\Vert_{\H1^{s}(0,T;X)}+\Vert f(0,\cdot)\Vert_{\H1^{s}(0,T;X)}\\
 & \leq CT^{\sigma-s}\Vert f-f(0,\cdot)\Vert_{\H1^{\sigma}(0,T;X)}+T^{1/2}\Vert f(0,\cdot)\Vert_{X}\\
 & \leq CT^{\sigma-s}\Vert f\Vert_{\H1^{\sigma}(0,T;X)}+\left(CT^{1/2+\sigma-s}+T^{1/2}\right)\Vert f(0,\cdot)\Vert_{X}.
\end{split}
\label{eq:est_appendix_BGT}
\end{align}

\begin{lem}[{{{\cite[Lemma A.5]{BGT2019}}}}]
\label{lem:est_interpolation_BGT}

a) Let $0\leq s\leq1$, $\sigma_{1},\,\sigma_{2}\geq0$ and set $\sigma:=s\sigma_{1}+(1-s)\sigma_{2}$.
Then 
\begin{align*}
\H1^{1}(\H1^{\sigma_{1}}(\OF))\cap\Lp^{2}(\H1^{\sigma_{2}}(\OF))\hookrightarrow\H1^{s}(\H1^{\sigma}(\OF))
\end{align*}
and there exists a constant $C>0$ independent of $T$ such that 
\begin{align*}
\Vert v\Vert_{\H1^{s}(\H1^{\sigma}(\OF))}\leq C\Vert v\Vert_{\H1^{1}(\H1^{\sigma_{1}}(\OF))}^{s}\Vert v\Vert_{\Lp^{2}(\H1^{\sigma_{2}}(\OF))}^{1-s}
\end{align*}
for all $v\in\H1^{1}(\H1^{\sigma_{1}}(\OF))\cap\Lp^{2}(\H1^{\sigma_{2}}(\OF))$.
\\
 b) Let $1\leq s\leq2$, $\sigma_{1},\,\sigma_{2}\geq0$ and set $\sigma:=(s-1)\sigma_{1}+(2-s)\sigma_{2}$.
Then 
\begin{align*}
\H1^{2}(\H1^{\sigma_{1}}(\OF))\cap\H1^{1}(\H1^{\sigma_{2}}(\OF))\hookrightarrow\H1^{s}(\H1^{\sigma}(\OF))
\end{align*}
and there exists a constant $C>0$ independent of $T$ such that 
\begin{align*}
\Vert v\Vert_{\H1^{s}(\H1^{\sigma}(\OF))}\leq C\Vert v\Vert_{\H1^{2}(\H1^{\sigma_{1}}(\OF))}^{s}\Vert v\Vert_{\H1^{1}(\H1^{\sigma_{2}}(\OF))}^{1-s}
\end{align*}
for all $v\in\H1^{2}(\H1^{\sigma_{1}}(\OF))\cap\H1^{1}(\H1^{\sigma_{2}}(\OF))$. 
\end{lem}

We recall some Sobolev embeddings on the interval $(0,T)$ to clarify
the dependence of the appearing constants on the interval length $T>0$. 
\begin{lem}
\label{lem:est_sobolev_T-1} a) Let $s\in(0,1/2)$ and set $q:=\frac{2}{1-2s}$.
Then $\H1^{s}(0,T)\hookrightarrow\Lp^{q}(0,T)$ and there exists a
constant $C>0$ independent of $T$ such that 
\begin{align*}
\Vert f\Vert_{\Lp^{q}(0,T)}\leq C\left(T^{-s}\Vert f\Vert_{\Lp^{2}(0,T)}+\Vert f\Vert_{\H1^{s}(0,T)}\right)
\end{align*}
holds for all $f\in\H1^{s}(0,T)$. \\
 b) Let $s\in(1/2,1)$. Then $\H1^{s}(0,T)\hookrightarrow\C^{0}(0,T)$
and there exists a constant $C>0$ independent of $T$ such that 
\begin{align*}
\Vert f\Vert_{\C^{0}(0,T)}\leq C\left(T^{-1/2}\Vert f\Vert_{\Lp^{2}(0,T)}+T^{s-1/2}\Vert f\Vert_{\H1^{s}(0,T)}\right)
\end{align*}
holds for all $f\in\H1^{s}(0,T)$. 
\end{lem}

\begin{proof}
After rescaling a given function $f\in\H1^{s}(0,T)$ to 
\begin{align*}
\tilde{f}(\tau):=T^{1/2}f(T\tau),\hspace{0.3cm}\tau\in(0,1),
\end{align*}
the estimates in a) and b) can be derived from the corresponding embeddings
on the interval $(0,1)$, cf.~\cite[Theorem~5.4]{Hitchhiker}, \cite[Theorem~6.7]{Hitchhiker}
and \cite[Theorem~8.2]{Hitchhiker}. 
\end{proof}
We often use the following estimate obtained by combining (\ref{eq:est_appendix_BGT})
and Lemma~\ref{lem:est_sobolev_T-1}~a): \\
 Let $s\in(0,1/2)$, $\sigma\in(1/2,1)$ and $f\in\H1^{\sigma}(0,T;X)$
for some Banach space $X$. Then for $q:=\frac{2}{1-2s}$, we can
choose $\alpha>0$ such that 
\begin{align}
\left\Vert \Vert f\Vert_{X}\right\Vert _{\Lp^{q}(0,T)}\leq CT^{\alpha}\left(\Vert f\Vert_{\H1^{\sigma}(0,T;X)}+\Vert f(0)\Vert_{X}\right).\label{eq:sobolev_inequality_T-1}
\end{align}

\subsection{The Stokes and the Lamé system}

We recall optimal regularity results for solutions to the Stokes and
the Lamé system.
\begin{thm}[{{{\cite[Theorem 4.1]{BP2007}}}}]
\label{thm:Stokes_BP} Let $D\subset\mathbb{R}^{3}$ be a domain
with a compact boundary $\partial D=\Gamma_{0}\cup\Gamma_{n}$ of
class $\C^{2,1}$, where $\Gamma_{0},\Gamma_{n}$ are open and closed
in $\partial D$. Consider given 
\begin{itemize}
\item $f\in\Lp^{2}((0,T)\times D)$, 
\item $v_{0}\in\H1^{1}(D)$ such that $\Div(v_{0})=0$ and $v_{0}\cdot n\vert_{\Gamma_{0}}=0$ 
\item and $g\in\H1^{1/4}(\Lp^{2}(\Gamma_{n}))\cap\Lp^{2}(\H1^{1/2}(\Gamma_{n}))$. 
\end{itemize}
Then the Stokes system 
\begin{align*}
\begin{cases}
\begin{array}{rcll}
\dot{v}-\Div(\sigma(v,q)) & = & f & \text{in }(0,T)\times D\\
\Div(v) & = & 0 & \text{in }(0,T)\times D\\
v & = & 0 & \text{on }(0,T)\times\Gamma_{0}\\
\sigma(v,q)n & = & g & \text{on }(0,T)\times\Gamma_{n}\\
v(0) & = & v_{0} & \text{in }D,
\end{array}\end{cases}
\end{align*}
admits a unique solution 
\begin{align*}
(v,q)\in\left(\Lp^{2}(\H1^{2}(D))\cap\H1^{1}(\Lp^{2}(D))\right)\times\Lp^{2}(\H1^{1}(D))
\end{align*}
such that $q\in\H1^{1/4}(\Lp^{2}(\Gamma_{n}))\cap\Lp^{2}(\H1^{1/2}(\Gamma_{n}))$,
where the corresponding norms of $(v,q)$ depend continuously on the
data. 
\end{thm}

\begin{thm}[{{{\cite[Theorem 7.5]{GS1991}}}}]
\label{thm:elliptic_stokes} Let $D\subset\mathbb{R}^{3}$ be a domain
with boundary $\partial D=\Gamma_{0}\cup\Gamma_{n}$ of class $\C^{2,1}$
and let $0\leq s\leq1$. Consider given $f\in\H1^{s}(D)$ and $g\in\H1^{1/2+s}(\Gamma_{n})$.
If the pair $(v,q)\in\H1^{2}(D)\times\H1^{1}(D)$ solves 
\[
\begin{cases}
\begin{array}{rcll}
-\Div(\sigma(v,q)) & = & f & \text{in }D,\\
\Div(v) & = & 0 & \text{in }D,\\
v & = & 0 & \text{on }\Gamma_{0},\\
\sigma(v,q)n & = & g & \text{on }\Gamma_{n},
\end{array}\end{cases}
\]
then it satisfies the estimate 
\[
\Vert v\Vert_{\H1^{2+s}(D)}+\Vert q\Vert_{\H1^{1+s}(D)}\leq C\left(\Vert f\Vert_{\H1^{s}(D)}+\Vert g\Vert_{\H1^{1/2+s}(\Gamma_{n})}\right).
\]
\end{thm}

The following two results are proven in \cite{LLT1986} for the wave
equation and can be transferred to the Lamé system. 
 Relation \eqref{eq:hiddenReg} is often called \emph{hidden regularity}. 
\begin{thm}[{{{\cite[Theorem 2.4]{LLT1986}}}}]
\label{thm:lame_LLT} Let $D\subset\mathbb{R}^{3}$ be an open domain
with boundary $\partial D$ of class $\C^{2}$. For $0\leq\theta<3/2$,
$\theta\neq1/2$, consider given 
\begin{itemize}
\item $\eta_{0}\in\H1^{1+\theta}(D)$, 
\item $\eta_{1}\in\H1^{\theta}(D)$, 
\item and $g\in\mathrm{M}^{1+\theta}(\partial D)=\H1^{1+\theta}(0,T;\Lp^{2}(\partial D))\cap\Lp^{2}(0,T;\H1^{1+\theta}(\partial D))$
such that $g(0)=\eta_{0}\vert_{\partial D}$ and, if $\theta>1/2$,
also $\dot{g}(0)=\eta_{1}\vert_{\partial D}$. 
\end{itemize}
Then the Lamé system 
\begin{align}
\begin{cases}
\begin{array}{rcll}
\ddot{\eta}-\Div(\Sigma(\eta)) & = & 0 & \text{in }(0,T)\times D,\\
\eta & = & g & \text{on }(0,T)\times\partial D,\\
\eta(0) & = & \eta_{0},\,\dot{\eta}(0)=\eta_{1} & \text{in }D,
\end{array}\end{cases}\label{eq:lame_system}
\end{align}
admits a unique solution 
\begin{align*}
\eta\in\C^{0}(\H1^{1+\theta}(D))\cap\C^{1}(\H1^{\theta}(D))\cap\C^{2}(\H1^{-1+\theta}(D)).
\end{align*}
Moreover, 
\begin{align}
\Sigma(\eta)n\vert_{\partial D}\in\mathrm{M}^{\theta}(\partial D)\label{eq:hiddenReg}
\end{align}
and the corresponding norms for $\eta$, $\Sigma(\eta)n\vert_{\partial D}$
depend continuously on the the data. 
\end{thm}

\begin{proof}
The proof of {\cite[Theorem 2.4]{LLT1986}} for the wave equation
can be adapted for the Lamé system by adding a further dependence
of the corresponding right-hand sides on the Dirichlet data $\Vert g\Vert_{\H1^{1+\theta}((0,T)\times\partial D)}$.
The critical part is to check the desired regularity of the Neumann
boundary $\Sigma(\eta)n\vert_{\partial D}$. As already mentioned
in \cite[Remark~3.4]{KTZ2011}, this can be done by proceeding as
in {\cite[Theorem~2.4]{LLT1986}} and using the estimate 
\[
\Vert\Sigma(\eta)n\Vert_{\mathrm{M}^{\theta}(\partial D)}\leq C\left(\left\Vert \frac{\partial\eta}{\partial n}\right\Vert _{\mathrm{M}^{\theta}(\partial D)}+\Vert\nabla_{\tau}\eta\Vert_{\mathrm{M}^{\theta}(\partial D)}\right)\leq C\left(\left\Vert \frac{\partial\eta}{\partial n}\right\Vert _{\mathrm{M}^{\theta}(\partial D)}+\Vert\eta\Vert_{\mathrm{M}^{1+\theta}(\partial D)}\right)
\]
proven in \cite{Horn}, where $\nabla_{\tau}\eta$ denotes the tangential
part of $\nabla\eta$ on the boundary. This yields another extra dependence
on $\Vert g\Vert_{\mathrm{M}^{1+\theta}(\partial D)}$, which however
does not influence the claimed result. 
\end{proof}
Even if the regularity required of the given Dirichlet boundary is
weakened, it is possible to keep the interior regularity of the solution:
\begin{thm}[{{{\cite[Theorem 3.4]{LLT1986}}}}]
\label{thm:lame_LLT_interior} Let $D\subset\mathbb{R}^{3}$ be an
open domain with boundary $\partial D$ of class $\C^{2}$. Consider
given 
\begin{itemize}
\item $\eta_{0}\in\H1^{1}(D)$, 
\item $\eta_{1}\in\Lp^{2}(D)$, 
\item and $g\in\C^{0}(\H1^{1/2}(\partial D))\cap\H1^{1}(\Lp^{2}(\partial D))$
such that $g(0)=\eta_{0}\vert_{\partial D}$. 
\end{itemize}
Then (\ref{eq:lame_system}) admits a unique solution 
\begin{align*}
\eta\in\C^{0}(\H1^{1}(D))\cap\C^{1}(\Lp^{2}(D))
\end{align*}
that depends continuously on the the given data. 
\end{thm}

The following result concerns the ellitpic Lamé system. 
\begin{thm}
\cite[Ex I.6.2c and Thm I.6.3.1]{agranovich}\label{thm:elliptic_lame}
Let $D\subset\mathbb{R}^{3}$ be a domain with boundary $\partial D$
of class $\C^{2,1}$ and let $0\leq s\leq2$. Consider given $f\in\H1^{s}(D)$
and $g\in\H1^{1/2+s}(\partial D)$. If $\eta\in\H1^{1}(D)$ solves
\[
\begin{cases}
\begin{array}{rcll}
-\Div(\Sigma(\eta)) & = & f & \text{in }D,\\
\eta & = & g & \text{on }\partial D,
\end{array}\end{cases}
\]
then it satisfies the estimate 
\[
\Vert\eta\Vert_{\H1^{1+s}(D)}\leq C\left(\Vert f\Vert_{\H1^{-1+s}(D)}+\Vert g\Vert_{\H1^{1/2+s}(\partial D)}\right).
\]
\end{thm}

\bibliographystyle{alpha}
\bibliography{literatur}

\newcommand{\etalchar}[1]{$^{#1}$}
\begin{thebibliography}{MMN{\etalchar{+}}22}

\bibitem[AB15]{AB2015}
G.~Avalos and F.~Bucci.
\newblock Rational rates of uniform decay for strong solutions to a
  fluid-structure {PDE} system.
\newblock {\em J. Differential Equations}, 258(12):4398--4423, 2015.

\bibitem[AES97]{agranovich}
M.~S. Agranovich, J.~V. Egorov, and M.~A. Shubin.
\newblock {\em Partial differential equations {IX}, {E}lliptic Boundary Value
  Problems}, volume~79 of {\em Encyclopaedia of Mathematical Sciences}.
\newblock Springer, Berlin Heidelberg, 1997.

\bibitem[AKT25]{AKT2025}
Thomas Alazard, Igor Kukavica, and Amjad Tuffaha.
\newblock Global-in-time weak solutions for an inviscid free surface
  fluid-structure problem without damping.
\newblock {\em Ann. PDE}, 11(1):40, 2025.
\newblock Id/No 15.

\bibitem[AT07]{5AT}
G.~Avalos and R.~Triggiani.
\newblock The coupled {PDE} system arising in fluid/structure interaction, part
  {I}: Explicit semigroup generator and its spectral properties.
\newblock {\em AMS Contemp. Math., Fluids and Waves}, 440:15--54, 2007.

\bibitem[AT08]{ATuniformstab}
G.~Avalos and R.~Triggiani.
\newblock Uniform stabilization of a coupled {PDE} system arising in
  fluid-structure interaction with boundary dissipation at the interface.
\newblock {\em Discrete Cont. Dyn. Sys.}, 22(4):817--835, 2008.

\bibitem[AT09a]{ATboundary}
G.~Avalos and R.~Triggiani.
\newblock Boundary feedback stabilization of a coupled parabolic-hyperbolic
  {S}tokes-{L}am{\'e} {PDE} system.
\newblock {\em J. Evol. Equ.}, 9:341--370, 2009.

\bibitem[AT09b]{8AT2009}
G.~Avalos and R.~Triggiani.
\newblock Coupled parabolic-hyperbolic {S}tokes-{L}am{\'e} {PDE} system: limit
  behaviour of the resolvent operator on the imaginary axis.
\newblock {\em Appl. Anal.}, 88(9):1357--1396, 2009.

\bibitem[AT09c]{4AT2009}
G.~Avalos and R.~Triggiani.
\newblock Semigroup well-posedness in the energy space of a
  parabolic-hyperbolic coupled {S}tokes-{L}am{\'e} {PDE} system of
  fluid-structure interaction.
\newblock {\em Discrete Contin. Dyn. Syst. - S}, 2(3):417--447, 2009.

\bibitem[AT13]{AT2}
G.~Avalos and R.~Triggiani.
\newblock Fluid-structure interaction with and without internal dissipation of
  the structure: A contrast study in stability.
\newblock {\em Evol. Equ. Control Theory.}, 2(4):563--598, 2013.

\bibitem[Ava07]{Astrongstab}
G.~Avalos.
\newblock The strong stability and instability of a fluid-structure semigroup.
\newblock {\em Appl. Math. Optim.}, 55:163--184, 2007.

\bibitem[BGT19]{BGT2019}
M.~Boulakia, S.~Guerrero, and T.~Takahashi.
\newblock Well-posedness for the coupling between a viscous incompressible
  fluid and an elastic structure.
\newblock {\em Nonlinearity}, 32:3548--3592, 2019.

\bibitem[BKS24a]{BKS2024}
Barbora Bene{\v{s}}ov{\'a}, Malte Kampschulte, and Sebastian Schwarzacher.
\newblock A variational approach to hyperbolic evolutions and fluid-structure
  interactions.
\newblock {\em J. Eur. Math. Soc. (JEMS)}, 26(12):4615--4697, 2024.

\bibitem[BKS24b]{BKS2024compressible}
Dominic Breit, Malte Kampschulte, and Sebastian Schwarzacher.
\newblock Compressible fluids interacting with 3d visco-elastic bulk solids.
\newblock {\em Math. Ann.}, 390(4):5495--5552, 2024.

\bibitem[Bou07]{Boulakia2007}
M.~Boulakia.
\newblock Existence of weak solutions for the three-dimensional motion of an
  elastic structure in an incompressible fluid.
\newblock {\em J. Math. Fluid Mech.}, 9:262--294, 2007.

\bibitem[BP07]{BP2007}
D.~Bothe and J.~Pr{\"u}ss.
\newblock {$\Lp^p$}-theory for a class of non-{N}ewtonian fluids.
\newblock {\em SIAM J. Math. Anal.}, 39(2):379--421, 2007.

\bibitem[CGH21]{CGH2021}
J.-J. Casanova, C.~Grandmont, and M.~Hillairet.
\newblock On an existence theory for a fluid-beam problem encompassing possible
  contacts.
\newblock {\em J. \'{E}c. polytech. Math.}, 8:933--971, 2021.

\bibitem[CS05]{CS2005}
D.~Coutand and S.~Shkoller.
\newblock Motion of an elastic solid inside an incompressible viscous fluid.
\newblock {\em Arch. Ration. Mech. Anal.}, 176:25--102, 2005.

\bibitem[CS06]{CS2006}
D.~Coutand and S.~Shkoller.
\newblock The interaction between quasilinear elastodynamics and the
  {N}avier-{S}tokes equations.
\newblock {\em Arch. Ration. Mech. Anal.}, 179:303--352, 2006.

\bibitem[DL22]{DL2022}
K.~Disser and M.~Luckas.
\newblock Existence of global solutions for 2{D} fluid-elastic interaction with
  small data.
\newblock In M.~I. Espanol, M.~Lewicka, L.~Scardia, and A.~Schlömerkemper,
  editors, {\em Research in Mathematics of Materials Science}. Springer, Cham,
  2022.
\newblock to appear.

\bibitem[GH16]{GH2016}
C.~Grandmont and M.~Hillairet.
\newblock Existence of global strong solutions to a beam-fluid interaction
  system.
\newblock {\em Arch. Ration. Mech. Anal.}, 220:1283--1333, 2016.

\bibitem[GHH06]{GHH06}
M.~Gei{\ss}ert, H.~Heck, and M.~Hieber.
\newblock On the equation $\text{div} u = g$ and {B}ogovskii's operator in
  {S}obolev spaces of negative order.
\newblock In {\em Partial Differential Equations and Functional Analysis, Oper.
  Theory Adv. Appl.}, volume 168, pages 113--121, Basel, 2006. Birkhäuser.

\bibitem[GHL19]{GHL2019}
C.~Grandmont, M.~Hillairet, and J.~Lequeurre.
\newblock Existence of local strong solutions to fluid-beam and fluid-rod
  interaction systems.
\newblock {\em Ann. Inst. Henri Poincaré (C) Anal. Non Linéaire},
  36:1105--1149, 2019.

\bibitem[GMZZ14]{GMZZ2014}
Giovanni Galdi, Mahdi Mohebbi, Rana Zakerzadeh, and Paolo Zunino.
\newblock Hyperbolic-parabolic coupling and the occurrence of resonance in
  partially dissipative systems.
\newblock In {\em Fluid-structure interaction and biomedical applications},
  pages 197--256. Basel: Birkh{\"a}user/Springer, 2014.

\bibitem[GS91]{GS1991}
G.~Grubb and V.~A. Solonnikov.
\newblock Boundary value problems for the nonstationary {N}avier-{S}tokes
  equations treated by pseudo-differential methods.
\newblock {\em Math. Scand.}, 69:217--290, 1991.

\bibitem[Hor98]{Horn}
M.~A. Horn.
\newblock Sharp trace regularity for the solutions of the equations of dynamic
  elasticity.
\newblock {\em J. Math. Syst. Estim. Control}, 8(2):1--11, 1998.

\bibitem[IKLT17]{IKLT2017}
M.~Ignatova, I.~Kukavica, I.~Lasiecka, and A.~Tuffaha.
\newblock Small data global existence for a fluid-structure model.
\newblock {\em Nonlinearity}, 30(2):848--898, 2017.

\bibitem[KO23]{KO2023}
Igor Kukavica and Wojciech~S. O{\.z}a{\'n}ski.
\newblock Global well-posedness and exponential decay for a fluid-structure
  model with small data.
\newblock {\em Indiana Univ. Math. J.}, 72(6):2701--2731, 2023.

\bibitem[KO24]{KO2024}
Igor Kukavica and Wojciech~S. O{\.z}a{\'n}ski.
\newblock On a model of an elastic body fully immersed in a viscous
  incompressible fluid with small data.
\newblock {\em SIAM J. Math. Anal.}, 56(1):746--761, 2024.

\bibitem[KT12]{KT2012}
I.~Kukavica and A.~Tuffaha.
\newblock Solutions to a fluid-structure interaction free boundary problem.
\newblock {\em Discrete Cont. Dyn. Sys.}, 32(4):1355--1389, 2012.

\bibitem[KT24]{KT2024}
Igor Kukavica and Amjad Tuffaha.
\newblock A free boundary inviscid model of flow-structure interaction.
\newblock {\em J. Differ. Equations}, 413:851--912, 2024.

\bibitem[KTZ11]{KTZ2011}
I.~Kukavica, A.~Tuffaha, and M.~Ziane.
\newblock Strong solutions to a {N}avier-{S}tokes-{L}am{\'e} system on a domain
  with a nonflat boundary.
\newblock {\em Nonlinearity}, 24:159--176, 2011.

\bibitem[LLT86]{LLT1986}
I.~Lasiecka, J.-L. Lions, and R.~Triggiani.
\newblock Non homogeneous boundary value problems for second order hyperbolic
  operators.
\newblock {\em J. Math. pures et appl.}, 65:149--192, 1986.

\bibitem[LM72]{LionsMagenes}
J.-L. Lions and E.~Magenes.
\newblock {\em Non-Homogeneous Boundary Value Problems and Applications {I}},
  volume 181 of {\em Die Grundlehren der mathematischen Wissenschaften}.
\newblock Springer, Berlin Heidelberg, 1972.

\bibitem[LR14]{LR2014}
D.~Lengeler and Michael Ru\v{z}i\v{c}ka.
\newblock Weak solutions for an incompressible {N}ewtonian fluid interacting
  with a {K}oiter type shell.
\newblock {\em Arch. Ration. Mech. Anal.}, 211(1):205--255, 2014.

\bibitem[MAA18]{MAA2018}
D.~Mokbel, H.~Abels, and S.~Aland.
\newblock A phase-field model for fluid-structure interaction.
\newblock {\em J. Comput. Phys.}, 372:823--840, 2018.

\bibitem[MC15]{MC2015}
B.~Muha and S.~Cani\'{c}.
\newblock Fluid-structure interaction between an incompressible, viscous 3{D}
  fluid and an elastic shell with nonlinear {K}oiter membrane energy.
\newblock {\em Interfaces Free Boundaries}, 17(4):465--495, 2015.

\bibitem[MMN{\etalchar{+}}22]{MMNRT2022}
V.~M\'{a}cha, B.~Muha, \v{S}. Ne\v{c}asov\'{a}, A.~Roy, and S.~Trifunovi\'{c}.
\newblock Existence of a weak solution to a nonlinear fluid-structure
  interaction problem with heat exchange.
\newblock {\em Comm. Partial Differential Equations}, 47(8):1591--1635, 2022.

\bibitem[MRR20]{MRR2020}
D.~Maity, J.-P. Raymond, and A.~Roy.
\newblock Maximal-in-time existence and uniqueness of strong solution of a 3{D}
  fluid-structure interaction model.
\newblock {\em SIAM J. Math. Anal.}, 52(6):6338--6378, 2020.

\bibitem[NPV12]{Hitchhiker}
E.~Di Nezza, G.~Palatucci, and E.~Valdinoci.
\newblock Hitchhiker's guide to the fractional {S}obolev spaces.
\newblock {\em Bull. Sci. Math.}, 136(5):512--573, 2012.

\bibitem[NPW15]{Korn}
P.~Neff, D.~Pauly, and K.-J. Witsch.
\newblock Poincar{\'e} meets {K}orn via {M}axwell: {E}xtending {K}orn's first
  inequality to incompatible tensor fields.
\newblock {\em J. Differential Equations}, 258(4):1267--1302, 2015.

\bibitem[OP99]{OP1999}
A.~Osses and J.-P. Puel.
\newblock Approximate controllability for a linear model of fluid structure
  interaction.
\newblock {\em ESAIM Control Optim. Calc. Var.}, 4:497--513, 1999.

\bibitem[RV14]{RV2014}
J.-P. Raymond and M.~Vanninathan.
\newblock A fluid-structure model coupling the {N}avier-{S}tokes equations and
  the {L}am{\'e} system.
\newblock {\em J. Math. Pures Appl.}, 102(3):546--596, 2014.

\bibitem[Tri78]{Triebel}
H.~Triebel.
\newblock {\em Interpolation Theory, Function Spaces, Differential Operators}.
\newblock North-Holland Publishing Company, Amsterdam, 1978.

\bibitem[vGL{\etalchar{+}}19]{CGLMTW2019}
S.~\v{C}ani\'{c}, M.~Gali\'{c}, M.~Ljulj, B.~Muha, J.~Tamba\v{c}a, and Y.~Wang.
\newblock Analysis of a linear 3{D} fluid-mesh-shell interaction problem.
\newblock {\em Z. Angew. Math. Phys.}, 70(2):Paper No. 44, 38, 2019.

\end{thebibliography}

\end{document}